\documentclass [12pt,oneside]{amsart}
\usepackage{mathrsfs}
\usepackage{xcolor}
\usepackage{amssymb,amsmath,amsthm,amsfonts,amscd}
\usepackage{enumerate}
\usepackage{setspace}
\usepackage{empheq}
\usepackage[all]{xy}
\usepackage{verbatim}
\usepackage[normalem]{ulem}
\usepackage{todonotes}
\usepackage[bookmarks,colorlinks,breaklinks]{hyperref}  
\hypersetup{linkcolor=blue,citecolor=blue,filecolor=blue,urlcolor=blue} 
\usepackage{time}

\numberwithin{equation}{section}
\numberwithin{figure}{section}

\newtheorem{theorem}{Theorem}[section]
\newtheorem{proposition}[theorem]{Proposition}
\newtheorem{prop}[theorem]{Proposition}

\newtheorem{lemma}[theorem]{Lemma}

\newtheorem{corollary}[theorem]{Corollary}

\theoremstyle{definition}

\newtheorem{remark}[theorem]{Remark}

\definecolor{myblue}{rgb}{0.6, 0.9, 1}

\makeatletter

\newcommand{\Rmnum}[1]{\expandafter\@slowromancap\romannumeral #1@}
\makeatother

\definecolor{myblue}{rgb}{0.6, 0.9, 1}
\definecolor{mygreen}{rgb}{0,0,1}
\definecolor{purple}{rgb}{0.6,0.2,1}
\definecolor{orange}{rgb}{0.8,0,0.2}

\newcommand{\bC}{\mathbb{C}}
\newcommand{\bP}{\mathbb{P}}
\newcommand{\C}{\mathbb{C}}
\newcommand{\Q}{\mathbb{Q}}

\newcommand{\bZ}{\mathbb{Z}}
\newcommand{\bQ}{\mathbb{Q}}

\newcommand{\R}{\mathbb{R}}
\newcommand{\bN}{\mathbb{N}}

\newcommand{\bA}{\mathbb{A}}

\newcommand{\Kbar}{\overline{K}}
\newcommand{\Qbar}{\overline{\bQ}}

\newcommand{\supp}{\operatorname{supp}}

\newcommand{\cH}{\mathcal{H}}
\newcommand{\Pic}{\operatorname{Pic}}
\newcommand{\Pichat}{\widehat{\Pic}}
\newcommand{\<}{\langle}
\renewcommand{\>}{\rangle}

\newcommand{\rank}{\operatorname{rank}}
\newcommand\volhat{\widehat{\mathrm{vol}}}
\newcommand\iso{\simeq}

\newcommand{\Z}{\mathbb{Z}}

\newcommand{\cal}{\mathcal}
\newcommand{\an}{\mathrm{an}}

\newcommand{\Spec}{\operatorname{Spec}}
\newcommand{\nef}{\mathrm{nef}}
\newcommand\Rat  {\mathrm{Rat}}

\newcommand\Lbar{\overline{L}}

\newcommand\Mbar {\overline{M}}  
\newcommand\Nbar {\overline{N}}

\newcommand\Preper {\mathrm{Preper}}

\newcommand\Crit {\mathrm{Crit}}

\newcommand\PGL{\mathrm{PGL}}

\newcommand\cX{\mathcal{X}}

\newcommand\actdim{\operatorname{act-dim}}
\newcommand\cE{\mathcal E}
\newcommand\cA{\mathcal A}
\newcommand\cB{\mathcal B}
\newcommand\cG{\mathcal G}
\newcommand\cQ{\mathcal Q}
\newcommand\cV{\mathcal V}
\newcommand\cY{\mathcal Y}

\newcommand\Betti{\mathrm{Betti}}

\begin{document}
	\title[]{Dynamical relations on $\bP^1$, higher bifurcations, and arithmetic bigness}

\author{Laura DeMarco}
\author{Niki Myrto Mavraki}


	\begin{abstract}
In this article, we study algebraic relations among rational maps on $\bP^1$ and among their preperiodic points.  Working over $\bC$, we characterize algebraic families of positive-dimensional varieties in $(\bP^1)^N$ and tuples of maps $(f_1, \ldots, f_N)$ for which the subvariety is dynamically special for a Zariski-dense set of parameters.  Moreover, outside of well-understood Latt\`es phenomena, we show that all families of maps and subvarieties that bifurcate are maximally bifurcating; that is, we show that a positive bifurcation (1,1)-current implies positivity of the associated bifurcation measure on the parameter space. Working over $\Qbar$, we characterize arithmetic and geometric potential bigness for the canonical line bundle associated to families of endomorphisms of $(\bP^1)^N$ when restricted to the subvarieties.  

As consequences, we obtain uniform bounds for the periods and preperiods of relations among non-exceptional maps, and we show uniform bounds on the geometry of preperiodic points satisfying these relations. Over $\Qbar$, our results yield corresponding uniform bounds for points and subvarieties satisfying a height gap. 
    \end{abstract}
	
	\maketitle


\tableofcontents

\section{Introduction}

In this article, we study the prevalence of algebraic relations among maps on $\bP^1$, and we provide uniform bounds on their geometry as well as the geometry of their preperiodic points; see Theorems \ref{preperiodbound} and \ref{degree bound}. 
These bounds are a consequence of our first main result, Theorem \ref{main1}, which characterizes algebraic families of product maps and subvarieties of $(\bP^1)^N$ with a Zariski-dense set of dynamically special fibers and shows that this condition is equivalent to the Zariski density of preperiodic points in sufficiently high fiber powers. 

Working over the field $\C$ of complex numbers, these questions are closely connected to the theory of bifurcations in holomorphic families, and we characterize the vanishing of the associated bifurcation measures, giving a criterion for geometric potential bigness of the canonical line bundles; see Theorems \ref{hypersurfaces} and \ref{true deg char}. When the families are defined over the field $\Qbar$ of algebraic numbers, we characterize arithmetic bigness of the canonical line bundles, distinguishing geometric degeneracy from an additional obstruction arising from simultaneous torsion specializations in algebraic groups; see Theorem \ref{main3}. 

As much as possible, we formulate our results in the language of complex dynamical systems, but we provide equivalent, and sometimes stronger, formulations in the language of heights and volumes of adelic line bundles.

The theorems below were inspired by recent developments both in complex dynamical systems, especially in the theory of bifurcation currents and measures \cite{Gauthier:Taflin, GauthierVigny2025, Gauthier:Taflin:Vigny, Astorg:Bianchi, Dujardin:higherbif}, and in arithmetic geometry, especially as applied to families of abelian varieties \cite{Gao:Habegger:RMM, Kuhne:RBC, Yuan:bigness, Yuan:Zhang:quasiprojective}, as well as the broader conjecture we formulated in \cite{DM:UDMM}.  

\subsection{Key definitions}
Fix an integer $N\geq 2$, and suppose that $f_1, \ldots, f_N: \bP^1\to\bP^1$ are rational maps defined over $\C$, of degrees $\deg f_i >1$.  A closed subvariety
	\[ X \subset (\bP^1)^N \]
is said to be {\bf $\Phi$-special} for $\Phi := (f_1, \ldots, f_N)$ if it is preperiodic for a product of iterates
	\[ (f_1^{n_1}, \ldots, f_N^{n_N}) : (\bP^1)^N \to (\bP^1)^N \]
or, more generally, for a product endomorphism $\Psi: (\bP^1)^N \to (\bP^1)^N$ which commutes with a product of iterates of $\Phi$ and has degree $>1$ in every coordinate.  By \cite{Ghioca:Nguyen:Ye:DMM2,Ghioca:Nguyen:Ye:DMM1,MSW2023, Hindry, Ghioca:Tucker:DMM}, these are precisely the subvarieties containing a Zariski-dense set of points in
	\[ \Preper(\Phi) = \Preper(f_1) \times \cdots \times \Preper(f_N). \]
We say that a proper $\Phi$-special subvariety $X \subset (\bP^1)^N$ defines a {\bf relation among the} $f_1, \ldots, f_N$ if $X$ projects surjectively to every factor $\bP^1$.  The existence of a relation forces compositional relations among the maps; see \S\ref{relations}.  For example, the diagonal line $\Delta \subset \bP^1\times\bP^1$ is periodic for $(f,g)$ if and only if $f^n = g^n$ for some $n\geq 1$.  It is preperiodic for $(f,g)$ if and only if the pair satisfies semiconjugacy relations of the form $f^m \circ g^m = f^{2m}$ and $g^m\circ f^m = g^{2m}$ for some $m\geq 1$. 

We say that $f: \bP^1\to\bP^1$ is {\bf exceptional} if it is the finite quotient of an isogeny of a one-dimensional algebraic group, either the multiplicative group $\mathbb G_m$ or an elliptic curve.  Those arising from elliptic curves are called {\bf Latt\`es maps}.  For non-exceptional maps, the special subvarieties are assembled from preperiodic curves in pairs of coordinates; see \cite{Medvedev:thesis, Medvedev:Scanlon} or \cite{Ghioca:Nguyen:Ye:DMM2, Ghioca:Nguyen:Ye:DMM1, MSW2023} for proofs.  Exceptional maps admit additional group-theoretic relations.  If each $f_i$ is exceptional, associated to a group $G_i$, then $X$ in $(\bP^1)^N$ is $\Phi$-special if and only if it is the quotient of a torsion-coset in the group $G_1\times\cdots \times G_N$.  In general, we will say that a subvariety $X \subset (\bP^1)^N$ is {\bf weakly $\Phi$-special} if it is either $\Phi$-special or if all maps $f_i$ are exceptional and $X$ is the quotient of a coset (torsion or not) in $G_1 \times \cdots \times G_N$.   

\subsection{Algebraic families over $\C$: the first main result}
Let $S$ be an irreducible quasiprojective variety over $\C$, and consider an algebraic family
	\[\Phi = (f_1, \ldots, f_N): S \times (\bP^1)^N \to S \times (\bP^1)^N, \]
given by
	\[\Phi(s,z_1,\ldots,z_N) = (s,(f_1)_s(z_1),\ldots,(f_N)_s(z_N)), \]
where $d_i=\deg(f_i)_s>1$ is independent of $s$.  Let $\cX\subset S\times(\bP^1)^N$ be (geometrically) irreducible and flat over $S$, and assume that its generic fiber projects dominantly to every coordinate $\bP^1$.  We let $X_s$ denote the fiber of $\cX$ over a point $s \in S(\C)$.  We say that $\cX$ is $\Phi$-special (or weakly $\Phi$-special) if its generic fiber is special (respectively, weakly special) for the generic fiber of $\Phi$.

Our first main theorem is the following.

\begin{theorem} \label{main1}
Let $\Phi = (f_1, \ldots, f_N)$ be an algebraic family of maps on $(\bP^1)^N$ and $\cX \subset S \times (\bP^1)^N$  irreducible, flat over $S$, and that its generic fiber projects dominantly to every coordinate.  The following are equivalent.
\begin{enumerate}
\item The set
	\[ \{s\in S(\C): X_s \mbox{ is $\Phi_s$-special}\} \]
is Zariski dense in $S$.
\item Either $\cX$ is $\Phi$-special, or at least two of the maps $f_i$ are exceptional and every essential hypersurface shadow of $\cX$ is weakly special, becoming special simultaneously on a Zariski-dense set of parameters.
\item The set $\Preper(\Phi^{(1+\dim S)})$ is Zariski dense in the fiber power $\cX^{(1+\dim S)}$.
\end{enumerate}
\end{theorem}

The essential hypersurface shadows in part (2) are introduced in \S\ref{decomposition}.  They are obtained by projecting an $r$-dimensional variety in $(\bP^1)^N$ to suitable collections of $r+1$ coordinates and then discarding variables on which the resulting hypersurface equation does not depend.  We will explain and characterize the simultaneity condition in (2) in \S\ref{simultaneous torsion}.  For a positive integer $m$, we write $\Phi^{(m)}$ and $\cX^{(m)}$ for the $m$-th fiber powers over $S$.  Thus a point of
	\[ \Preper(\Phi^{(m)})\cap\cX^{(m)} \]
is a collection of $m$ points in a single fiber $X_s$, all of whose coordinates are preperiodic for the corresponding maps $(f_i)_s$.  In the case of one-parameter families of curves in $(\bP^1)^2$, defined over $\Qbar$, Theorem \ref{main1} is established in \cite[Theorems 1.9 and 5.7]{Mavraki:Schmidt}.

\subsection{Some consequences of Theorem \ref{main1}}
We begin with uniform bounds on the complexity of relations.  Fix once and for all a projective embedding of $(\bP^1)^N$, so that degrees of subvarieties are defined.

\begin{theorem} \label{preperiodbound}
Fix integers $N, D \geq 2$ and $E \geq 1$.  There exists $M=M(N,D,E)$ with the following property.  Let $f_1,\ldots,f_N$ be non-exceptional maps of degrees at most $D$, and suppose that a subvariety $V\subset(\bP^1)^N$ of degree at most $E$ defines a relation among them.  Then there are integers
	\[ 1\leq a_i\leq M, \qquad 0\leq m<n\leq M, \]
such that, for
	\[ F_{\mathbf a}=(f_1^{a_1},\ldots,f_N^{a_N}), \]
we have
	\[ F_{\mathbf a}^n(V)=F_{\mathbf a}^m(V). \]
\end{theorem}

In particular, for fixed degrees of the maps $f_i$ and bounded degree
of $V$, the locus of relations is a finite union of algebraic families
in the parameter space of non-exceptional maps and proper irreducible
subvarieties of degree $\leq E$ that dominate every coordinate.
\begin{remark}
Uniform period bounds for invariant curves in $\bP^1\times\bP^1$ were known for non-exceptional polynomial pairs \cite{GhiocaNguyen2017, Pakovich2025}, and bounds depending on primes of good reduction are known for rational maps over number fields \cite{BellGhiocaTucker2015}.  For each $d\geq2$, Pakovich constructs a nonempty Zariski-open subset $U\subset\Rat_d$ such that every irreducible periodic curve in $\bP^1\times\bP^1$ with dominant coordinate projections for a pair $(f,g)\in U\times U$ is invariant \cite[Theorem~1.2]{PakovichGeneral2025}.  Theorem \ref{preperiodbound} treats arbitrary products of non-exceptional rational maps over $\C$, and it bounds both period and tail length, though in this form it must also depend on the degree of $V$.
\end{remark}

For the diagonal line in $\bP^1\times\bP^1$, we obtain the following.

\begin{corollary} \label{iterateshare}
Fix $D>1$.  There exists $M=M(D)$ such that, whenever non-exceptional rational maps $f$ and $g$ of degrees at most $D$ share an iterate, one can write
	\[ f^n=g^m \]
with $n,m\leq M$.
\end{corollary}

\begin{remark}
For polynomials, this type of conclusion goes back to Ritt \cite{Ritt:permutable}; see also the later work on common iterates and common Julia sets, for example \cite{BakerEremenko1987, Beardon:symmetries, Schmidt:Steinmetz}.  Explicit bounds are known for commuting rational maps \cite{Pakovich2020}, and for bicritical maps of even degree one can take the optimal bound $M=2$ \cite{KochLindseySharland2024}.
\end{remark}

The fiber-power condition in Theorem \ref{main1} (3) gives uniform control on the geometry of the Zariski closure of preperiodic points, even allowing a mix of exceptional and non-exceptional maps:

\begin{theorem} \label{degree bound}
Fix integers $D,E,N\geq2$.  There exists $M=M(D,E,N)$ such that, if $V\subset(\bP^1)^N$ is irreducible, has degree $<E$, and projects surjectively to every coordinate, and if $\Phi=(f_1,\ldots,f_N)$ with $\deg f_i\leq D$, then
	\[ \deg\left(\overline{V\cap\Preper(\Phi)}^{\mathrm{Zar}}\right)\leq M. \]
\end{theorem}

\noindent
By convention, the degree of a reducible subvariety will mean the sum of the degrees of its irreducible components.

The proof, given in \S\ref{proof degree bound}, follows an argument we used in \cite[\S5]{DMY:UniformDAO}, but here we rely on the statement of Theorem \ref{main1}; see also \cite{Yuan:bigness, Gao:Ge:Kuhne} for related uniformity arguments.  Recently, Ji and Yap obtained a related uniformity theorem \cite[Theorem 1.4]{JiYap2026}. Assuming the $f_i$ have equal degrees and at most one is exceptional, they showed that the variety
	\[\overline{V\cap\Preper(\Phi)}^{\mathrm{Zar}} \subset (\bP^1)^N \] 
has at most $M$ irreducible components, for a uniform $M$ as in our theorem.  The proofs are not completely independent, as we rely on their archimedean measure-rigidity theorem \cite[Theorem 1.1]{JiYap2026} in our proof, but the overall strategy is very different; while they build on the quantitative fundamental inequality in \cite{Mavraki:Yap}, we build on the theory of Yuan--Zhang \cite{Yuan:Zhang:II} with the goal of proving Thm 1.13 below. 

Combining the uniformity statements above with the characterization of maps sharing their preperiodic points gives the following.

\begin{theorem}\label{theorem: common}
Fix $D>0$.  There exist $M_1(D),M_2(D)>0$ such that the following hold.
\begin{enumerate}
\item If $f$ and $g$ are rational maps over $\C$ of degrees at most $D$, then either
	\[ \#\bigl(\Preper(f)\cap\Preper(g)\bigr)\leq M_1(D) \]
or $\Preper(f)=\Preper(g)$.
\item For degrees $d_1,d_2\leq D$, the locus of non-exceptional pairs $(f,g)$ satisfying $\Preper(f)=\Preper(g)$ is a closed algebraic subvariety of $\Rat_{d_1}^{\mathrm{nonexc}}\times\Rat_{d_2}^{\mathrm{nonexc}}$.  Every pair in this locus satisfies
	\[	f^n\circ g^m=f^{2n}, \qquad g^m\circ f^n=g^{2m} \]
for some $n,m\leq M_2(D)$.
\end{enumerate}
\end{theorem}

\begin{remark}
For non-exceptional $f$ and $g$, results of Levin, Levin--Przytycki, and Yuan--Zhang \cite{Levin:symmetries, Levin:Przytycki, Yuan:Zhang:II} can be combined to identify equality of preperiodic sets with equality of maximal-entropy measures and with the composition relations displayed in Theorem \ref{theorem: common}(2).  Their arguments do not control the iterates $n$ and $m$, however.
\end{remark}

\begin{remark}
Part (1) of Theorem \ref{theorem: common} resolves a conjecture of DeMarco--Krieger--Ye \cite{DKY:quad}. The result also appears in the new article of Ji and Yap \cite[Theorem 1.3]{JiYap2026}, with a different proof.  The polynomial case was also obtained recently by Gong--Yap \cite{GongYap2026}. Earlier partial results include \cite{DKY:quad, DKY:Lattes, Poineau:BFT, Mavraki:Schmidt, Mavraki:Yap, DM:commonprep, Gauthier2024, AngYap2026}.  In previous work we showed that the non-exceptional locus with $\Preper(f)=\Preper(g)$ was {\em contained in} a proper algebraic subvariety of $\Rat_{d_1}^{\mathrm{nonexc}}\times\Rat_{d_2}^{\mathrm{nonexc}}$ \cite{DM:commonprep}; Theorem \ref{theorem: common}(2) rules out infinitely many irreducible components to this locus.
\end{remark}

We conclude this subsection with a uniform bound on solutions to a mixed ``dynamical unit equation", illustrating that exceptional and non-exceptional coordinates can be treated together.

\begin{theorem}
\label{thm:mixed}
Fix integers $d\geq2$ and $q\geq1$.  There exists $B=B(d,q)$ with the following property.  Let $E$ be an elliptic curve over $\C$, let $\pi:E\to E/\{\pm1\}\simeq\bP^1$ be a quotient map, let $f\in\Rat_d(\C)$ be non-exceptional, and let $a_0,a_1,b_1,\ldots,b_q\in\C^*$.  Then
	\[ a_0x+a_1y+\sum_{j=1}^q b_j\zeta_j=1 \]
has at most $B$ nondegenerate solutions
	\[ (x,y,\zeta_1,\ldots,\zeta_q)\in\pi(E_{\mathrm{tors}})\times\Preper(f)\times(\mu_\infty)^q. \]
\end{theorem}

\noindent
Here a solution is {\bf nondegenerate} if there is no nonempty subset $I\subseteq\{1,\ldots,q\}$ such that $\sum_{i\in I} b_i\zeta_i=0$, and $\mu_\infty$ denotes the roots of unity in $\C^*$.

\begin{remark}
Mann's theorem \cite{Mann1965} implies a uniform bound for the number of nondegenerate solutions of linear equations in roots of unity with rational coefficients.  Schlickewei \cite{Schlickewei1996} established such a bound for arbitrary complex coefficients, subsequently sharpened by Evertse \cite{Evertse1999}.  Evertse--Schlickewei--Schmidt
\cite{EvertseSchlickeweiSchmidt2002} later generalized these results proving uniform bounds
for solutions in multiplicative groups of finite rank, depending only on the number of variables and the rank.  Theorem \ref{thm:mixed} is a dynamical analogue with an elliptic torsion coordinate, a non-exceptional preperiodic coordinate, and arbitrarily many multiplicative torsion coordinates. It is merely illustrative; the same argument allows more coordinates and more general bounded-degree subvarieties.
\end{remark}

\subsection{Bifurcations and non-degeneracy: the second main result} \label{stableintro}
Let $S$ be an irreducible quasiprojective variety over $\C$, let
	\[ \Phi=(f_1,\ldots,f_N):S\times(\bP^1)^N\to S\times(\bP^1)^N \]
be an algebraic family, and let $\cX\subset S\times(\bP^1)^N$ be irreducible and flat over $S$, of relative dimension $r$.  Write $\pi:S\times(\bP^1)^N\to S$ for the projection.  Following \cite{GauthierVigny2025}, we define the {\bf bifurcation current} by
	\[ T_{\Phi,\cX}=\pi_*\left(\hat T_\Phi^{\wedge(r+1)}\wedge[\cX]\right) \]
on $S$, where $\hat T_\Phi$ is the dynamical Green current on $S\times (\bP^1)^N$.  This $T_{\Phi,\cX}$ is a positive $(1,1)$-current on $S$, with continuous local potentials.  Put $\ell=\dim S$.  Following \cite{Gauthier:Taflin:Vigny}, the associated {\bf bifurcation measure} on $S(\C)$ is
	\[ \mu_{\Phi,\cX}=\pi^{(\ell)}_*\left(\hat T_{\Phi^{(\ell)}}^{\wedge\ell(r+1)}\wedge[\cX^{(\ell)}]\right), \]
for the projection $\pi^{(\ell)}: S \times (\bP^1)^{N\ell} \to S$.  

\begin{remark}
If $\cX$ is a hypersurface, then direct computation shows that
	\[ \mu_{\Phi,\cX}=c\,(T_{\Phi,\cX})^{\wedge\ell} \]
on $S$, for an explicit constant $c>0$, and in general we have $\mu_{\Phi,\cX}\geq c\,(T_{\Phi,\cX})^{\wedge\ell}$ for any relative dimension $r$ \cite[Proposition 1.4]{Gauthier:Taflin:Vigny}.  In Corollary \ref{measure cor}, we will see that $\mu_{\Phi,\cX} = 0$ $\iff$ $(T_{\Phi, \cX})^{\wedge \ell} = 0$ for any family $\cX$ with generic fiber that projects dominantly to each $\bP^1$ factor; compare \cite[Theorem 4.1]{Yuan:bigness}.
\end{remark}

A hypersurface in $(\bP^1)^N$ is {\bf essential} if its defining equation has positive degree in every coordinate, equivalently if every projection obtained by forgetting one coordinate is dominant.  The pair $(\Phi,\cX)$ is {\bf maximally varying} if its pair-moduli map, recording both the tuple of maps and the hypersurface up to coordinatewise conjugacy, is generically finite onto its image; see \S\ref{moduli map} for definitions.  We first state our result in the simpler setting where $\cX$ is a maximally varying essential hypersurface:

\begin{theorem} \label{hypersurfaces}
Let $S$ be irreducible and quasiprojective over $\C$, with $\dim S>0$.
Let $\Phi=(f_1,\ldots,f_N)$ be an algebraic family with $\deg f_i>1$,
and let $\cX\subset S\times(\bP^1)^N$ be a flat family of irreducible
essential hypersurfaces.  Assume that $(\Phi,\cX)$ is maximally varying.
If the maps $f_i$ are not all Latt\`es families, then the following are equivalent:
\begin{enumerate}
\item $T_{\Phi,\cX}=0$;
\item $\cX$ is $\Phi$-special;
\item $\mu_{\Phi,\cX}=0$.
\end{enumerate}
If all $f_i$ are Latt\`es families, then (1) and (2) remain equivalent, and we have $T_{\Phi,\cX}\neq0$ with ($T_{\Phi,\cX})^{\wedge\dim S}=0$ if and only if $\cX$ is weakly special but not $\Phi$-special and $\dim S>1$.
\end{theorem}

The equivalence of (1) and (2) in Theorem \ref{hypersurfaces} follows from the theorems of Mavraki--Schmidt \cite{Mavraki:Schmidt} in the case that neither $f_i$ is isotrivial; we provide a proof here also for the full statement.  Our emphasis in Theorem \ref{hypersurfaces} is the higher bifurcation information:  away from the all-Latt\`es exception, every nonzero bifurcation current has nonzero top wedge power.

Theorem \ref{hypersurfaces} should be compared with what is known for the classical bifurcation current for a family $f:S\times\bP^1\to S\times\bP^1$ with $\cX=\Crit(f)$.  A non-isotrivial stable family must have a preperiodic critical locus (and therefore is Latt\`es) \cite[Theorem 2.2]{McMullen:families}, while any bifurcating family must bifurcate maximally \cite[Lemma 6.8]{GauthierOkuyamaVigny2019}; see also \cite{D:current, Bassanelli:Berteloot, Dujardin:higherbif, Buff:Epstein:PCF, DM:commonprep}.

We say that $\cX$ is {\bf non-degenerate for $\Phi$} if
	\[ \hat T_\Phi^{\wedge\dim\cX}\wedge[\cX]\neq0, \]
where $\dim\cX=r+ \dim S$, and {\bf degenerate} otherwise.  We say that $\cX$ is {\bf potentially non-degenerate} if
	\[ \hat T_{\Phi^{(\ell)}}^{\wedge\dim\cX^{(\ell)}}\wedge[\cX^{(\ell)}]\neq0, \]
for the $\ell$-th fiber power, and it is {\bf truly degenerate} if $\hat T_{\Phi^{(\ell)}}^{\wedge\dim\cX^{(\ell)}}\wedge[\cX^{(\ell)}] = 0$. Thus true degeneracy is exactly the vanishing of the bifurcation measure $\mu_{\Phi,\cX}$ on $S$. (It is not hard to see that true degeneracy implies degeneracy of all fiber powers $\cX^{(m)}$ for $\Phi^{(m)}$, for all $m\geq 1$.)  

For a general family of subvarieties $\cX$, with any relative dimension $r$, the relevant invariant is what we call the activity dimension of its essential hypersurface shadows:

\begin{theorem} \label{true deg char}
Suppose that $\dim S > 0$.  Let $\Phi = (f_1, \ldots, f_N)$ be an algebraic family of maps on $(\bP^1)^N$ and $\cX \subset S \times (\bP^1)^N$ irreducible, flat over $S$, with relative dimension $0 < \dim_S \cX < N$ such that its generic fiber projects dominantly to every coordinate $\bP^1$.  Then $\cX$ is truly degenerate for $\Phi$ if and only if
	\[ \actdim(\Phi,\cX)<\dim S. \]
\end{theorem}

This notion of degeneracy comes from the study of subvarieties within families of abelian varieties, and it was recently extended to a more general study of polarized endomorphisms (as a notion of ``geometric bigness") in \cite[\S6.2.2]{Yuan:Zhang:quasiprojective} and ``potential bigness" in \cite{Yuan:bigness}.  Theorem \ref{true deg char} is a dynamical analogue of the non-degeneracy characterizations for subvarieties of abelian schemes that enter the relative Manin--Mumford and uniform Mordell--Lang problems; see \cite{GaoBettiRank, GaoBettiRankCorrigendum, Gao:Habegger:RMM, Gao:Ge:Kuhne}.  The activity dimension $\actdim(\Phi,\cX)$ is defined in \S\ref{active dim}.  Non-exceptional shadows of $(\Phi, \cX)$ contribute the complex rank of their pair-moduli maps, while strictly weakly special Latt\`es shadows contribute (half of) the real rank of the corresponding Betti maps.  

\subsection{Algebraic families over $\Qbar$: the third main result}
Assume now that $(\Phi,\cX)$ is defined over $\Qbar$.  The canonical metrics associated to the factors $f_i$ define a nef adelic line bundle $\Lbar_\Phi \in \Pichat\left( (S\times(\bP^1)^N) / \Z\right)$, in the sense of Yuan and Zhang \cite{Yuan:Zhang:quasiprojective}; see Section \ref{section: prelims}.

\begin{theorem} \label{main3}
Let $(\Phi,\cX)$ be as in Theorem \ref{main1}, defined over $\Qbar$ with $\dim S>0$.  Assume that the exceptional factors are not of mixed group type; that is, there do not exist $i\neq j$ such that $f_i$ is exceptional of multiplicative type and $f_j$ is exceptional of elliptic type. The following are equivalent.
\begin{enumerate}
\item The adelic line bundle
	\[ \Lbar_{\Phi^{(1+\dim S)}}|\cX^{(1+\dim S)} \]
has arithmetic volume zero.
\item For one, and hence every, ample Weil height $h_S$ on $S$, there exists a Zariski-dense sequence $s_n\in S(\Qbar)$ such that
	\[ \frac{\hat h_{\Lbar_{\Phi_{s_n}}}(X_{s_n})}{\max\{1,h_S(s_n)\}}\rightarrow0,\text{ as }n\to\infty. \]
\item Either $\cX$ is truly degenerate for $\Phi$, or its non-special essential shadows are simultaneously torsion weakly special.
\end{enumerate}
\end{theorem}

Theorem \ref{main3} says that there are only two mechanisms for failure of arithmetic bigness.  The first is geometric, where the bifurcation measure vanishes, and the second is arithmetic, where the geometric line bundle is big but a collection of quotient sections becomes torsion simultaneously on a dense set.  Outside of these two cases, the canonical height of $X_s$ grows at least linearly with a height of the parameter, away from a proper closed subset of $S$.

\begin{remark} \label{remark:mixed exceptional} 
The restriction on exceptional factors in Theorem \ref{main3} is used only at the final ``Relative Bogomolov" step of the proof (Proposition \ref{homogeneous relative Bogomolov}).  When all exceptional factors are Latt\`es maps, we rely on K\"uhne's relative Bogomolov theorem \cite{Kuhne:RBC}; when they all come from multiplicative groups, we rely on Zhang's toric Bogomolov theorem \cite{Zhang:positive}.  
\end{remark}

\subsection{Proof ideas} \label{sketch}
We view Theorem \ref{hypersurfaces}, and its extension Theorem \ref{true deg char}, as the principal new input.  It is what allows the existing arithmetic equidistribution and rigidity machinery to yield Theorems \ref{main1}, \ref{degree bound}, and \ref{main3}.

To prove Theorem \ref{hypersurfaces}, we first assume that $\cX$ is a family of essential hypersurfaces in $(\bP^1)^N$.  The proof separates according to whether any $f_i$ is a Latt\`es map.  With at least one (but not all) Latt\`es factor $f_i$, we use complex-analytic bifurcation theory, as in \cite{DM:commonprep, DM:UDMM}.  When $\cX$ is not special for $\Phi$, we are able to build proper intersections between the hypersurface $\cX$ and preperiodic points of $\Phi$ that map to repelling cycles, leading to positivity of the bifurcation measure.  If all maps $f_i$ are Latt\`es, we rely on the characterizations of non-degeneracy in \cite{GaoBettiRank, GaoBettiRankCorrigendum} and \cite{Habegger:special}. 

When no factor is Latt\`es, the proof passes through a tower of one-dimensional function fields
	\[ \C=K_0\subset K_1=K_0(\sigma_1)\subset\cdots\subset K_m=K_{m-1}(\sigma_m)\subset\overline{\C(S)}. \]
The functions $\sigma_j$ are algebraically independent multipliers of periodic cycles of some $f_i$.  Passing successively to their generic fibers reduces the dimension of the parameter space while preserving maximal variation of the pairs $(\Phi, \cX)$ over these new fields.  At the terminal stage, we embed $K_m$ into $\C$ and apply McMullen's rigidity theorem \cite[Corollary 2.3]{McMullen:families}; there are no nonconstant multipliers of the resulting families of maps $f_i$, so each such map is isotrivial.  

The main goal is to prove inductively that, assuming $\cX$ is not special for $\Phi$, a certain ``geometric" volume $\volhat_{K_0}\left(\widetilde L_{\Phi^{(\ell)}}|{\cX^{(\ell)}} \right)$ is positive, as this turns out to be equivalent to the positivity of the bifurcation measure $\mu_{\Phi, \cX}$.  The base case of the induction is the isotrivial setting (over the field $K_m$), where the maximal variation of the hypersurface family $\cX$ gives positivity of the corresponding geometric volume directly.  We then propagate this positivity back through the multiplier slices, relying on the formalism from \cite{Yuan:Zhang:quasiprojective}, which relates the geometric volume over our field $K_{i-1}$ to an ``arithmetic" volume over the field $K_{i}$, as 
	\[\volhat_{K_i} \left( \Lbar_{\Phi^{(j)}}|{\cX^{(j)}} \right) = \volhat_{K_{i-1}}\left( \widetilde L_{\Phi^{(j)}}|{\cX^{(j)}} \right) \]
for certain fiber powers $j$.  But, at each step of the induction, we need to boost the geometric positivity of volume over $K_i$ to an arithmetic positivity of volume over that field (as well as boosting the fiber power $j$ by 1 at each step).  This is where we use the dynamics and the choice of multiplier functions $\sigma_i$. If arithmetic positivity were lost at some step, the two-sided Fundamental Inequality of Yuan--Zhang \cite[Theorem 5.3.3]{Yuan:Zhang:quasiprojective} (which requires the geometric positivity and that we work over the function field of a {\em curve}) would produce a generic sequence of small points for two distinct auxiliary adelic line bundles, each obtained by omitting one coordinate of $(\bP^1)^N$.  By the arithmetic equidistribution theorem \cite[Theorem 5.4.3]{Yuan:Zhang:quasiprojective}, the associated measures must agree.  Using the theory of non-archimedean currents of Chambert-Loir and Ducros \cite{CLD}, we slice these measure identities on Berkovich spaces to produce measure identities on families of Berkovich-analytic curves for pairs of maps $(f_{i_1}, f_{i_2})$.  The choice of multiplier functions guarantees the maps have positive Lyapunov exponents at all parameters, and we can apply the non-archimedean measure rigidity theorem of Ji--Yap \cite[Theorem 1.1]{JiYap2026} for curves in $\bP^1\times\bP^1$, which tells us that the curve slices must be preperiodic.  This contradicts essentiality of the hypersurface $\cX$ and will complete the proof of Theorem \ref{hypersurfaces}.

For a subvariety $X \subset (\bP^1)^N$ of arbitrary relative dimension $r$, we consider projections of $X$ to subspaces of the form $(\bP^1)^{r+1}$; the hypersurface images $\cH_J$ (for subsets of indices $|J| = r+1$) are called shadows of $X$.  We show that that the bifurcation current of a family $\cX$ is a positive linear combination of the bifurcation currents of its shadows.  The activity dimension of $\cX$ is defined to record the dimension of the image of the parameter space $S$ in an appropriate moduli space of shadow pairs $(\Phi_J, \cH_J)$ (with a modified definition for the shadows where all factor maps are Latt\`es).  To obtain the characterization of true degeneracy in Theorem \ref{true deg char}, we are then able to combine Theorem \ref{hypersurfaces} with an analytic characterization of positivity of the bifurcation currents and measures, in terms of the existence of repelling cycles that intersect properly with the various shadows.  

To obtain the arithmetic consequences in Theorem \ref{main1} (over $\Qbar$) and Theorem \ref{main3}, we combine this geometric non-degeneracy (i.e., positivity of the bifurcation measure $\mu_{\Phi, \cX}$) with the complex measure rigity of \cite[Theorem 3.1]{Mavraki:Schmidt} building on the equidistribution strategy from \cite[proof of Theorem 1.8]{Mavraki:Schmidt}, later extended and simplified in \cite[Theorem 5.1]{DM:commonprep}.

We first prove the relative dynamical Manin-Mumford type result in Theorem \ref{main1} over $\Qbar$. To pass to $\C$ we argue with induction on the transcendence degree of a field of definition; compare with \cite[\S 10]{Gao:Habegger:RMM} and other classical specialization approaches. 
After reducing to essential hypersurfaces, we spread out one transcendence variable at a time, where the enlarged base dimension matches the available fiber power.  Theorem \ref{hypersurfaces} gives geometric bigness, unless the hypersurface is weakly special, and Yuan--Zhang's height inequality \cite{Yuan:Zhang:quasiprojective} then gives a function-field height gap, contradicting the density of preperiodic points.

Finally, for Theorem \ref{main3}, we use the Deligne pairing of Yuan--Zhang to construct an adelic line bundle on $S$ whose height records the canonical heights of the fibers. Yuan's 
criterion \cite[Theorem 4.1]{Yuan:bigness} identifies its arithmetic bigness with that of the canonical bundle on $\cX^{(1+\dim S)}$. We show that this Deligne bundle is a positive linear combination of the corresponding bundles for the essential shadows, just as we had already done over $\C$ with the bifurcation currents.  Assuming vanishing arithmetic volume and excluding true degeneracy, the fundamental inequality gives a generic sequence small for the Deligne bundle, which then has to be small for every shadow bundle. The arithmetic rigidity argument then forces every non-special shadow to be weakly special.  It remains to pass from simultaneous smallness to simultaneous torsion. For elliptic quotients, we combine K\"uhne's relative Bogomolov theorem \cite{Kuhne:RBC} with the Betti-rank arguments of \cite{GaoBettiRank, Gao:Habegger:RMM}; for multiplicative quotients, we use Zhang's Bogomolov theorem \cite{Zhang:positive}. This is the step where we require the hypothesis in Theorem \ref{main3} that exceptional factors have a common group type.

\subsection{Acknowledgements}
We would like to thank many people for helpful conversations as we worked on this project, especially Ziyang Gao, Thomas Gauthier, Mattias Jonsson, Holly Krieger, Curt McMullen, Harry Schmidt, Jit Wu Yap, Hexi Ye, and Xinyi Yuan. We specially thank Zhuchao Ji and Jit Wu Yap for sharing early drafts of their preprint \cite{JiYap2026}.  Our proof idea for Theorem \ref{hypersurfaces} only crystallized after we learned about \cite[Theorem 1.1]{JiYap2026}.  Our research was supported in part by the National Science Foundation, the Natural Sciences and Engineering Research Council of Canada, the Isaac Newton Institute, and the Harvard Radcliffe Institute.

\subsection{AI statement}
This project extended over several years and progressed at a slow human pace.  However, during the last 6 months, we used various versions of ChatGPT, at first for help with translations and explanations of the results in \cite{CLD}, then for checking our arguments and details, and finally we experimented with drafting some of the text.  

\bigskip
\section{Relations and bifurcations in families}

In this section, we fix terminology and provide background from complex dynamics.  In Theorem \ref{isotrivial case}, we prove a special case of Theorem \ref{hypersurfaces} where the family of maps $\Phi$ is isotrivial (so only the family of hypersurfaces $\cX$ is varying).

\subsection{Relations among maps on $\bP^1$} \label{relations}

We begin with a few standard facts that explain our terminology.  Let
	\[ \Phi=(f_1,\ldots,f_N): (\bP^1)^N\longrightarrow(\bP^1)^N \]
be a product map over $\C$, with $\deg f_i>1$ for every $i$.  The dynamical Manin--Mumford theorems of Ghioca--Nguyen--Ye and Mavraki--Schmidt--Wilms imply that a subvariety containing a Zariski-dense set of $\Phi$-preperiodic points is special in the sense of the Introduction \cite{Ghioca:Nguyen:Ye:DMM2,Ghioca:Nguyen:Ye:DMM1,MSW2023}.  For non-exceptional maps, these subvarieties are assembled from preperiodic curves in pairs of coordinates; see
\cite[Proposition 2.21 and Remark 2.31]{Medvedev:Scanlon}.  

The following elementary observation will be used repeatedly.

\begin{proposition} \label{relation semiconjugacy}
Let $f_1,f_2:\bP^1\to\bP^1$ have degrees $>1$, and let $C\subset\bP^1\times\bP^1$ be an irreducible curve that projects dominantly to both factors.  If $C$ is fixed by $(f_1,f_2)$, then
\[ \deg f_1=\deg f_2. \]
Moreover, writing $\phi=(f_1,f_2)|_C:C\to C$ and $\pi_i:C\to\bP^1$ for the coordinate projections, we have
\[ f_i\circ\pi_i=\pi_i\circ\phi, \qquad i=1,2. \]
\end{proposition}

\begin{proof}
The semiconjugacy relations follow from the definition of $\phi$.  Taking degrees gives
\[ (\deg f_i)(\deg\pi_i)=(\deg\pi_i)(\deg\phi), \]
and the projections are nonconstant, so $\deg f_1=\deg\phi=\deg f_2$.
\end{proof}

For example, if $C \subset \bP^1\times\bP^1$ is the graph of a M\"obius transformation $A$, then invariance of $C$ by $(f_1, f_2)$ is equivalent to conjugacy of $f_1$ and $f_2$ by $A$.  If $C$ is the graph of a rational function $h$ of higher degree, then invariance gives a semiconjugacy relation $h\circ f_1=f_2\circ h$.  For pairs of non-exceptional rational maps, Pakovich describes periodic and preperiodic curves in terms of semiconjugacies and decompositions of iterates \cite[Theorems~4.6 and~4.15]{PakovichInvariant2023}.

The diagonal line in $\bP^1\times\bP^1$ gives a relation that is especially important for our applications.  We record the precise form that we need.

\begin{proposition} \label{diagonal relations}
Let $f,g:\bP^1\to\bP^1$ have degrees $>1$, and let $\Delta\subset\bP^1\times\bP^1$ be the diagonal.
\begin{enumerate}
\item The curve $\Delta$ is periodic for $(f,g)$ if and only if $\deg f = \deg g$ and $f$ and $g$ share an iterate.
\item The curve $\Delta$ is preperiodic for a map $F := (f^a,g^b)$ for some $a,b\in\bN$ if and only if 
$d:=\deg(f)^a=\deg(g)^b$ and there exists a positive integer $k$ such that 
	\[ f^{ka}\circ g^{kb}=f^{2ka} \qquad\hbox{and}\qquad g^{kb}\circ f^{ka}=g^{2kb}. \] 
If $F^{q+p}(\Delta)=F^q(\Delta)$, the integer $k$  may be bounded in terms of $p,q,d$.
\end{enumerate}
\end{proposition}

Combining Proposition \ref{diagonal relations} with Theorem \ref{preperiodbound} proves the bounded-complexity assertion in Theorem \ref{theorem: common}(2).  We return to this in \S\ref{uniformity}.

\begin{proof}
Statement (1) is immediate (with Proposition \ref{relation semiconjugacy}). 
For (2), we follow a step in the proof of \cite[Theorem 3.1]{Ye:symmetries}, which itself builds on the proofs in \cite{Levin:symmetries, Levin:Przytycki}.  

Assume that $\Delta$ is preperiodic for the map $F$.  From Proposition \ref{relation semiconjugacy}, we can replace $f$ and $g$ with iterates $f^a$ and $g^b$ with the same degree.  So we may assume that $d := \deg f = \deg g$ and that $F^{p+q}(\Delta) = F^q(\Delta)$ for integers $q\geq 0$ and $p > 0$.  Set $C = F^q(\Delta)$, and let $\pi_i: C \to \bP^1$ be the projection to the $i$-th coordinate, for $i = 1,2$.  Choose an open subset $V \subset \bP^1$ containing no critical values for $f^q$ nor for $\pi_1$, and choose a branch of the inverse of $\pi_1$ and set 
	\[ h = \pi_2\circ \pi_1^{-1}. \]
Shrink $V$ if necessary so that there are no critical values of $g^q$ in $V' = h(V)$.  Since the curve $C$ is parameterized by both $u \mapsto (f^q(u), g^q(u))$ and $w \mapsto (f^{q+p}(w), g^{q+p}(w))$, we can find open $U, W \subset \bP^1$ so that 
	\[ h\circ f^q = g^q  \mbox{ on } U \]
and 
	\[ h\circ f^{q+p} = g^{q+p} \mbox{ on } W.\]
It follows immediately that 
	\[ (f^{-q} \circ f^q) \circ f^p = (g^{-q} \circ g^q) \circ g^p \]
from $W$ to $U$, for inverse branches of $f^q$ and $g^q$ from $V$ and $V'$ back to $U$.  Or, alternatively, we observe that 
	\[ g^q \circ (f^{-q} \circ f^q) \circ f^p = g^{q+p}. \]
Now we repeat the construction above for multiples of $p$ making compatible choices as follows. Choose $z_0$ outside the postcritical set of $f$ such that $f^q(z_0)$ is not a critical value of $f^q$ or $\pi_1$, and choose a backward orbit
	\[f^p(z_{n+1})=z_n,~ n\geq0.\]
For each $n\geq1$, choose the branch $h_n$ of $\pi_2\circ\pi_1^{-1}$ near $f^q(z_0)$ for which
	\[h_n\circ f^{q+np}=g^{q+np}~\text{near }z_n.\]
Such a branch exists because $F^{q+np}(\Delta)=C$.  As above, choose an inverse branch $f^{-q}$ satisfying $g^q\circ f^{-q}=h_n$. Then $a_n=f^{-q}\circ f^q$ is a germ at the common base point $z_0$, and
	\[g^q\circ a_n\circ f^{np}=g^{q+np}~\text{near }z_n.\]

There are at most $d^q$ possibilities for the germs $a_n$ at $z_0$, so there exist indices $n''>n'>0$ such that $a_{n''}=a_{n'}$. Since $f^{(n''-n')p}(z_{n''})=z_{n'},$
we have 
\begin{align*}
g^{q+n''p}&=g^q\circ a_{n''}\circ f^{n''p}\\
&=g^q\circ a_{n'}\circ f^{n'p} \circ f^{(n''-n')p}\\
&=g^{q+n'p}\circ f^{(n''-n')p}
\end{align*}
near $z_{n''}$, and hence globally as an identity of rational
functions.
Note that we can choose both $n'$ and $n''$ to be bounded in terms of $d,p,q$ and so that $m:=(n''-n')p\geq q+n'p$. 
Therefore, composing both sides of the relation on the left with $g^{m-(q+n'p)}$, we obtain
\[g^m\circ g^m=g^m\circ f^m.\]
By symmetry, we obtain the other relation, though possibly with a different $m$. As the relation holds when its exponent is replaced by positive multiples of $m$, we can obtain a common exponent $k$ bounded in terms of $p,q,d$. In the original notation, this gives
	\[f^{ka}\circ g^{kb}=f^{2ka},~g^{kb}\circ f^{ka}=g^{2kb}.\]

For the converse implication in (2), again writing $f,g$ for $f^a,g^b$, assume that
$g^m\circ g^m=g^m\circ f^m$ and $f^m\circ f^m=f^m\circ g^m$ for some $m\geq1$. Then
\begin{eqnarray*}
(f^m,g^m)^2(\Delta)
	&=&\{(f^{2m}(z),g^{2m}(z)):z\in\bP^1\}\\
	&=&\{(f^m(f^m(z)),g^m(f^m(z))):z\in\bP^1\}=(f^m,g^m)(\Delta),
\end{eqnarray*}
so $\Delta$ is preperiodic.
\end{proof}

\subsection{Moduli spaces and maximal variation} \label{moduli map}
Let $S$ be an irreducible quasiprojective variety, let
\[ \Phi=(f_1,\ldots,f_N):S\times(\bP^1)^N\rightarrow S\times(\bP^1)^N \]
be an algebraic family of product maps, with $d_i = \deg f_i > 1$.  Let $\cX\subset S\times(\bP^1)^N$ be a flat family of subvarieties over $S$.  Fix a projective embedding of $(\bP^1)^N$ to define degrees of subvarieties.  Fixing a degree bound of $E$ on the subvarieties, let $\mathcal C_E$ be the corresponding Chow variety.  The pair $(\Phi, \cX)$ determines a rational moduli map
\begin{equation} \label{induced from S}
\rho_{\Phi,\cX}:S\dashrightarrow
	\left(\Rat_{d_1}\times\cdots\times\Rat_{d_N}\times\mathcal C_E\right)/(\PGL_2)^N,
\end{equation}
for some $E$, where $(\PGL_2)^N$ acts by coordinatewise conjugacy on the maps and by the induced action on the cycle.  We say that $(\Phi,\cX)$ is {\bf maximally varying} if $\rho_{\Phi,\cX}$ is generically finite onto its image.  We say that $(\Phi, \cX)$ is {\bf isotrivial} if $\rho_{\Phi, \cX}$ is a constant map.

\subsection{The dynamical Green current}  \label{Green}
Let $\omega_0$ denote a smooth and positive (1,1)-form on $\bP^1$ (such as the Fubini-Study form), normalized so that $\int_{\bP^1} \omega_0 = 1$.  For a family 	
	\[ f: S \times \bP^1\to S \times\bP^1 \] 
of maps of degree $d>1$ on $\bP^1$, recall that its dynamical Green current is defined by 
	\[ \hat T_f = \lim_{n\to \infty} \frac{1}{d^n} (f^n)^* (p^*\omega_0),\]
for the projection $p: S \times \bP^1 \to \bP^1$.  This current has continuous local potentials and satisfies $\frac{1}{d} f^* \hat T_f = \hat T_f$.  Slicing $\hat T_f$ along a fiber $\bP^1$ over $s \in S$ gives the measure $\mu_{f_s}$ of maximal entropy for $f_s$; see for example \cite[\S3]{Dujardin:Favre:critical}.

Given a family $\Phi=(f_1,\ldots,f_N):S\times(\bP^1)^N\longrightarrow S\times(\bP^1)^N$ as in \S\ref{moduli map}, we let
	\[ \hat T_{\Phi} := \sum_j (\pi_j)^* \hat T_{f_j} \] 
on $S\times(\bP^1)^N$, for each of the projections $\pi_j: S \times (\bP^1)^N \to S \times \bP^1$.  Note that this current $\hat T_\Phi$ is not invariant under (normalized) pullback by $\Phi$ unless the degrees of the $f_j$ coincide.  Nevertheless, the top wedge power $\hat T_{\Phi}^{\wedge N}$ is strongly invariant, in the sense that 
\begin{align}\label{top wedge invariance}  
\frac{1}{d_1d_2\cdots d_N} \Phi^*  (\hat T_{\Phi}^{\wedge N}) = \hat T_{\Phi}^{\wedge N}.
\end{align}

Now let $\cX\subset S\times(\bP^1)^N$ be an irreducible and flat variety over $S$, of relative dimension $r$.  Write $\pi:S\times(\bP^1)^N\to S$ for the projection.  Recall that the bifurcation current was defined by
	\[ T_{\Phi,\cX}=\pi_*\left(\hat T_\Phi^{\wedge(r+1)}\wedge[\cX]\right) \]
on $S$ \cite{GauthierVigny2025}.  We say that the pair $(\Phi, \cX)$ is {\bf stable} if $T_{\Phi, \cX} = 0$.

\subsection{Isotrivial or special families}
Here we observe that stability holds in the expected settings, proving the ``easy" implication in Theorem \ref{hypersurfaces} and a step in the proof of Theorem \ref{true deg char}. 

\begin{proposition} \label{easy stability}
Let $\Phi$ be a family of product maps on $S\times(\bP^1)^N$, and let $\cX\subset S\times(\bP^1)^N$ be flat of relative dimension $r$.  If $\cX$ is $\Phi$-special, or if the pair $(\Phi,\cX)$ is isotrivial, then
	\[ T_{\Phi,\cX}=\pi_*\left(\hat T_\Phi^{\wedge(r+1)}\wedge[\cX]\right)=0. \]
\end{proposition}

\begin{proof}
If the pair is isotrivial, then after base change and conjugacy, both $\Phi$ and $\cX$ are pulled back from a single pair $(F,X)$ over a point, via a projection $p: S \times(\bP^1)^N \to (\bP^1)^N$, and $\dim X = r$.  Thus
\[ \hat T_\Phi^{\wedge(r+1)}\wedge[\cX] =p^*\left(T_F^{\wedge(r+1)}\wedge[X]\right)=0 \]
for dimension reasons.  If $\cX$ is special, replace $\Phi$ by a commuting polarized product map for which $\cX$ is preperiodic.  The invariance of the Green current then gives the vanishing; equivalently, this is one of the stability criteria in \cite[Theorem 4.4]{GauthierVigny2025}; see also \cite[Proposition 2.2]{DM:UDMM}.
\end{proof}

The following is an important special case of Theorem \ref{hypersurfaces}, and it will serve as a base case for the inductive proof that we give in general.

\begin{theorem} \label{isotrivial case}
Let $S$ be irreducible of dimension $\ell>0$, let $\Phi$ be an isotrivial family of product maps on $(\bP^1)^N$, and let $\cH\subset S\times(\bP^1)^N$ be a flat family of irreducible hypersurfaces.  If $(\Phi,\cH)$ is maximally varying, then
	\[ \hat T_{\Phi^{(\ell)}}^{\wedge N\ell}\wedge[\cH^{(\ell)}]\neq0.  \] 
In particular, $\mu_{\Phi,\cH}\neq0$.
\end{theorem}

\begin{proof}
After a generically finite base change and coordinatewise conjugacy, we may assume that $\Phi$ is independent of $s\in S$, equal to a single map $\Phi_0$ on $(\bP^1)^N$.  Let
	\[ q_\ell:\cH^{(\ell)}\rightarrow((\bP^1)^N)^\ell \]
be the evaluation map. Since $\Phi$ is constant, maximal variation of $(\Phi,\cH)$ implies that the Chow map $s\mapsto[H_s]$ is generically finite. By \cite[Proposition 2.1]{DMY:UniformDAO}, the evaluation map $q_\ell$ is generically finite. Since its source and target both have dimension $N\ell$, it is also dominant.   

The Green current of the constant family is pulled back from the dynamical space.  Consequently,
\[
\hat T_{\Phi^{(\ell)}}^{\wedge N\ell}\wedge[\cH^{(\ell)}]
=q_\ell^*\left(T_{\Phi_0^{(\ell)}}^{\wedge N\ell}\right),
\]
up to the positive degree of $q_\ell$.  The measure on the right has positive mass, so the current on the left is nonzero.
\end{proof}

\subsection{Rigid prerepelling intersections}   \label{rigid repellers} 
We recall the local criterion that equates positivity of a current with the existence of a dynamical intersection.  Let $S$ be quasiprojective complex algebraic variety over $\C$ of dimension $\ell > 0$.  Suppose $\Phi = (f_1, \ldots, f_N):  S\times(\bP^1)^N \to S \times (\bP^1)^N$ is an algebraic family of maps defined over $\C$, and let $\pi: S \times(\bP^1)^N \to S$ denote the projection.  Suppose that $\cX\subset S \times(\bP^1)^N$ is an irreducible family of subvarieties of codimension $\ell = \dim S$, which is flat over $S$.  Note that $N = \dim \cX$.  

We say that a point $y_0 \in \cX(\C)$ is a {\bf rigid repeller} for the pair $(\Phi, \cX)$ if  
\begin{enumerate}
\item  some iterate $x_0 = \Phi^{n_0}(y_0)$ is a repelling periodic point for $\Phi_{s_0}$, where $s_0 = \pi(y_0)$;  
\item there is a holomorphic section $\eta$ over a neighborhood of $s_0$ in $S$ parameterizing a repelling periodic point of $\Phi_s$, with $\eta(s_0) = x_0$, so that $x_0$ is an isolated point of the intersection of $\Phi^{n_0}(\cal{X})$ with the image of $\eta$.
\end{enumerate} 
Compare \cite[\S4.3]{DM:commonprep} where this notion is defined for families of polarized dynamical systems.  

\begin{proposition} \label{analytic characterization}
Let $S$ be quasiprojective complex algebraic variety over $\C$ of dimension $\ell > 0$.  Suppose $\Phi = (f_1, \ldots, f_N):  S\times(\bP^1)^N \to S \times (\bP^1)^N$ is an algebraic family of maps defined over $\C$.  Suppose that $\cX\subset S \times(\bP^1)^N$ is an irreducible family of subvarieties of codimension $\ell = \dim S$, which is flat over $S$.  Then
	\[ \hat T_\Phi^{\wedge N}\wedge[\cX]\neq0 \]
if and only if $\cX$ contains a rigid repeller for the pair $(\Phi, \cX)$.  Moreover, rigid repellers are dense in the support of this positive measure. If the measure is nonzero, then, for every sufficiently large $n$, the image $\Phi^n(\cX)$ contains a Zariski-dense set of repelling periodic points whose local periodic sections meet $\Phi^n(\cX)$ transversely.
\end{proposition}

\begin{proof}
A proof that the existence of the repeller implies positivity of the measure is given in \cite[Proposition 4.8]{DM:commonprep} for families of polarized dynamical systems; it is itself a minor modification of \cite[Lemma 4.8]{GauthierVigny2025} and \cite[Proposition 3.7]{Berteloot:Bianchi:Dupont}.  While our $\Phi$ need not be polarized, the top wedge power $\hat T_\Phi^{\wedge N}$ is invariant, as in \eqref{top wedge invariance}, so the same proof works.  

For the converse implication and density in the support, we use the construction in \cite[\S6]{DM:UDMM}, following \cite[Theorem 0.1]{Dujardin:higherbif}, \cite[Theorem 1.6 and Proposition 3.8]{Berteloot:Bianchi:Dupont}, and \cite[Theorem 2.2]{Gauthier:abscont}. We explain its coordinatewise form, which also gives the assertion for a fixed iterate. Choose a small parameter ball $U$ over which the top Green measure is nonzero. For each coordinate map $f_i$, choose a moving hyperbolic set with dense repelling periodic points and let $L_i$ be the associated uniformly laminar $(1,1)$-current of fiber mass $1$, pulled back to $U\times(\bP^1)^N$. These currents have bounded local potentials; see \cite[Lemmas 4.1--4.2]{Dujardin:higherbif}. Write $T_i$ for the coordinate Green currents and $d_i=\deg f_i$. On a smaller parameter ball we may write $L_i=T_i+dd^cu_i$ with $u_i$ bounded. Invariance gives
\[
d_i^{-n}(\Phi^n)^*L_i
=T_i+dd^c\bigl(d_i^{-n}u_i\circ\Phi^n\bigr).
\]
The local potentials therefore converge uniformly on compact subsets to those of $T_i$, and hence
\[\left(\bigwedge_{i=1}^N d_i^{-n}(\Phi^n)^*L_i\right)\wedge[\cX]\rightarrow\left(\bigwedge_{i=1}^N T_i\right)\wedge[\cX]=\frac1{N!}\hat T_\Phi^{\wedge N}\wedge[\cX].\]
The projection formula therefore gives
\[L_1\wedge\cdots\wedge L_N\wedge[\Phi^n(\cX)]\ne0\]
for every sufficiently large $n$. For each fixed such $n$, this is a top-degree measure formed from currents with bounded local potentials, so it gives no mass to proper algebraic subsets of $\Phi^n(\cX)$. By \cite[Theorem 3.1]{Dujardin:higherbif}, the intersection is geometric, and almost every contributing section meets $\Phi^n(\cX)$ transversely. We may therefore choose such an intersection outside any prescribed proper algebraic subset. Approximating each coordinate by repelling periodic points preserves transversality and avoids that subset. This proves the last assertion. The same argument localized near a point in the support of the original measure gives the stated density of rigid repellers in $\cX$.
\end{proof}

\bigskip
\section{Adelic line bundles} \label{section: prelims}

In this section we introduce some notation and results from Yuan--Zhang \cite{Yuan:Zhang:quasiprojective} that will be used throughout the paper.  We work in one of the following three settings.
\begin{enumerate}
\item $k=\bZ$ and $K$ is a number field;
\item $k$ is a field of characteristic $0$ and $K$ is the function field of a curve over $k$; or
\item $k=K=\C$, in which case we work locally with line bundles carrying continuous Hermitian metrics.
\end{enumerate}
In the first two cases we use the adelic Picard groups of \cite{Yuan:Zhang:quasiprojective}.  In the complex case the same notation will refer to the corresponding local metrized Picard group, where forgetting the metric gives the ordinary line bundle.  

\subsection{Adelic line bundles and intersection theory}
Let $X$ be an essentially quasiprojective integral scheme over $k$;  see \cite[\S2.3.2]{Yuan:Zhang:quasiprojective}. We write
	\[\Pichat(X/k)\]
for the group of isomorphism classes of adelic line bundles on $X$; see \cite[\S2.5]{Yuan:Zhang:quasiprojective}.  We write $\Pichat(X/k)_{\nef}$ for the nef cone and $\Pichat(X/k)_{\mathrm{int}}$ for the group of integrable classes.  Recall that a strongly nef class is represented by a Cauchy sequence of nef model line bundles, while a nef class $\Lbar \in \Pichat(X/k)$ is one for which there exists a strongly nef $\overline{M} \in \Pichat(X/k)$ such that $m\Lbar + \Mbar$ is strongly nef for all positive integers $m$ \cite[\S2.5.2]{Yuan:Zhang:quasiprojective}.

The construction is functorial.  A morphism $u:X\to Y$ induces
\[u^*:\Pichat(Y/k)\rightarrow\Pichat(X/k),\]
and, when $X$ is an essentially quasiprojective scheme over $k$, there is a natural varying the base map
	\[\Pichat(X/k)\rightarrow\Pichat(X/K);\]
see \cite[\S2.5.5]{Yuan:Zhang:quasiprojective}.  For $\Lbar\in\Pichat(X/k)$, we denote its image by $\widetilde L \in \Pichat(X/K)$ and call it the {\bf geometric part} of $\Lbar$.  In the complex setting, where $X$ is a quasiprojective variety over $\C$, $\widetilde L$ is simply the underlying line bundle.

If $X$ is flat and integral over $k$, Yuan--Zhang define multilinear intersection products of integrable adelic line bundles; see \cite[Proposition 4.1.1]{Yuan:Zhang:quasiprojective}.  We will repeatedly use the projection formula.  If $u:X\to Y$ is dominant and generically finite and the absolute dimensions agree, then
\begin{equation} \label{proj formula}
(u^*\Lbar_1)\cdots(u^*\Lbar_d)
=(\deg u)\,\Lbar_1\cdots\Lbar_d,
\end{equation}
where $d$ is the absolute dimension of quasi-projective models for both $X$ and $Y$; see \cite[Proposition 4.1.2]{Yuan:Zhang:quasiprojective}.  The same formula holds with $\deg u = 0$ when $u$ is not dominant.

For a projective flat morphism $p:X\to S$ of relative dimension $r$, there is also a Deligne pairing
	\[\left\langle\Lbar_0,\ldots,\Lbar_r\right\rangle_{X/S}\in\Pichat(S/k),\]
multilinear and symmetric in the entries, compatible with base change and with the usual projection formulas; see \cite[Theorem 4.1.3 and Lemma 4.6.1]{Yuan:Zhang:quasiprojective}.  Its curvature at a place $v$ is the fiber integral
	\[c_1\left(\left\langle\Lbar_0,\ldots,\Lbar_r\right\rangle_{X/S} \right)_v 
		= p_*\left(c_1(\Lbar_0)_v\wedge\cdots\wedge c_1(\Lbar_r)_v\right).\]
This identity is interpreted in the sense of Chambert-Loir--Ducros \cite{CLD} at a non-archimedean place and in the usual sense of currents over $\C$.

\subsection{Volumes, heights, and equidistribution}
Let $X$ be a quasiprojective variety over $k$.  For a nef adelic line bundle $\Lbar$ on $X$, we have 
\begin{equation} \label{adelic volume formula}
	\volhat(\Lbar)=\Lbar^{\, d_k(X)},
\end{equation}
where $d_k(X)$ is the absolute dimension of $X$ over $k$; see \cite[Theorem 5.2.2]{Yuan:Zhang:quasiprojective}.  The geometric part is big if and only if
	\[ (\widetilde L)^{\dim_K X}>0. \]
At every place $v$ of $K$, this geometric intersection number is the mass of the associated measure (the Chambert-Loir measure at a non-archimedean place; see \cite[Definition 2.4 and Proposition 2.7]{ChambertLoir:equidistribution}),
\begin{equation} \label{geometric volume formula}
	(\widetilde L)^{\dim_K X} = \int_{X_v^{\an}}c_1(\Lbar)_v^{\wedge\dim_K X};
\end{equation}
see \cite[Lemma 5.4.4]{Yuan:Zhang:quasiprojective} and \cite[Theorem 1.1]{Guo} for the case of $k = K = \C$.  Also in the complex setting, the equality \eqref{geometric volume formula} for the dynamical metrics used in this article is proved in \cite[Lemma 1.9]{Gauthier:Taflin:Vigny}.

Assume now that $K$ is either a number field or a one-variable function field over $k$.  The line bundle $\Lbar$ defines a height
	\[ h_{\Lbar}:X(\Kbar)\rightarrow\R. \]
For an integral subvariety $Z\subset X$ with $(\widetilde L|_Z)^{\dim Z}>0$, we use the normalization
	\[ h_{\Lbar}(Z) = \frac{(\Lbar|_Z)^{\dim Z+1}} {(\dim Z+1)(\widetilde L|_Z)^{\dim Z}}. \]
The {\bf essential minimum} is
	\[ e(X,\Lbar) = \sup_{U\subset X}\inf_{x\in U(\Kbar)}h_{\Lbar}(x), \]
where $U$ ranges over nonempty Zariski-open subsets.

\begin{theorem}[Fundamental inequality] \label{fundamental inequality}
Suppose that $K$ is either a number field or a one-variable function field over $k$.  Let $X$ be a quasiprojective variety of dimension $n$ over $K$, and let $\Lbar\in\Pichat(X/k)_{\nef}$.  If the geometric part $\widetilde L$ is big, then
	\[ e(X,\Lbar) \geq h_{\Lbar}(X) \geq \frac{1}{n+1}e(X,\Lbar).  \]
\end{theorem}

\begin{proof}
This is \cite[Theorem 5.3.3]{Yuan:Zhang:quasiprojective}.  For projective $X$ and number fields $K$, it is due to Zhang \cite{Zhang:adelic}, and if defined over $K = k(B)$ for a curve $B$, the corresponding fundamental inequality in the projective case is due to Gubler \cite{Gubler:Bogomolov}.
\end{proof}

A sequence $\{x_n\}\subset X(\Kbar)$ is {\bf generic} if no subsequence is contained in a proper closed subvariety of $X$.

\begin{theorem}\label{equidistribution}
Let $X$ and $\Lbar$ be as in Theorem \ref{fundamental inequality}.  Suppose that $\{x_n\}$ is generic and
	\[ h_{\Lbar}(x_n)\rightarrow h_{\Lbar}(X) \text{, as } n\to\infty.\]
Then, at every place $v$ of $K$, the Galois orbits of $x_n$ equidistribute on $X_v^{\an}$ toward the normalized measure
	\[ \frac{1}{(\widetilde L)^n} c_1(\Lbar)_v^{\wedge n}. \]
\end{theorem}

\begin{proof}
This is \cite[Theorem 5.4.3]{Yuan:Zhang:quasiprojective}.
\end{proof}

\subsection{Projections and non-archimedean slicing} \label{slicing}
We state a non-archimedean Fubini-type statement that we need for our proofs in Section \ref{ordinary hypersurfaces}.  Let
	\[ p:A\rightarrow B \]
be flat and projective of relative dimension $m$, with $A$ and $B$ smooth and quasiprojective over $K$.  Let
	\[ \Lbar_1,\ldots,\Lbar_m\in\Pichat(A/k)_{\nef} \]
be strongly nef, and let
	\[ \Nbar_1,\ldots,\Nbar_b\in\Pichat(B/k)_{\nef}, \qquad b=\dim B, \]
be strongly nef.  At a non-archimedean place $v$, put
	\[ \mu_{B,v}=c_1(\Nbar_1)_v\wedge\cdots\wedge c_1(\Nbar_b)_v \]
and
	\[ \mu_{A,v}=c_1(\Lbar_1)_v\wedge\cdots\wedge c_1(\Lbar_m)_v\wedge p^*\mu_{B,v}. \]
For $y\in B_v^{\an}$, let $A_y$ be the analytic fiber over the completed residue field $\cH(y)$ and set
	\[ \mu_{y,v}=c_1(\Lbar_1|_{A_y})_v\wedge\cdots\wedge c_1(\Lbar_m|_{A_y})_v. \]

\begin{lemma} \label{fiber continuity}
For every continuous compactly supported function $\varphi$ on $A_v^{\an}$, the function
	\[ y\mapsto\int_{A_y}\varphi\,d\mu_{y,v} \]
is continuous on $B_v^{\an}$.
\end{lemma}

\begin{proof}
This statement is essentially contained in \cite[Theorem 4.1.3]{Yuan:Zhang:quasiprojective}, but we provide details here.

First choose quasiprojective models over $k$ for $A$ and $B$ and Cauchy sequences of projective compactifications of $A$ and adelic line bundles $(A_n, (\Lbar_{i,n})_n, \ell_{i,n})$ in $\Pichat(A/k)$, for each $i = 1, \ldots, m$, in the sense of \cite{Yuan:Zhang:quasiprojective}.  We may choose the $A_n$ so that $p_n$ extends to a morphism $p: A_n \to B_n$ for some projective compactification of $B$.  

Restricting to any compact subset $E$ of $A_v^\an$, we may approximate the given continuous function $\phi$ uniformly by model functions $\phi_j$ defined on the compactification $A_1$; see \cite[Theorem 7.12]{Gubler:local} and \cite[Lemma 3.5]{Yuan:equidistribution}.  The function $\phi_j$ induces a function on each $A_n$ by pullback that we continue to write as $\phi_j$, and we may write 
	\[ \phi_j = -\log\|1\|_j  \]
for a model metric on the trivial line bundle $\Mbar_{j,n}$ on $A_n$.  Moreover, we may assume that $\Mbar_{j,n}$ converges as $n\to\infty$ in the boundary topology of \cite{Yuan:Zhang:quasiprojective} to an element $\Mbar_{\phi_j} \in \Pichat(A/k)$ which is integrable.  

Consider the functions 
	\[ F_{j,n}(b) = \int_{A_b} \phi_j  \, c_1(\Lbar_{1,n}|_{A_b})_v \wedge \cdots \wedge c_1(\Lbar_{m,n}|_{A_b})_v \]
on $B_v^{\an}$.  As described in \cite[Chapter 4]{Yuan:Zhang:quasiprojective}, this is the Green's function for the metric on the Deligne pairing 
	\[ \< \Mbar_{j,n}, \Lbar_{1,n}, \ldots, \Lbar_{m,n}\> \]
on $B_n$; these metrics are continuous over $B_v^{\an}$.  

From \cite[Theorem 4.1.3]{Yuan:Zhang:quasiprojective} and its proof, the sequence $\< \Mbar_{j,n}, \Lbar_{1,n}, \ldots, \Lbar_{m,n}\>$ converges as $n\to\infty$ in the boundary topology to the element 
	\[\< \Mbar_{\phi_j}, \Lbar_{1}, \ldots, \Lbar_{m}\> \in \Pichat(B/k). \]
In particular, its Green's function is given by 
	\[F_j(b) :=  \int_{A_b} \phi_j \, \mu_{b,v}  \]
as in \cite[Theorem 4.6.2]{Yuan:Zhang:quasiprojective}.  The function $F_j$ is continuous, and the convergence of $F_{j,n}$ to $F_j$ is locally uniform, as in \cite[Lemma 3.6.3]{Yuan:Zhang:quasiprojective}.

For each fixed $b \in B_v^{\an}$, the fiber $A_b$ is projective, and the measures $c_1(\Lbar_{1,n}|_{A_b})_v \wedge \cdots \wedge c_1(\Lbar_{m,n}|_{A_b})_v$ converge weakly to $\mu_{b,v}$.  Moreover, the total masses converge to the integral of $\mu_{b,v}$, by the computation of volume in the projective case as the limit of intersection numbers; see \cite[Definition 2.4 and Proposition 2.7]{ChambertLoir:equidistribution} and \cite[Proposition 3.12]{Gubler2007Tropical}.  Flatness of $p$ over $B$ implies that these masses are independent of the base point $b$.  It then follows from the uniform convergence of $\phi_j$ to $\phi$ (as $j\to\infty$) on the compact subset $E$ of $A_v^\an$ that the functions $F_j$ converge uniformly to $F$ on $E$.  We conclude that $F$ is continuous.  
\end{proof}

\begin{proposition} \label{Fubini}
For every continuous compactly supported function $\varphi$ on $A_v^{\an}$,
\[
\int_{A_v^{\an}}\varphi\,d\mu_{A,v}
=
\int_{B_v^{\an}}
\left(\int_{A_y}\varphi\,d\mu_{y,v}\right)d\mu_{B,v}(y).
\]
\end{proposition}

\begin{proof}
As in the proof of Lemma \ref{fiber continuity}, we first choose quasiprojective models over $k$ for $A$ and $B$ and Cauchy sequences of nef model line bundles $\Lbar_{i,r}$ on projective compactifications $A_r$ of $A$, and $\Nbar_{j,n}$ on projective compactifications $B_n$ of $B$, representing $\Lbar_i$ and $\Nbar_j$, respectively, in the sense of \cite{Yuan:Zhang:quasiprojective}.  We use independent indices $r$ and $n$ for the approximations on $A$ and $B$.  After passing to common refinements, we may extend $p$ to the projective compactifications used for each pair $(r,n)$.  Each approximation comes equipped with a formal metric (in the sense of \cite[\S17.7]{CLD}) at the place $v$.  We work with the restrictions of these metrics to $A_v^{\an}$ and $B_v^{\an}$, and put
	\[ \mu_{B,n}=c_1(\Nbar_{1,n})\wedge\cdots\wedge c_1(\Nbar_{b,n}) =\sum_y\lambda_{y,n}\delta_y, \]
where the sum is taken over the finite support of this measure in $B_v^{\an}$.  More precisely, as the approximating metrized line bundles are nef, the curvature currents are positive and their products are well defined.  In \cite[\S18.4]{CLD}, Chambert-Loir and Ducros identify these products with Chambert-Loir measures of \cite{ChambertLoir:equidistribution}, with computations given in \cite[\S18.5]{CLD}.  For these formal metrics, \cite[Proposition 18.6.1(b)]{CLD} states that the current
	\[\bigwedge_{j=1}^{b}c_1(p^*\Nbar_{j,n})=\sum_y\lambda_{y,n}[A_y]\]
on $A_v^{\an}$ is a weighted sum of currents of integration over fibers of $p$.  We then wedge with the currents $c_1(\Lbar_{1,r})\wedge\cdots\wedge c_1(\Lbar_{m,r})$ to form a measure, giving
\begin{align*}
c_1(\Lbar_{1,r})\wedge\cdots\wedge c_1(\Lbar_{m,r})\wedge
\bigwedge_{j=1}^{b}c_1(p^*\Nbar_{j,n}) = \sum_y\lambda_{y,n}\,
c_1(\Lbar_{1,r}|_{A_y})\wedge\cdots\wedge c_1(\Lbar_{m,r}|{A_y})
\end{align*}
for each pair $(r,n)$; see \cite[Proposition 18.3.2]{CLD} with $j$ the inclusion morphism of each fiber.

Now we integrate against $\phi$ and first keep $n$ fixed while letting $r\to\infty$.  As observed in \cite[\S3.6.7]{Yuan:Zhang:quasiprojective}, the metric of a strongly nef metrized line bundle is locally psh-approachable, and the approximating metrics converge locally uniformly.  Thus the corresponding products of curvature currents converge weakly; see also \cite[Proposition 8.7.2]{CLD}.  On the right-hand side, the base measure is fixed and has finite support, so we may pass to the limit in each of the finitely many fiber integrals.  Writing
	\[F(y)=\int_{A_y}\phi\,d\mu_{y,v},\]
we obtain, for each fixed $n$,
	\[\int_{A_v^{\an}}\phi\,c_1(\Lbar_1)\wedge\cdots\wedge c_1(\Lbar_m)
		\wedge\bigwedge_{j=1}^{b}c_1(p^*\Nbar_{j,n})=\sum_y\lambda_{y,n}F(y).\]
Finally, we let $n\to\infty$.  The same weak convergence gives
\[\int_{A_v^{\an}}\phi\,c_1(\Lbar_1)\wedge\cdots\wedge c_1(\Lbar_m)\wedge\bigwedge_{j=1}^{b}c_1(p^*\Nbar_{j,n})\rightarrow\int_{A_v^{\an}}\phi\,d\mu_{A,v}.\]
On the other hand, $F$ is continuous by Lemma \ref{fiber continuity}, and its support is contained in the compact set $p(\operatorname{supp}\phi)$.  Since $\mu_{B,n}$ converges weakly to $\mu_{B,v}$, we also have
\[\sum_y\lambda_{y,n}F(y)=\int_{B_v^{\an}}F\,d\mu_{B,n}\rightarrow\int_{B_v^{\an}}F\,d\mu_{B,v},\]
completing the proof.
\end{proof}

\begin{proposition} \label{measure equality}
Let $\Mbar_1,\ldots,\Mbar_m\in\Pichat(A/k)_{\nef}$ be strongly nef.  If
\[
c_1(\Mbar_1)_v\wedge\cdots\wedge c_1(\Mbar_m)_v\wedge p^*\mu_{B,v}
=
\mu_{A,v},
\]
then
\[
c_1(\Mbar_1|_{A_y})_v\wedge\cdots\wedge c_1(\Mbar_m|_{A_y})_v
=
\mu_{y,v}
\]
for every $y\in\supp\mu_{B,v}$.
\end{proposition}

\begin{proof}
Suppose there is some $b_0 \in \supp \mu_{B,v}$ at which the measures differ on $A_{b_0}$.  Then there exists a continuous function $\phi_{b_0}$ on $A_{b_0}$ so that 
	\[ \int_{A_{b_0}} \phi_{b_0} \, c_1(\Mbar_1)_v \cdots c_1(\Mbar_m)_v > \int_{A_{b_0}} \phi_{b_0} \, c_1(\Lbar_1|_{A_{b_0}})_v\wedge \cdots \wedge c_1(\Lbar_m|_{A_{b_0}})_v. \]
Let $\phi$ be any continuous extension of $\phi_{b_0}$ to $A_v^\an$.  By Lemma \ref{fiber continuity}, the function 
	\[ D(b) = \int_{A_b} \phi \, c_1(\Mbar_1)_v \cdots c_1(\Mbar_m)_v - \int_{A_b} \phi \, c_1(\Lbar_1|_{A_b})_v\wedge \cdots \wedge c_1(\Lbar_m|_{A_b})_v \]
is continuous and positive in a neighborhood of $b_0$.  Replacing $\phi$ with the product $\phi \cdot (h\circ p)$ for a continuous bump function $h$ on $B_v^\an$ supported near $b_0$, we may assume that $D\geq 0$ on all of $B_v^{\an}$, with $D>0$ near $b_0$, so that 
	\[ \int_{B_v^\an} D \, \mu_{B,v} > 0. \]   
But 
\[ \int_{B_v^\an} D \, \mu_{B,v} = \int_{A_v^\an} \phi \, c_1(\Mbar_1)_v \cdots c_1(\Mbar_m)_v \wedge p^*\mu_{B,v} - \int_{A_v^\an} \phi \, c_1(\Lbar_1)_v \cdots c_1(\Lbar_m)_v \wedge p^*\mu_{B,v} \]
by Proposition \ref{Fubini}, which contradicts the given equality of the measures on $A_v^{\an}$.
\end{proof}

\subsection{A gluing lemma}
The following consequence of the projection formula will be used at every stage of the induction argument in the proof of Theorem \ref{hypersurfaces}.  

\begin{proposition} \label{gluing}\label{prop:gluing positivity}
Let $p:A\to B$ be flat and projective of relative dimension $m$.  Let $\Lbar_A\in\Pichat(A/k)_{\nef}$ and $\Lbar_B\in\Pichat(B/k)_{\nef}$.  Assume that $\Lbar_B$ is big and that
	\[ a=(\widetilde L_A|_{A_{\bar\eta}})^m>0 \]
on the geometric generic fiber of $p$.  If $d=d_k(B)$ is the dimension over $k$, then
	\[ \Lbar_A^m\cdot(p^*\Lbar_B)^d =a\,\Lbar_B^d>0.  \]
In particular, $\Lbar_A+p^*\Lbar_B$ is big on $A$.
\end{proposition}

\begin{proof}
By the relative projection formula for Deligne pairings,
\[\left\langle\Lbar_A^m,p^*\Lbar_B\right\rangle_{A/B}=a\, \Lbar_B.\]
Intersecting with $\Lbar_B^{d-1}$ gives the displayed identity.  Since $\Lbar_B$ is nef and big, $\Lbar_B^d>0$ by \eqref{adelic volume formula}.  Finally, all terms in the expansion of
\[(\Lbar_A+p^*\Lbar_B)^{d+m}\]
are nonnegative, and the term with $m$ copies of $\Lbar_A$ is positive.
\end{proof}

\bigskip
\section{Canonical bundles for product dynamical systems}

In this section, we construct the nef adelic line bundle $\Lbar_\Phi$ associated to an algebraic family of product maps $\Phi$ on $(\bP^1)^N$, allowing the factors to have different degrees, and we discuss some of their basic properties.

\subsection{The invariant adelic line bundle}
Let $S$ be an integral quasiprojective variety over $k$ in one of the three settings of \S\ref{section: prelims}.  After replacing $S$ by an affine open subset, let
	\[ f:S\times\bP^1\rightarrow S\times\bP^1 \]
be an algebraic family of rational maps of degree $d>1$.  Write $L$ for the pullback of $O_{\bP^1}(1)$ to $S \times\bP^1$.

\begin{theorem} \label{Lf}
There is a unique nef adelic line bundle
	\[ \Lbar_f\in\Pichat(S\times\bP^1/k) \]
with underlying line bundle $L$ and satisfying
	\[ f^*\Lbar_f\simeq d\Lbar_f. \]
In the complex setting, its curvature current is the fibered Green current $\hat T_f$.
\end{theorem}

\begin{proof}
In the arithmetic and function-field settings this is \cite[Theorems 6.1.1 and 6.1.2]{Yuan:Zhang:quasiprojective}.  The construction is by the normalized pullback of any initial semipositive metric.  The same limit over $\C$ produces the usual dynamical Green current; see, for example, \cite[Proposition 2.7]{GauthierVigny2025}.  Uniqueness follows from the functional equation and the contraction by $d^{-1}f^*$ on continuous metrics.
\end{proof}

When $K$ is either a number field or a function field, and for each point $s \in S(\Kbar)$, the restriction of $\Lbar_f$ to $\{s\}\times\bP^1$ is the canonical metrized line bundle for $f_s$ (possibly up to a normalization).  Consequently,
	\[ h_{\Lbar_f}(s,z)=\hat h_{f_s}(z). \]
Over $\C$, note that the current $\hat T_f$ satisfies
\begin{equation} \label{square zero}
	\hat T_f\wedge\hat T_f=0.
\end{equation}
This follows directly from the dynamical construction of the Green metric; see, for example, \cite{GauthierVigny2025}.

\subsection{Products of maps on $\bP^1$}
Let
	\[ \Phi=(f_1,\ldots,f_N):S\times(\bP^1)^N\rightarrow S\times(\bP^1)^N \]
be a family of product maps, with $d_i=\deg f_i>1$.  For the projection
	\[ \pi_i:S\times(\bP^1)^N\rightarrow S\times\bP^1 \]
to the $i$-th coordinate, set
\begin{equation} \label{LPhi}
	\Lbar_\Phi = \bigotimes_{i=1}^N\pi_i^*\Lbar_{f_i} \in \Pichat(S\times(\bP^1)^N/k).
\end{equation}
This line bundle is nef, even though $\Phi$ need not be polarized when the degrees $d_i$ are different.

When $K$ is either a number field or a function field, the associated height is the sum of the coordinatewise canonical heights:
\begin{equation} \label{product height sum}
	h_{\Lbar_\Phi}(s,z_1,\ldots,z_N) = \sum_{i=1}^N\hat h_{(f_i)_s}(z_i).
\end{equation}
In particular, this height is nonnegative and vanishes at the $\Phi_s$-preperiodic points.  

Write $T_i=\pi_i^*\hat T_{f_i}$.  Then
	\[ \hat T_\Phi=\sum_{i=1}^N T_i. \]
By \eqref{square zero}, every nonzero term in a wedge product contains each coordinate current at most once.  In particular,
\begin{equation} \label{product equilibrium measure}
	\hat T_\Phi^{\wedge N} = N!\; T_1\wedge\cdots\wedge T_N,
\end{equation}
which restricts on every fiber to a positive multiple of the product of the maximal-entropy measures of the maps $(f_i)_s$.

The construction is invariant under coordinatewise conjugacy.  More precisely, if $A=(A_1,\ldots,A_N)$ is an algebraic family of products of M\"obius transformations and
	\[ \Phi^A=A\circ\Phi\circ A^{-1}, \]
then $A^*\Lbar_{\Phi^A}\simeq\Lbar_\Phi$.  Therefore all heights, intersection numbers, and volume computations below are unchanged by the conjugacies used to trivialize an isotrivial family.

\subsection{Generic fibers over rational functions} \label{generic fibers}
Our inductive argument in Section \ref{ordinary hypersurfaces} below requires that we repeatedly pass to a generic level set of a rational function on $S$.  We record the construction and the two properties that will be used.

Let $k$ be a field of characteristic $0$, let $S$ be an integral quasiprojective $k$-variety of dimension $\ell\geq1$, and let $\sigma\in k(S)$ be transcendental over $k$.  On the largest open subset $U_\sigma\subset S$ where $\sigma$ is regular, it defines a dominant morphism
	\[\sigma:U_\sigma\rightarrow\bP^1_k.\]
Put
	\[ K_\sigma=k(\sigma) \qquad\hbox{and}\qquad S_\sigma=U_\sigma\times_{\bP^1_k}\Spec K_\sigma. \]
Then $S_\sigma$ is integral over $K_\sigma$,
	\[ K_\sigma(S_\sigma)=k(S), \qquad \dim_{K_\sigma}S_\sigma=\ell-1.  \]
For a family $(\Phi,\cX)$ over $S$, we write $(\Phi_\sigma,\cX_\sigma)$ for its base change to $S_\sigma$.

\begin{proposition} \label{natural identification}
After varying the base from $k$ to $K_\sigma$, the geometric part $\widetilde{L}_\Phi$ of $\Lbar_\Phi$  is canonically identified with $\Lbar_{\Phi_\sigma}$.  The same is true for every fiber power of $\Phi$ and after restriction to fiber powers of $\cX_\sigma$.  In particular,
	\[ \volhat\left(\widetilde L_{\Phi^{(a)}}|_{\cX^{(a)}}\right) 
	= \volhat\left(\Lbar_{\Phi_\sigma^{(a)}}|_{\cX_\sigma^{(a)}}\right) \]
for all $a\geq 1$.
\end{proposition}

\begin{proof}
This is completely formal from Yuan-Zhang \cite{Yuan:Zhang:quasiprojective}.  The varying-the-base construction is functorial, and the resulting line bundle satisfies the same coordinatewise functional equations as the canonical line bundle for $\Phi_\sigma$.  The identification follows from Theorem \ref{Lf}.  Compatibility of intersection numbers is part of the direct-limit construction of the adelic Picard group; see \cite[\S2.5.5]{Yuan:Zhang:quasiprojective}.
\end{proof}

We will also need the following lemma about maximal variation under this process of generic slicing:  

\begin{lemma} \label{generic maximally varying}
If a pair $(\Phi,\cX)$ is maximally varying over $S$, then $(\Phi_\sigma,\cX_\sigma)$ is maximally varying over $S_\sigma$.  The conclusion remains true after a finite extension of $K_\sigma$ and passage to an irreducible component of the base change.
\end{lemma}

\begin{proof}
Let $\rho_{\Phi, \cX} :S\dashrightarrow\mathscr M$ be the moduli map defined in \S\ref{moduli map} and put $Z=\overline{\rho(S)}$.  Through pullback we have an inclusion of function fields
	\[ k(Z)\subset k(S), \]
and maximal variation means that $k(S)/k(Z)$ is finite.  The function field of the image of the induced moduli map on $S_\sigma$ is the compositum $k(Z)K_\sigma$ inside $k(S)$.  Since
	\[ k(Z)\subset k(Z)K_\sigma\subset k(S), \]
the extension $k(S)/k(Z)K_\sigma$ is finite.  This is precisely generic finiteness of the sliced moduli map.  Finite extension of the ground field and passage to an irreducible component preserve this property.
\end{proof}

\bigskip
\section{Bifurcations for essential hypersurfaces, without Latt\`es components}
\label{ordinary hypersurfaces}

In this section we prove Theorem \ref{hypersurfaces} under the assumption that none of the maps $f_i$ is a Latt\`es family.  Let $S$ be an irreducible quasiprojective variety of dimension
	\[ \ell=\dim S>0, \]
let
	\[ \Phi=(f_1,\ldots,f_N):S\times(\bP^1)^N\rightarrow S\times(\bP^1)^N \]
be an algebraic family of product maps, and let $\cH\subset S\times(\bP^1)^N$ be an irreducible and flat family of essential hypersurfaces.  

By Proposition \ref{easy stability}, we know that if $\cH$ is $\Phi$-special then the pair $(\Phi, \cH)$ has vanishing bifurcation current.  In this section, we prove a converse statement in the following form, with the bifurcation measure instead of the current.

\begin{theorem} \label{non-Lattes hypersurface}
Suppose that $\Phi=(f_1,\ldots,f_N):S\times(\bP^1)^N\rightarrow S\times(\bP^1)^N$ is an algebraic family of product maps, with $\deg f_i > 1$ for each $i$, defined over $\C$.  Let $\cH\subset S\times(\bP^1)^N$ be an irreducible hypersurface defined over $\C$, forming a flat family of essential hypersurfaces over $S$.  Assume that 
\begin{enumerate}
\item[(i)]  no $f_i$ is a Latt\`es family; 
\item[(ii)] the pair $(\Phi,\cH)$ is maximally varying; and 
\item[(iii)] $\cH$ is not $\Phi$-special.
\end{enumerate}
Then 
	\[ T_{\Phi^{(\ell)}}^{\wedge N\ell}\wedge[\cH^{(\ell)}]\neq 0 \]
so that the bifurcation measure $\mu_{\Phi, \cH}$ is nonzero on $S$.
\end{theorem}

\subsection{Dimension reduction by multiplier functions} \label{dim red}
Choose a maximal collection of algebraically independent functions
	\[  \sigma_1,\ldots,\sigma_m \in \overline{\C(S)}  \qquad 0\leq m\leq\ell, \]
among all multipliers of periodic points of the factors $f_j$.  After a generically finite base change, we may mark the associated periodic points and view each $\sigma_i$ as a regular function on $S$.   We construct a tower of fields
	\[ \C=K_0\subset K_1\subset\cdots\subset K_m, \qquad K_i=K_{i-1}(\sigma_i), \]
and varieties $S_i$ as follows.  Put $S_0=S$.  Inductively, $S_i$ is the generic fiber of
	\[ \sigma_i:S_{i-1}\dashrightarrow\bP^1_{K_{i-1}}, \]
viewed as a variety over $K_i$; the irreducibility of $S$ implies that $S_i$ is integral over $K_i$.  Let $(\Phi_i,\cH_i)$ be the induced pair over $S_i$.  Then
	\[ \dim_{K_i}S_i=\ell-i, \]
and $(\Phi_i,\cH_i)$ is maximally varying by Lemma \ref{generic maximally varying}.

At the terminal stage $m$, the multipliers of all periodic points of the factors $f_j$ are constant on $S_m$ (meaning that they lie in the algebraic closure of $K_m$).  Choose an embedding of $\Kbar_m$ into $\C$.  McMullen's rigidity theorem then implies that every factor of $\Phi_m$ is either isotrivial or Latt\`es \cite{McMullen:families}.  By hypothesis, none of the factors are Latt\`es, so all factors are isotrivial.  After a further finite base change, finite extension of the field $K_m$, and coordinatewise conjugacy, $\Phi_m$ is constant over $S_m$, and it can be viewed as one map over the field $K_m$, independent of the parameter in $S_m$.

\subsection{Positive Lyapunov exponents} \label{pos Lyap}
For $i\geq1$, let $v_i$ be the non-archimedean place of $K_i=K_{i-1}(\sigma_i)$ at which $|\sigma_i|_{v_i}>1$.  In other words, we let $v_i$ represent the point at $\infty$ of $\bP^1$ in the image of $\sigma_i: S_{i-1} \dashrightarrow \bP^1$.  The periodic cycle with multiplier $\sigma_i$ is therefore repelling at $v_i$, at all parameters of $S_i$.  In particular, the corresponding factor of $\Phi_i$ has a positive non-archimedean Lyapunov exponent at $v_i$, for all parameters in $S_i$; see \cite[Corollaire D]{FavreRiveraLetelier2026}.

\subsection{The terminal geometric volume}
Set $q_i :=\ell-i = \dim_{K_i}S_i$.  We first prove that
\begin{equation} \label{base volume}
\volhat\left(
\widetilde L_{\Phi_m^{(q_m)}}|_{\cH_m^{(q_m)}}
\right)>0,
\end{equation}
with the convention that the left side is $1$ when $q_m=0$.

Assume $q_m>0$.  Since $\Phi_m$ is constant over $S_m$, the canonical line bundle associated to $\Phi_m^{(q_m)}$ is pulled back from $((\bP^1)^N)^{q_m}$ over $K_i$.  Maximal variation of $\cH_m$ implies that the evaluation map
	\[ \cH_m^{(q_m)}\rightarrow((\bP^1)^N)^{q_m} \]
is generically finite and so dominant; see for example \cite[Proposition 2.1]{DMY:UniformDAO}.  The projection formula of \cite[Proposition 4.1.2]{Yuan:Zhang:quasiprojective} therefore gives \eqref{base volume}, similar to the proof of Theorem \ref{isotrivial case} over $\C$.  Note that the coordinate changes used to trivialize $\Phi_m$ do not affect this computation, by the conjugacy invariance following \eqref{product equilibrium measure}.

\subsection{The induction step}
We descend from $i=m$ to $i=0$.  Fix $1\leq i\leq m$, and assume
\begin{equation} \label{vol i}
	\volhat\left( \widetilde L_{\Phi_i^{(q_i)}}|_{\cH_i^{(q_i)}} \right)>0,
\end{equation}
with $q_i =\ell-i$.  For $q_i=0$, this is our convention at the terminal stage.

For $q_i>0$, the projection
	\[ p:\cH_i^{(q_i+1)}\rightarrow\cH_i^{(q_i)} \]
has fibers isomorphic to the projective hypersurfaces $(H_i)_s$ for $s \in S_i$.  The restriction of the canonical geometric line bundle to every such fiber is ample.  Proposition \ref{gluing} therefore implies
\begin{equation} \label{geometric volume +1}
	\volhat\left( \widetilde L_{\Phi_i^{(q_i+1)}}|_{\cH_i^{(q_i+1)}} \right)>0.
\end{equation}
whenever $q_i>0$.  For $q_i=0$, so that $i = \ell = m$, the hypersurface $\cH_m$ is projective, so the geometric volume $\volhat\left( \widetilde L_{\Phi_i^{(1)}}|_{\cH_i^{(1)}} \right)$ is positive.

We now regard $K_i$ as a one-variable function field over $K_{i-1}$, so that $S_i$ and $\cH_i$ are essentially quasiprojective schemes over $K_{i-1}$.  Let
	\[ \Mbar = \Lbar_{\Phi_i^{(q_i+1)}}|_{\cH_i^{(q_i+1)}} \in \Pichat(\cH_i^{(q_i+1)}/K_{i-1})_{\nef}. \]
Its geometric part is big by \eqref{geometric volume +1}.  We claim that
\begin{equation} \label{arithmetic induction volume}
	\volhat(\Mbar)>0.
\end{equation}

Suppose to the contrary that $\volhat(\Mbar) = 0$.  Since $\Mbar$ is nef, its height on the total variety is $0$, and Theorem \ref{fundamental inequality} gives
	\[ e(\cH_i^{(q_i+1)},\Mbar)=0. \]
Thus there is a generic sequence
	\[ P_n\in\cH_i^{(q_i+1)}(\overline{K_i}) \]
with $h_{\Mbar}(P_n)\to0$.

For $q_i>0$ and for each $j\in\{1,\ldots,N\}$, let
	\[ p_j:\cH_i^{(q_i+1)} \rightarrow \cH_i^{(q_i)}\times_{S_i} (S_i \times (\bP^1)^{N-1}) \]
be the projection that forgets the $j$-th coordinate in the last copy of $\cH_i$.  For $q_i = 0$, take the corresponding projection 
		\[ p_j:\cH_i \rightarrow S_i \times (\bP^1)^{N-1}. \]
Essentiality of the hypersurface $\cH_i$ implies that $p_j$ is dominant and generically finite.  We use $p_j$ to pull back the canonical line bundle for the product dynamical system consisting of the first $q_i$ complete copies of $\Phi_i$ and all factors except $f_j$ in the last copy; denote the result by 
	\[ \Mbar_j \in \Pichat(\cH_i^{(q_i+1)}/K_{i-1}).\]  
Proposition \ref{gluing} and \eqref{vol i} imply that the geometric part of every $\Mbar_j$ is big.  By \eqref{product height sum},
	\[ 0\leq h_{\Mbar_j}(P_n)\leq h_{\Mbar}(P_n), \]
so $P_n$ is small for every $\Mbar_j$.  

Recall that we defined the place $v_i$ of the field $K_i$ in \S\ref{pos Lyap}.  Applying Theorem \ref{equidistribution} at the place $v_i$, we conclude that the corresponding measures on $\cH_i^{(q_i+1)}$ are proportional:
\begin{equation} \label{muj proportional}
	\mu_{j,v_i}\propto\mu_{j_0,v_i} \qquad\text{for all }j,
\end{equation}
where $j_0$ is an index for which the marked cycle with multiplier $\sigma_i$ belongs to $f_{j_0}$.

Fix $j\neq j_0$ and forget both the $j$-th and $j_0$-th coordinates in the last copy of $\cH_i$.  This gives a flat morphism, after restriction to a dense open set,
\begin{equation} \label{proj 2}
  	p_{j,j_0}:\cH_i^{(q_i+1)} \rightarrow 
		B_{j,j_0} :=\cH_i^{(q_i)}\times_{S_i}(S_i \times (\bP^1)^{N-2}). 
\end{equation}
Its fibers define curves $C_y \subset \bP^1\times\bP^1$ in the $(j,j_0)$ coordinates, for each $y \in B_{j, j_0}$.  Proposition \ref{measure equality}, applied to \eqref{muj proportional}, gives
\begin{equation} \label{on X_y}
	\pi_j^*\mu_{(f_j)_y} \propto \pi_{j_0}^*\mu_{(f_{j_0})_y}
\end{equation}
on the Berkovich analytification $C_y^{\an}$ for every $y$ in the support of the canonical measure on $B_{j,j_0}^{\an}$.

We use two general facts about these measures to interpret the relation of \eqref{on X_y}.

\begin{lemma} \label{scalar extension measures}
Equality, or proportionality, of the measures in \eqref{on X_y} is preserved after extension from the completed residue field $\cH(y)$ to a complete algebraically closed extension.
\end{lemma}

\begin{proof}
Set $F=\cH(y)$, $\bC_F=\widehat{\overline F}$, and $X=C_y$.  Let $\nu_1$ and $\nu_2$ be the two measures on $X_{\bC_F}^{\an}$, rescaling one of them in the case of proportionality. These measures are $G$-invariant for $G=\operatorname{Gal}(\overline F/F)$ and have equal pushforwards under the natural map $q:X_{\bC_F}^{\an}\to X_F^{\an}$, by compatibility of canonical metrics and Chambert--Loir measures with scalar extension; see \cite[Corollary 18.3.3(a)]{CLD}.  Since $q$ identifies $X_F^{\an}$ with $X_{\bC_F}^{\an}/G$ \cite[Corollary 1.3.6]{Berkovich}, averaging continuous test functions over $G$ with respect to normalized Haar measure gives $\nu_1=\nu_2$.
\end{proof}

\begin{lemma} \label{supports}
Suppose that $X$ is a smooth and irreducible quasiprojective algebraic variety defined over $K$.  Suppose that $\Lbar_1, \ldots, \Lbar_n \in \Pichat(X/k)$ are strongly nef adelic line bundles.  Let 
	\[ \mu_v = c_1(\Lbar_1) \wedge \cdots \wedge c_1(\Lbar_n) \]
be the associated Chambert-Loir measure on $X_v^\an$ at an analytification of $X$ at a place $v$ of $K$.  Then for every Zariski-closed subset $Y\subsetneq X$ defined over $\Kbar$, one has
	\[ \mu(Y^\an_v)=0. \]
In particular, if $\mu\neq 0$, then $\supp \mu$ is not contained in any countable union of proper algebraic subvarieties of $X$.
\end{lemma}

\begin{proof}
The metrics are locally psh-approachable by \cite[\S3.6.7]{Yuan:Zhang:quasiprojective}.  Chambert-Loir and Ducros have shown that the top-degree measure does not charge algebraic subsets of smaller dimension; see \cite[Corollaire 16.9.4]{CLD}. 
\end{proof}

Because of Lemma \ref{scalar extension measures}, the measure proportionality \eqref{on X_y} holds on $C_y$ base changed to the algebraic closure of $\cH(y)$.   We are able to apply the non-archimedean rigidity theorem of Ji--Yap \cite[Theorem 1.1]{JiYap2026} to \eqref{on X_y} on each curve $C_y$, for $y$ in the support of the canonical measure on $B_{j,j_0}$.  By construction, the map $(f_{j_0})_y$ has positive Lyapunov exponent at $v_i$, so the rigidity theorem tells us that $C_y$ is dynamically special for the pair $((f_j)_y,(f_{j_0})_y)$.  But note that special relations are determined by countably many algebraic conditions.  Lemma \ref{supports} therefore implies that one of these conditions holds identically on $B_{j,j_0}$.  
For $N=2$, this says that $\cH_i$ itself is $\Phi_i$-special.  

For $N>2$, this implies that an equation for $\cH_i$ depends only on the coordinates $j$ and $j_0$, contradicting essentiality. In detail, we view $B_{j,j_0}$ as a parameter space for the family of pairs $((f_j, f_{j_0}), \{C_y\}_{y\in B_{j,j_0}})$ in $\bP^1\times\bP^1$ over the field $K_i$.  We will conclude that this family of curves is persistently special over $B_{j,j_0}$.  Let $S_B$ be the projection of $B_{j,j_0}$ to $S_i$.  Fixing any $s \in S_B(\Kbar_i)$ while letting the other parameters vary (over a Zariski open subset of $(\cH_i)_s^{\ell-i} \times  (\bP^1)^{N-2}$), we see that $\{C_y\}$ must be a family of curves depending only on $s$, because a single map $(f_j, f_{j_0})_s$ cannot have preperiodic curves existing in nontrivially varying families.  In other words, the original hypersurface $(\cH_i)_s$ in $\{s\} \times (\bP^1)^N$ must project to this curve $C_s$ in $(\bP^1\times\bP^1)$, but then the family $\cH$ cannot be essential.  

In either case we have contradicted our assumptions on $\cH$.  This proves \eqref{arithmetic induction volume}.

By Proposition \ref{natural identification}, \eqref{arithmetic induction volume} is exactly
	\[ \volhat\left( \widetilde L_{\Phi_{i-1}^{(q_i+1)}}|_{\cH_{i-1}^{(q_i+1)}} \right) = \volhat\left( \widetilde L_{\Phi_{i-1}^{(\ell - i + 1)}}|_{\cH_{i-1}^{(\ell - i + 1)}} \right) >0. \]
The induction therefore gives
	\[ \volhat\left( \widetilde L_{\Phi^{(\ell)}}|_{\cH^{(\ell)}} \right)>0. \]
By \eqref{geometric volume formula}, this is equivalent to $T_{\Phi^{(\ell)}}^{\wedge N\ell}\wedge[\cH^{(\ell)}]\neq 0$, and Theorem \ref{non-Lattes hypersurface} is proved.

\bigskip
\section{Latt\`es factors and mixed families}
\label{lattes section}

In this section, we complete the proof of Theorem \ref{hypersurfaces}.  After Theorem \ref{non-Lattes hypersurface}, it only remains to treat families of maps $\Phi$ for which some factor $f_j$ is Latt\`es.  When all of the maps $f_1, \ldots, f_N$ are Latt\`es, we appeal to Gao's non-degeneracy characterization \cite{GaoBettiRank, GaoBettiRankCorrigendum} (see also Habegger's earlier non-degeneracy arguments for powers of a non-isotrivial elliptic scheme over a curve \cite[\S3]{Habegger:special}) for our Theorem \ref{lattes nondegeneracy}.  For the mixed cases, where some $f_i$ is Latt\`es and another $f_j$ is not, we use complex-analytic arguments, similar to our proofs in \cite{DM:commonprep}, and we rely on the Dynamical Manin-Mumford theorem for curves in $\bP^1\times\bP^1$ over $\C$ from \cite{Ghioca:Nguyen:Ye:DMM1}.  

Throughout this section we work over $\C$.  Let $S$ be an irreducible quasiprojective variety of positive dimension.

\subsection{The Betti form and dynamical Green currents}  
Each elliptic curve $E$ over $\C$ is diffeomorphic to the real torus $\R^2/\Z^2$.  Suppose that $\cal E \to S$ is a family of elliptic curves over $S$.  For a simply-connected (analytic) open subset $U$ of $S$, the {\bf Betti map} is defined by choosing a real-analytic trivialization of the bundle $\cE_U = p^{-1}(U) \iso U \times (\R^2/\Z^2)$, compatible with the group structure on each fiber, and projecting to the torus component: 
	\[ \beta_U:  \cE_U \to \R^2/\Z^2. \]
It is known that $\beta_U$ has complex-analytic fibers in $\cE_U$.  These fibers define the {\bf leaves of the Betti foliation} of $\cE$; see, for example, \cite{ACZ:Betti}, \cite{CDMZ}, and \cite{UU2} for more information.   For any choices of trivializations $\beta_U$, the local forms $\beta_U^*(dx\wedge dy)$ glue to a positive closed $(1,1)$-form $\omega_{\Betti}$ on $\cE$.  It restricts to normalized Haar measure on every fiber and vanishes on the Betti leaves.  In particular, 
\begin{equation} \label{Betti square zero}
	\omega_{\Betti}\wedge\omega_{\Betti}=0.
\end{equation}

Suppose that
	\[ f:S\times\bP^1\rightarrow S\times\bP^1 \]
is a family of Latt\`es maps.  After shrinking $S$ and making a generically finite base change, there is an elliptic scheme
	\[ p_f:\cE^f\rightarrow S, \]
a fiberwise affine isogeny $\varphi_f:\cE^f\to\cE^f$, and a finite morphism
	\[ \psi_f:\cE^f\rightarrow S\times\bP^1 \]
such that $f\circ\psi_f=\psi_f\circ\varphi_f$. The dynamical Green current of the Latt\`es family is the push-forward of the Betti form as a current:
\begin{equation} \label{Betti Green}
	\hat T_f= (\psi_f)_*\omega_{\Betti}.
\end{equation}
Indeed, the two sides restrict to the equilibrium measure of $f_s$ on each fiber, and they satisfy the same invariance relation.  

On $S \times \bP^1$, recall that the Julia sets for the dynamically stable family $f$ are moving holomorphically over all of $S$ \cite{Mane:Sad:Sullivan}.  The leaves of the Betti foliation project to the leaves of this holomorphic motion, which define the {\bf $J$-stable foliation} in $S\times \bP^1$.

We will use the following form of what is called the ``theorem of the fixed part." 

\begin{proposition} \label{leaf algebraicity}
Let $\cE\to S$ be a non-isotrivial elliptic scheme, and let $\sigma:S\to\cE$ be an algebraic section.  If the Betti coordinates of $\sigma$ are locally constant, then $\sigma$ is torsion.  In particular, a leaf of the Betti foliation of $\cE$ is an algebraic subvariety of $\cE$ if and only if it is a torsion leaf.  If $\cE$ is isotrivial, then $\sigma^*\omega_{\Betti}=0$ if and only if $\sigma$ is constant after a finite base change and a trivialization of $\cE$.
\end{proposition}

\begin{proof}
The first assertion is the relative Manin-kernel theorem for an elliptic scheme \cite{Manin1963}; see also \cite[\S1.3(i)]{ACZ:Betti}.  A dynamical proof follows from \cite[Theorem 1.1]{D:stableheight}, upon projecting to an associated family of Latt\`es maps on $\bP^1$.   The general statement about algebraic leaves follows by replacing $S$ with a finite branched cover.  The assertion in the isotrivial case follows immediately from the product description of the Betti map.
\end{proof}

\subsection{Weakly special hypersurfaces}
\label{weakly special quotient}
Suppose now that
	\[ \Phi=(f_1,\ldots,f_N):S\times(\bP^1)^N \rightarrow S\times(\bP^1)^N  \]
is a family of products of Latt\`es maps.  After a generically finite base change, the map $\Phi$ lifts to the abelian scheme
\begin{equation} \label{abelian scheme product}
	\cA=\cE^{f_1}\times_S\cdots\times_S\cE^{f_N}; 
\end{equation}
let $\psi:\cA\to S\times(\bP^1)^N$ be the quotient map.

Suppose that $\cH\subset S\times(\bP^1)^N$ is a weakly special essential hypersurface.  Recall this means that $\cH$ is the quotient of a family of cosets for the group structure on the fibers of $\cA$, and essentiality implies that each projection of $\cH$ forgetting a single factor $\bP^1$ is dominant.  An irreducible component of its lift in $\cA$ can be written
	\[ \cB+\sigma, \]
where $\cB\subset\cA$ is an abelian subscheme of codimension one and $\sigma$ is a section of $\cA \to S$.  The quotient
	\[ q:\cA\rightarrow\cQ:=\cA/\cB  \]
is an elliptic scheme, and the image of $\cB+\sigma$ is a section we denote by
\begin{equation} \label{weak section}
	\bar\sigma_{\cH}:=q\circ\sigma:S\rightarrow\cQ.
\end{equation}
Note that the hypersurface $\cH$ is $\Phi$-special if and only if $\bar\sigma_{\cH}$ is a torsion section.  If $\cQ$ is isotrivial, then we say the section $\bar\sigma_{\cH}$ is {\bf constant} if, after base change so that the elliptic scheme is a product, the resulting section is the graph of a constant function.

\begin{proposition} \label{weakly special bif}
Suppose that $\Phi :S\times(\bP^1)^N \rightarrow S\times(\bP^1)^N$ is a product of Latt\`es maps, and suppose that $\cH$ is a weakly special and irreducible essential hypersurface in $S \times (\bP^1)^N$.  There is a constant $a_{\cH}>0$ such that
	\[	T_{\Phi,\cH}  = a_{\cH}\;\bar\sigma_{\cH}^*\omega_{\cQ} \]
on $S$, where $\omega_{\cQ}$ is the Betti form on $\cQ$.  Consequently,
	\[  T_{\Phi,\cH}\wedge T_{\Phi,\cH}=0. \]
Moreover, $T_{\Phi,\cH}=0$ if and only if either $\cH$ is $\Phi$-special, or $\cQ$ is isotrivial (and therefore so is $\Phi$) and $\bar\sigma_{\cH}$ is constant.
\end{proposition}

\begin{proof}
Write $\omega_i$ for the Betti form on $\cE^{f_i}$ and also for its pullback to $\cA$.  From \eqref{Betti Green}, the pullback of $\hat T_\Phi$ to $\cA$ is, up to positive constants, the sum
	\[ \omega_1+\cdots+\omega_N. \]
Since $\cH$ is essential, the homomorphism from every factor $\cE^{f_i}$ to $\cQ$ induced by $q$ is nonzero, and hence is an isogeny.  It follows that each $\omega_i$ is a positive multiple of the pullback of $\omega_{\cQ}$ to the corresponding factor.  Using \eqref{Betti square zero}, expanding the $N$-th wedge power of $\hat T_\Phi$, and applying the projection formula to $\cB+\sigma$, we obtain
	\[ T_{\Phi,\cH} = \pi_*\left(\hat T_\Phi^{\wedge N}\wedge[\cH]\right)  = a_{\cH}\,\bar\sigma_{\cH}^*\omega_{\cQ}  \]
for some positive constant $a_{\cH}$.  The square $T_{\Phi, \cH} \wedge T_{\Phi, \cH}$ vanishes by \eqref{Betti square zero}.   We also deduce from $T_{\Phi,\cH} = a_{\cH}\,\bar\sigma_{\cH}^*\omega_{\cQ}$ that $T_{\Phi, \cH} = 0$ whenever $\cH$ is $\Phi$-special or if $\cQ$ is isotrivial and the section $\bar \sigma_{\cH}$ is constant. 

If $T_{\Phi, \cH} = 0$, then the restriction of a Betti map $\beta_U$ for $\cQ$ to the image of $\bar\sigma_{\cH}$ must be constant over any analytic $U$ where we can define the Betti map.  From Proposition \ref{leaf algebraicity}, this implies that either $\cQ$ is isotrivial or $\bar\sigma_{\cH}$ is a torsion section.  

Finally, we observe that if $\cQ$ is isotrivial, then each of the factors $\cE^f$ must be isotrivial, implying that $\Phi$ is itself isotrivial.  
\end{proof}

\subsection{Nondegeneracy characterization when all maps are Latt\`es} 
When all of the maps $f_i$ are Latt\`es maps, we rely on Gao's non-degeneracy characterization for abelian schemes \cite{GaoBettiRank, GaoBettiRankCorrigendum} (see also Habegger's \cite[\S3]{Habegger:special}) to prove the implications in Theorem \ref{hypersurfaces}.  More precisely:

\begin{theorem} \label{lattes nondegeneracy}
Suppose that $S$ is a quasiprojective variety over $\C$ of dimension $\ell \geq 1$, $\Phi : S\times(\bP^1)^N \rightarrow S\times(\bP^1)^N$ is a product of Latt\`es maps, and suppose that $\cH$ is an essential hypersurface in $S \times (\bP^1)^N$.  Assume, in addition, that $(\Phi, \cH)$ is maximally varying and $\cH$ is not $\Phi$-special.  Then
	\[  \hat T_{\Phi^{(\ell)}}^{\wedge N\ell} \wedge[\cH^{(\ell)}]\neq0 \]
unless $\cH$ is weakly special and $\ell = \dim S > 1$.
\end{theorem}

\begin{proof}
As in \S\ref{weakly special quotient}, we make a finite base change so that we can lift $\cH$ to a subvariety of an abelian scheme $\cA \to S$.  As in the proof of Proposition \ref{weakly special bif}, the current $T_{\Phi^{(\ell)}}^{\wedge N\ell}$ coincides, up to a positive constant and finite pushforward, with the appropriate top exterior power of the Betti form on the corresponding fiber power of the lift.  Write $\cX\subset \cA$ for the lift of $\cH$, and suppose that $\cH$ is not
weakly special. We aim to use Gao's fiber-power theorem \cite[Theorem 10.1(i)]{GaoBettiRank}, and to this end we will ensure that its hypotheses hold for a subscheme induced by $\cX$ on an appropriate quotient of $\cA$. Let $\mathcal{C}$ be the identity component of the stabilizer of the generic fiber of $\cX$. After further shrinking and base change, we form the quotient
	\[q:\cA\rightarrow\cQ:=\cA/\mathcal{C},\]
so that $\cX=q^{-1}(\cY)$ for $\cY=q(\cX)\subset \cQ$.  Write $g$ for the relative dimension of
$\cQ$.  Since $\cX$ is not a coset, $g\geq2$, and the fibers of $\cY$ are irreducible divisors. Moreover, they have finite stabilizer, so that $\cY$ is not invariant under translation by any positive-dimensional abelian subscheme. In particular, conditions (a), (b), and (c) of Gao's theorem hold for $\cY$.

We first observe that the map from $S$ to a moduli space of pairs of abelian varieties and subvarieties, associated to the pairs $(\cQ, \cY)$, is generically finite.  Indeed, if not, then we can restrict to a positive-dimensional component of a general fiber where the pair \((\mathcal Q,\cY)\) is constant.  Since \(\mathcal H\) is essential, every elliptic factor maps nontrivially to \(\mathcal Q\), and hence is isotrivial on this component.  After a finite base change, \(q\) is constant as well. Thus \((\Phi,\mathcal H)\) is constant up to coordinatewise conjugacy on this component, contradicting maximal variation.

Now let $\iota^{(\ell)}$ denote the induced moduli map for the fiber power of $\cQ^{(\ell)}$ to a universal abelian variety $\mathcal{U}_{g\ell}$, as in \cite{GaoBettiRank}.  We will now observe that its restriction to $\cY^{(\ell)}$ is generically finite.  Indeed, after shrinking $S$, consider the locus in $\cY^{(\ell)} \times_{\mathcal{U}_{g\ell}} \cY^{(\ell)}$ of parameters \(s,t\in S\) for which \(Q_s\) and \(Q_t\) are isomorphic as polarized abelian varieties, together with \(\ell\) points belonging to both \(Y_s\) and \(Y_t\) under such an identification.  If the two divisors coincide, generic finiteness of the moduli map of pairs shows that
this locus has dimension at most
	\[ \ell+\ell (g-1)= \ell g. \]
If they are distinct, their intersection has dimension at most \(g-2\), so the corresponding locus has dimension at most
	\[ 2\ell +\ell (g-2)=\ell g. \]
Since \(\dim\mathcal Y^{(\ell)}=\ell g\), the fiber product of \(\mathcal Y^{(\ell)}\) with itself over its modular image has dimension at most \(\ell g\). Thus \(\iota^{(\ell)}|_{\mathcal Y^{(\ell)}}\) is generically finite.

Gao's fiber-power theorem \cite[Theorem 10.1(i)]{GaoBettiRank}, now gives 
	\[\operatorname{rank}_{\R} \left(d\beta_{\cQ}^{(\ell)}|_{Y^{(\ell)}}\right)=2g\ell.\]
Since $\cX=q^{-1}(\cY)$, the Betti map on $\cX^{(\ell)}$ also has the full rank $2\ell(N-g)$ in the
directions of $\mathcal{C}^{(\ell)}$.  Its rank is therefore $2g\ell+2\ell(N-g)=2N\ell$, proving the required nonvanishing.

From Proposition \ref{weakly special bif}, the maximal variation implies that a non-special weakly special $\cH$ will have $T_{\Phi, \cH} \not=0$, but $(T_{\Phi, \cH})^{\wedge 2} = 0$.  This completes the proof of the theorem.  
\end{proof}

\subsection{Mixed Latt\`es and non-Latt\`es factors}
Suppose now that
	\[ \Phi=(f_1,\ldots,f_N):S\times(\bP^1)^N \rightarrow S\times(\bP^1)^N  \]
is a family of products where at least one $f_i$ is a Latt\`es family and at least one $f_j$ is not.  

The key idea in this case is to use the construction of rigid repellers, similar to our arguments in \cite{DM:commonprep}, and the analytic characterization of positivity in Proposition
\ref{analytic characterization}.  

One step in the argument relies on the following result about families of maps on $\bP^1$, following Baker's works \cite{Baker:ff,Baker2006} and proved explicitly in \cite{Mavraki:Schmidt}.  (An alternative proof of Proposition \ref{bounded degree} follows from the complex-analytic arguments in \cite{D:stableheight}.)  

\begin{prop} \label{bounded degree}  \cite[Corollary 4.9]{Mavraki:Schmidt}
Let $K$ be the function field of a complex curve, and let $f\in K(z)$ be non-isotrivial of degree at least $2$.  For every integer $D>0$, the set
	\[  \{P\in\Preper(f):[K(P):K]\leq D\}  \]
is finite.
\end{prop}

\begin{theorem} \label{mixed prerepelling}
Suppose that $S$ is a quasiprojective variety over $\C$ of dimension $\ell \geq 1$, $\Phi= (f_1, \ldots, f_N) : S\times(\bP^1)^N \rightarrow S\times(\bP^1)^N$ is an algebraic family, and $\cH$ is an essential hypersurface in $S \times (\bP^1)^N$.  Assume there exist $i \not= j$ so that $f_i$ is a Latt\`es family and $f_j$ is not.  Assume also that $(\Phi,\cH)$ is maximally varying.  Then every nonempty analytic open subset of $S$ contains a parameter over which
$\cH^{(\ell)}$ has a rigid repeller for the pair $(\Phi^{(\ell)}, \cH^{(\ell)})$.  Consequently,
	\[ \hat T_{\Phi^{(\ell)}}^{\wedge N\ell} \wedge[\cH^{(\ell)}]\neq0.  \]
In fact, the bifurcation measure $\mu_{\Phi, \cX}$ on $S$ has $\supp \mu_{\Phi, \cX} = S(\C)$.
\end{theorem}

\begin{proof}
Let us first assume that $N=2$, so $\Phi = (f_1, f_2)$ with $f_1$ Latt\`es.  By replacing $S$ with a Zariski open subset, we can assume that $f_2$ is never Latt\`es over $S$.  Suppose that $\cal C \subset S \times (\bP^1\times\bP^1)$ is a family of essential curves over $S$ and that the pair $(\Phi, \cal C)$ is maximally varying.  

We aim to show that this set-up always leads to positive bifurcation measure $\mu_{\Phi, \cX}$ with $\supp \mu_{\Phi, \cX} = S(\C)$.  

We first assume that $f_1$ is non-isotrivial over $S$.  Fix any parameter $s_0 \in S$, and choose a small analytic open neighborhood $U$ of $s_0$ and $V \subset \bP^1$ for which $J(f_{2,s}) \cap V \neq \emptyset$ for all $s \in U$ and so that some local branch of $\cal C$ is the graph of conformal $h_s: V \to W_s$ for all $s \in U$.  We can use $h_s$ to build proper intersections between repelling cycles of $\Phi$ and $\cal C$ over $U$.  

Recall that the Latt\`es family $f_1$ is stable, and so the holomorphic motion of its Julia set  defines a foliation $\mathcal{F}$ of $S \times \bP^1$.  A given repelling periodic point of $f_{2, s_0}$ in $V$ can be followed holomorphically over a small neighborhood of $s_0$.  If, via the local correspondence $h_s$, this point coincides with a leaf of $\mathcal{F}$, then it extends by analytic continuation to coincide with a leaf globally over $S$; the leaf of $\mathcal{F}$ is therefore an algebraic subvariety of $S \times \bP^1$, so it must be a preperiodic leaf for $f_1$, by the non-isotriviality of $f_1$ and Proposition \ref{leaf algebraicity}.  If this holds for infinitely many repelling cycles for $f_{2, s_0}$ in the open set $V$, then $\cal C_{s_0}$ contains infinitely many preperiodic points of $\Phi_{s_0}$.  The equidistribution argument in \cite[\S6, proof of Theorem~1.3]{Ghioca:Nguyen:Ye:DMM1} then implies that $f_{2, s_0}$ is also a Latt\`es map.  This contradicts our assumption on $f_2$.  So there must be a proper intersection between a repelling periodic point for $\Phi$ and $\cal C$ over some parameter $s_1 \in U$.   

Now suppose that $f_1$ is isotrivial.  If $f_2$ is also isotrivial, then we can perform a base change and coordinate change on $\bP^1\times\bP^1$ over $S$ so that $\Phi$ will be independent of $s \in S$.  As the pair $(\Phi,\cal C)$ is maximally varying over $S$, the curve $\cal C$ must be nontrivially varying in this coordinate system.  If we cannot find a proper intersection between a repelling cycle of $\Phi$ and the curve family $\cal C$ over the neighborhood $U$ of $s_0$, then the curves must also be independent of $s\in S$, a contradiction.  

Now suppose that $f_1$ is isotrivial but $f_2$ is not.  If, via the local correspondence $h_s$ defined above, a repelling periodic point for $f_2$ coincides with a leaf of $\mathcal{F}$, then it extends by analytic continuation to coincide with a leaf globally over $S$.  The isotriviality of $f_1$ implies that all of the leaves of the foliation are algebraic, and they project with bounded degree to the base $S$.  If all of the preperiodic points for $f_2$ in that neighborhood $V$ correspond to leaves of this foliation, then this would contradict Proposition \ref{bounded degree} (upon restricting to a 1-parameter family along which $f_2$ is non-isotrivial). 

In each of these three cases, we have found a proper intersection between a marked repelling cycle of $\Phi$ and the curve family $\cal C$ over a parameter $s_1 \in U$.  
For $\dim S = 1$, this gives positivity of the current $\hat T_{\Phi}^{\wedge 2} \wedge[\cal C]$ from Proposition \ref{analytic characterization} and completes the proof in the case $N=2$ and $\dim S =1$.

Now assume $\ell := \dim S > 1$, continuing to assume $N=2$.  We will proceed inductively, as in \cite[proof of Proposition~8.5]{DM:commonprep}, to build a rigid repeller at some parameter in $U$ for the fiber power $\cal C^{(\ell)}$.  

Write $S_0 = S$, and let $S_1\subset S_0$ be the codimension 1 subvariety containing $s_1$ along which the preperiodic point for $\Phi$ is persistently in $\cal C$.  Shrink the neighborhood $U$ if necessary so the cycle it maps to remains repelling throughout $U \cap S_1$.  We repeat the arguments above to find a parameter $s_2 \in U \cap S_1$ where there is another proper intersection between a pre-repelling point of $\Phi$ and the curve family $\cal C$.  At the final stage of the induction, we will have marked $\ell$ preperiodic points of $\Phi$ over a parameter $s_\ell \in S$ contained in the original open set $U \subset S$, that defines a proper intersection with $\cal C^{(\ell)}$ in $S \times (\bP^1\times\bP^1)^\ell$, and consequently 
	\[ \hat T_{\Phi^{(\ell)}}^{\wedge 2\ell} \wedge[\cal C^{(\ell)}]\neq0 \]
in $U \times (\bP^1\times\bP^1)^\ell$.  This completes the proof in the case $N=2$.  

For the case of general $N$, we choose a Latt\`es index $i$ and a non-Latt\`es index $j$.  As above, we replace $S$ with a Zariski-open subset over which $f_j$ is never Latt\`es.  Fix a small analytic $U \subset S$, and mark repelling periodic points over $U$ for each of the maps $\not= f_i$ or $f_j$.  Let 
	\[ \cal C_{ij} \subset U \times (\bP^1\times\bP^1) \]
be the induced family of curves where we evaluate the polynomials defining $\cH$ at these marked points, and pass to an irreducible component.  Consider the dynamics of $\Phi_{ij} = (f_i, f_j)$.

\begin{lemma}  \label{curve restriction}
For some choice of $i, j$ and some choice of marked points, the pair $(\Phi_{ij}, \cal C_{ij})$ is non-isotrivial over $U$.  
\end{lemma}

\begin{proof}
If either $f_i$ or $f_j$ is not isotrivial, then there is nothing to show.  So we may assume that all maps $f_1, \ldots, f_N$ are isotrivial over $S$.  Changing coordinates on $(\bP^1)^N$ so that the maps are independent of $s \in S$, we have a family of polynomials defining $H_s$ by $\{F_s = 0\}$ that depend nontrivially on $s$.  The Zariski density of the repelling periodic points for a map $f_k$ implies that we can find a repelling periodic point $c_k$ where a specialization $F_s(\cdots, c_k, \cdots)$ continues to depend nontrivially on $s$ for a constant $c_k$. Proceed inductively.
\end{proof}

To complete the proof of Theorem \ref{mixed prerepelling}, we apply the arguments from the $N=2$ case to build a proper intersection between a repelling point of $\Phi_{ij}$ and $\cal C_{ij}$ over $U$.  We then proceed inductively to construct subvarieties $S_k \subset S$ of codimension $k$, and each time choose a new $i,j$ (if necessary) and new repelling points (if necessary) so that Lemma \ref{curve restriction} holds, so that we can build new intersections until we have constructed a rigid repeller in $\cH^{(\ell)}$.  The proof is completed with Proposition \ref{analytic characterization}.
\end{proof}

\subsection{Proof of Theorem \ref{hypersurfaces}}
If $\cH$ is $\Phi$-special, then Proposition \ref{easy stability} implies the vanishing of the current $T_{\Phi, \cH}$ on $S$ and therefore also the measure $\mu_{\Phi, \cH}$.  

If not all maps $f_i$ are Latt\`es, then Theorems \ref{non-Lattes hypersurface} and \ref{mixed prerepelling} establish the equivalences of Theorem \ref{hypersurfaces}.  If all maps are Latt\`es, we obtain the non-degeneracy characterization from Theorem \ref{lattes nondegeneracy} and Proposition \ref{weakly special bif}.  
\qed

\bigskip
\section{Arithmetic bigness for essential hypersurfaces}
\label{S:arithmetic bigness}

We now assume that all data are defined over $\Qbar$.  In this section, we prove Theorem \ref{main3} in the case where $\cX$ is an essential hypersurface.   To this end, we combine Theorem \ref{hypersurfaces} with the arithmetic bigness theorem of Yuan \cite[Theorem 4.1]{Yuan:bigness} and the measure-rigidity theorem of Mavraki--Schmidt \cite[Theorem 1.10, Theorem 3.1]{Mavraki:Schmidt} (which itself extends a result of Levin--Przytycki \cite{Levin:Przytycki}).

\subsection{Weakly special hypersurfaces}
Let
	\[	F=(f_1,\ldots,f_N):(\bP^1)^N\rightarrow(\bP^1)^N \]
be a product map, and let $H\subset(\bP^1)^N$ be an irreducible essential hypersurface.  Recall that $H$ is weakly $F$-special if either $H$ is $F$-special, or all of the maps $f_i$ are exceptional and $H$ is the quotient of a group coset.  In the latter case, since $H$ is an essential hypersurface, it follows that all $f_i$ have a common group type (all elliptic and from isogenous elliptic curvves, or all multiplicative) and there are quotient maps
	\[ \Psi_i:G\rightarrow\bP^1, \]
a codimension-one algebraic subgroup $B\subset G^N$, and a point $\sigma\in G^N$ such that
	\[ H=\overline{(\Psi_1,\ldots,\Psi_N)(B+\sigma)}. \]
We see that $H$ is $F$-special if and only if the class
	\[ \bar\sigma=\sigma+B\in G^N/B \]
is torsion.  

Let $\Phi$ and $\cH$ be families over an irreducible base $S$.  We use the same terminology when the geometric generic fiber has the stated property.  

\begin{lemma} \label{strictly weakly special torsion}
Assume that $\dim S>0$ and that $(\Phi,\cH)$ is maximally varying.  If $\cH$ is weakly $\Phi$-special, then 
	\[\{s \in S:  H_s \mbox{ is } \Phi_s\mbox{-special}\} \]
is Zariski dense in $S$.
\end{lemma}

\begin{proof}
If $\cH$ is $\Phi$-special, then the conclusion clearly holds, so we may assume that $\cH$ is weakly $\Phi$-special but not $\Phi$-special.

After shrinking and making a generically finite base change, write the exceptional lift as
	\[ \cB+\sigma\subset\cG^N, \qquad \bar\sigma:S\rightarrow\cQ:=\cG^N/\cB. \]
In the multiplicative case, the maps are isotrivial over $S$, so maximal variation means that the hypersurface $\cH$ must be varying. This implies that the image of $\bar\sigma$ in $\cQ \simeq \mathbb{G}_m$ is non-constant.  The conclusion follows.  

In the elliptic case, if $\cQ$ is a non-isotrivial family of elliptic curves, Proposition \ref{leaf algebraicity} implies that $\bar\sigma$ cannot omit torsion values, and the conclusion follows.  When $\cQ$ is isotrivial, the maximal variation of $(\Phi, \cH)$ implies that $\bar\sigma$ is non-constant, and again we have a Zariski-dense set of parameters where $H_s$ is special.
\end{proof}

\subsection{Arithmetic rigidity for hypersurfaces}
The following statement is the arithmetic analog of a step in the inductive proof from Section \ref{ordinary hypersurfaces}.  Here we use a complex-analytic slicing argument for currents, following the proofs of \cite[Theorem 1.8]{Mavraki:Schmidt} and \cite[Theorem 5.1]{DM:commonprep}.

\begin{proposition} \label{arithmetic rigidity hypersurface}
Assume that $(\Phi,\cH)$ is maximally varying and that
	\[ \widetilde L_{\Phi^{(\ell+1)}}|_{\cH^{(\ell+1)}} \]
is big, where $\ell=\dim S>0$.  If
	\[ \volhat\left( \cH^{(\ell+1)}, \Lbar_{\Phi^{(\ell+1)}}|_{\cH^{(\ell+1)}} \right)=0, \]
then $\cH$ is weakly $\Phi$-special.
\end{proposition}

\begin{proof}
Put $m=\ell+1$ and
	$\Mbar=\Lbar_{\Phi^{(m)}}|_{\cH^{(m)}}. $
The given hypotheses tell us that the geometric part of $\Mbar$ is big, and $\Mbar$ is nef.  The
fundamental inequality therefore gives a generic sequence $P_n\in\cH^{(m)}(\Qbar)$ with
	\[ h_{\Mbar}(P_n)\rightarrow0,\text{ as }n\to\infty.\]

As in the proof of \eqref{arithmetic induction volume}, forget one coordinate in the last copy of $\cH$ in $\cH^{(m)}$.  Essentiality makes each of the resulting projections generically finite, and Proposition \ref{gluing} shows that the corresponding adelic line bundles remain geometrically big.  Since their heights are nonnegative and bounded above by $h_{\Mbar}$, the same sequence is small for all of them. Equidistribution at the archimedean place gives proportionality of the resulting measures over $\C$, as in \cite{Mavraki:Schmidt}.

Now forget two coordinates and consider the projections as we did in \eqref{proj 2}.  We form a family of curves in $\bP^1\times\bP^1$, where we consider the action of two of the maps $(f_i, f_j)$.  For a generic such curve, we deduce from slicing the resulting measures that the pullbacks of the two maximal-entropy measures are proportional, giving an archimedean version of the proportionality in \eqref{on X_y}, exactly as in \cite{Mavraki:Schmidt} and \cite{DM:commonprep}.  More precisely, this measure relation holds on all fibers of the projection over the support of the measure in base.  For each such curve, the measure rigidity theorem of Mavraki--Schmidt \cite[Theorem 3.1]{Mavraki:Schmidt} implies that this curve is weakly special for appropriate iterates of the corresponding pair of maps.  This holds for any pair $(i,j)$, so either all of the maps are exceptional and arise as quotients of the same group $G$ and the curves are weakly special, or all of these curves are dynamically special.  Moreover, the support of the measure in the base does not lie in a countable union of hypersurfaces, and so the algebraic relations must hold persistently for the entire curve family.  In particular, the curve family is weakly special on the generic slice.  We finish the proof as in the conclusion fo the proof of Theorem \ref{non-Lattes hypersurface}.  That is, for $N=2$, we see immediately that $\cH$ is weakly $\Phi$-special.  

Suppose that $N>2$ and that the maps are not all exceptional. Holding the parameter $s$ in the base constant implies that all of the curves for the slices over $s$ are preperiodic for the corresponding slices of pairs of iterates of $f_i$ or that the curve family is also constant (as they are all preperiodic for a single map on $\bP^1\times\bP^1$). Thus the defining equation of $\cH$ does not involve all coordinates.  The latter contradicts essentiality. 

It remains to treat the exceptional case. We work over the geometric generic fiber of $S$ and choose compatible group covers $\Psi_i:G\to\bP^1$. Let $X\subset G^N$ be an irreducible component of the lift of $H$. Fixing all coordinates except $i$ and $N$ and specializing at general values gives curves whose images are the weakly special curves obtained above. By the exceptional case of the measure-rigidity argument, their irreducible components are translates of one-dimensional algebraic subgroups of $G^2$. For each $i<N$, there are only countably many possible subgroups underlying these curves. For any fixed such subgroup $D$,
viewed in coordinates $i$ and $N$ of $G^N$, the condition $x+D\subset X$ is Zariski closed in $x\in X$. These closed sets cover a nonempty Zariski open subset of $X$, so one
of them must be all of $X$. Thus there is a one-dimensional subgroup $B_i$, supported in coordinates $i$ and $N$, such that translation by $B_i$ preserves $X$.  Essentiality implies that $B_i$ projects surjectively to both coordinates $i$ and $N$. The subgroup $B$ generated by $B_1,\ldots,B_{N-1}$ therefore projects surjectively to the first $N-1$ coordinates. Since translation by $B$ preserves $X$, and $\dim X=N-1$, we conclude that $\dim B=N-1$ and that $X$ is a translate of $B$. Consequently, $\cH$ is weakly $\Phi$-special.
\end{proof}

\subsection{The arithmetic bigness criterion}
Let $\cH\subset S\times(\bP^1)^N$ be a flat family of essential hypersurfaces defined over $\Qbar$.  We write
\[ \overline{\mathcal M}_{\Phi,\cH} = \left\langle
		\Lbar_\Phi|_{\cH},\ldots,\Lbar_\Phi|_{\cH} \right\rangle_{\cH/S} \in\Pichat(S/\bZ),\]
where the Deligne pairing has $N$ entries.  Its curvature at an archimedean place is the bifurcation current $T_{\Phi,\cH}$, up to the normalization
of $\Lbar_\Phi$.  Its fiber height is
\begin{equation} \label{hypersurface Deligne height}
 h_{\overline{\mathcal M}_{\Phi,\cH}}(s) =N\deg_{L_{\Phi_s}}(H_s)\,\hat h_{\Phi_s}(H_s),
\end{equation}
where (abusing the notation slightly) we write $\hat{h}_{\Phi_s}:= h_{\overline{L}_{\Phi_s}}$ even in the case that the component maps have different degrees.
This follows from base-change compatibility of Deligne pairings and the
definition of the height of a subvariety; see \cite[Theorem 4.1.3 and \S5.3.1]{Yuan:Zhang:quasiprojective}.

Recall the definition of the moduli space map $\rho_{\Phi, \cH}$ from \S\ref{moduli map}. 

\begin{theorem} \label{arithmetic bigness: essential case}
Let $S$ be an irreducible quasiprojective variety over $\Qbar$ with $\ell=\dim S>0$.  Let $\Phi$ be a family of product maps over $S$, and let $\cH\subset S\times(\bP^1)^N$ be a flat family of irreducible essential hypersurfaces.  The following are equivalent.
\begin{enumerate}
\item $\volhat\left(\Lbar_{\Phi^{(\ell+1)}}|_{\cH^{(\ell+1)}} \right)=0$.
\item $\volhat\left(S,\overline{\mathcal M}_{\Phi,\cH}\right)=0$.
\item Either $\dim\overline{\rho_{\Phi,\cH}(S)}<\ell$, or $\cH$ is weakly $\Phi$-special.
\end{enumerate}
\end{theorem}

\begin{proof}
The equivalence of (1) and (2) is the fiber-power criterion for
arithmetic bigness of the Deligne pairing
\cite[Theorem 4.1]{Yuan:bigness}.  We prove the equivalence with (3).

Assume first that the pair is not maximally varying.  Put
	\[ q=\dim\overline{\rho_{\Phi,\cH}(S)}<\ell. \]
After shrinking $S$ and making a generically finite base change, the
pair is, up to coordinatewise conjugacy, the pullback of a pair
$(\Phi_T,\cH_T)$ over a variety $T$ of dimension $q$.  Put $m=\ell+1$.
The line bundle on $\cH^{(m)}$ is pulled back from $\cH_T^{(m)}$, whose
absolute dimension over $\bZ$ is
	\[ q+m(N-1)+1<Nm. \]
The top arithmetic intersection therefore vanishes.  The projection formula \eqref{proj formula} gives (1).

Suppose next that $\cH$ is weakly special and the pair $(\Phi, \cH)$ is maximally varying.  From Lemma \ref{strictly weakly special torsion}, we know that there is a Zariski-dense set of parameters $s$ so that $H_s$ is $\Phi_s$-special.  Therefore, $H_s \cap\Preper(\Phi_s)$ is 
Zariski dense in $H_s$ for these parameters.  Their $m$-fold fiber powers give a Zariski-dense set of points of height zero in $\cH^{(m)}$.  
A big adelic line bundle has height bounded below by a strictly positive constant away from a proper closed subset \cite[Theorem 5.3.7]{Yuan:Zhang:quasiprojective}.  Thus the line bundle $\Lbar_{\Phi^{(\ell+1)}}|_{\cH^{(\ell+1)}}$ cannot be big, and (1) follows.

Conversely, assume (1) and maximal variation.  If $\cH$ is not weakly $\Phi$-special, then Theorem \ref{hypersurfaces}, followed by Proposition \ref{gluing}, shows that the geometric part of
	\[ \Lbar_{\Phi^{(\ell+1)}}|_{\cH^{(\ell+1)}} \]
is big. But this contradicts Proposition \ref{arithmetic rigidity hypersurface}, thus proving (3).
\end{proof}

We can now prove Theorem \ref{main1} in the case of essential hypersurfaces, over $\Qbar$.

\begin{corollary} \label{dense special hypersurfaces}
In the setting of Theorem \ref{arithmetic bigness: essential case}, the following are equivalent.
\begin{enumerate}
\item The set of $s\in S(\Qbar)$ for which $H_s$ is $\Phi_s$-special is
Zariski dense.
\item The set $\Preper(\Phi^{(\ell+1)})\cap\cH^{(\ell+1)}$ is Zariski dense in $\cH^{(\ell+1)}$.
\item Either $\cH$ is $\Phi$-special, or $\cH$ is weakly $\Phi$-special and $\dim\overline{\rho_{\Phi,\cH}(S)}>0.$
\end{enumerate}
\end{corollary}

\begin{proof}
The set of preperiodic points is Zariski-dense in any special subvariety, so the implication from (1) to (2) is clear.  Condition (3) implies (1) with Lemma \ref{strictly weakly special torsion}. 
Finally assume (2). Put $q:=\dim\overline{\rho_{\Phi,\cH}(S)}$. After shrinking $S$ and making a generically finite base change, the pair is, up to coordinatewise conjugacy, the pullback of a maximally varying pair $(\Phi_T,\cH_T)$ over a variety $T$ of dimension $q$. Descent to $T$ and projection to $q+1$ copies give a Zariski-dense set $\Preper(\Phi_T^{(q+1)})\cap\cH_T^{(q+1)}$ in $\cH_T^{(q+1)}$.  If $q=0$, the fixed hypersurface $H_T$ contains a Zariski-dense set of preperiodic points and is therefore $\Phi_T$-special, as recalled in \S\ref{relations}. Consequently, $\cH$ is $\Phi$-special.  If $q>0$, the adelic line bundle $\Lbar_{\Phi_T^{(q+1)}}|_{\cH_T^{(q+1)}}$ has a Zariski-dense set of height-zero points, so (as in the proof of Theorem \ref{arithmetic bigness: essential case} using \cite[Theorem 5.3.7(1)]{Yuan:Zhang:quasiprojective}) we infer that its arithmetic volume is zero. Since $(\Phi_T,\cH_T)$ is maximally varying, Theorem \ref{arithmetic bigness: essential case} implies that $\cH_T$, and hence $\cH$, is weakly special. This proves (3) in every case.
\end{proof}

\bigskip
\section{True degeneracy}
\label{ess decomposition}

Let $S$ be an irreducible quasiprojective variety over $\C$, and let
\[ \Phi=(f_1,\ldots,f_N):S\times(\bP^1)^N \rightarrow S\times(\bP^1)^N \]
be a product family.  In this section, $\cX\subset S\times(\bP^1)^N$ is irreducible and flat over $S$, its geometric generic fiber $X$ over $S$ is geometrically integral of dimension
$r$, where
	\[ 1\leq r\leq N-1, \]
and $X$ dominates every coordinate $\bP^1$.  In this section, we introduce the essential hypersurface decomposition of $X$ and the activity dimension of the pair $(\Phi, \cX)$.  We prove Theorem \ref{true deg char}, our characterization of true degeneracy. 

\subsection{The essential hypersurface shadows}
\label{decomposition}
For $J\subset\{1,\ldots,N\}$, let
	\[ Y_J=\overline{\pi_J(X)}\subset(\bP^1)^{J} \]
for the projection $\pi_J: (\bP^1)^N \to (\bP^1)^J$ to the coordinates indexed by $J$.  Define
\begin{equation} \label{J(X)}
	\mathcal J(X)=\{J\subset\{1,\ldots,N\}:|J|=r+1\mbox{ and }\dim Y_J=r\}.
\end{equation}
This collection $\cal J(X)$ is nonempty.  Indeed, among the coordinate functions on $X$ one can choose $r$ which form a transcendence basis for $\C(X)/\C$.  Adding any further coordinate gives an element of $\mathcal J(X)$.

For $J\in\mathcal J(X)$, the variety $Y_J$ is a hypersurface in $(\bP^1)^J$, but its equation need not involve every coordinate.  Let $E_J\subset J$ be the set of variables which occur with positive degree in an irreducible multihomogeneous equation for $Y_J$.  The same equation defines an irreducible essential hypersurface
	\[ H_J\subset(\bP^1)^{E_J}, \]
and
\begin{equation} \label{shadow product}
	Y_J=(\pi_{E_J}^{J})^{-1}(H_J) =H_J\times(\bP^1)^{J\setminus E_J}
\end{equation}
for the projection $\pi_{E_J}^J: (\bP^1)^J \to (\bP^1)^{E_J}$.  We call $(E_J,H_J)$ an \textbf{essential hypersurface shadow} of $X$.
The indexed collection
	\[ \mathscr H(X)=\{(E_J,H_J):J\in\mathcal J(X)\} \]
is the essential hypersurface decomposition of $X$.  

After shrinking $S$, the sets $E_J$ are constant, every $H_J$ spreads
out to a flat family
	\[ \cH_J\subset S\times(\bP^1)^{E_J}, \]
and the maps $\cX\to\cY_J$ are dominant and generically finite.  We
make this restriction without further comment.

\begin{proposition} \label{essential intersection}
The variety $X$ is an irreducible component of
\[ Z(X)=\bigcap_{J\in\mathcal J(X)}
       \pi_{E_J}^{-1}(H_J) \;  \subset \; (\bP^1)^N. \]
Moreover, $X$ is $\Phi$-special if and only if every shadow $H_J$ is $\Phi_{E_J}$-special.
\end{proposition}

\begin{proof}
Choose a set $R \subset\{1, \ldots, N\}$ of $r$ coordinates for which $X\to(\bP^1)^R$ is dominant and generically finite.  For every $i\notin R$, the set $R\cup\{i\}$ lies in $\mathcal J(X)$.  Over a dense open subset of $(\bP^1)^R$, each $Y_{R\cup\{i\}}\to(\bP^1)^R$ has finite fibers.  Their fiber product over $(\bP^1)^R$ therefore has dimension $r$, and it contains $Z(X)$ which in turn contains the $r$-dimensional $X$.  This proves the first assertion.

For the second statement, it is immediate that the essential shadows of a special variety are special.  

Conversely, choose a nonempty open $U\subset(\bP^1)^R$ over which all projections $Y_{R\cup\{i\}}\to(\bP^1)^R$ and $X\to(\bP^1)^R$ have finite fibers. For each $i\notin R$, specialness gives a split map $\Psi_i$, commuting with a product of iterates of the component maps in $\Phi_{R\cup\{i\}}$, for which $Y_{R\cup\{i\}}$ is preperiodic.
Commuting maps have the same preperiodic points.
Since the aforementioned maps are finite over $U$, every lift
to $Y_{R\cup\{i\}}$ of a $\Phi_R$-preperiodic point in $U$ is $\Phi_{R\cup\{i\}}$-preperiodic. Consequently, the fibers of $X$ over $U\cap\prod_{j\in R}\Preper(f_j)$ consist of $\Phi$-preperiodic points. Their union is Zariski dense in $X$. So $X$ is $\Phi$-special by \cite{Ghioca:Nguyen:Ye:DMM2,MSW2023, Hindry, Ghioca:Tucker:DMM}.
\end{proof}

We will need the following for our proof of Theorem \ref{true deg char}.

\begin{prop} \label{essential product}
Suppose that $X \times Y \subset (\bP^1)^N \times (\bP^1)^M$ is a closed subvariety defined over a field of characteristic 0.   Then each essential shadow of $X \times Y$ is either an essential shadow of $X$ or an essential shadow of $Y$ (appropriately indexed).  
\end{prop}

\begin{proof}
Let $r = \dim X$ and $s = \dim Y$.   We first show that the essential shadows of $X$ and $Y$ arise for $X\times Y$.  Suppose that $(E_J, \cH_J)$ is an essential shadow for $X \subset (\bP^1)^N$, for $J \in \cal J(X)$.  Let $K\subset \{1, \ldots, M\}$ be any subset with $|K|=s$ so that $Y \subset (\bP^1)^M$ projects dominantly to $(\bP^1)^K$.  Let $K+N\subset \{1, 2, \ldots, N+M\}$ denote the shifted set of indices (adding $N$ to each element of $K$).  Then $J \cup (K+N)$ defines a set of indices in $\{1, 2, \ldots, N+M\}$ for which $|J \cup (K+N)| = r+s+1$ and the projection of $X\times Y$ to $(\bP^1)^{J \cup (K+N)}$ has dimension $r+s$.  We see immediately that $(E_J, \cH_J)$ is the associated essential shadow, for index set $E_{J \cup (K+N)} = E_J$.  Similarly for the essential shadows of $Y$.

Now suppose that $L \subset \{1, \ldots, N+M\}$ is any subset with $|L| = r+s+1$.  For $|L \cap \{1, \ldots, N\}| < r$ or $|L \cap \{N+1, \ldots, N+M\}| < s$, the projection of $X \times Y$ to $(\bP^1)^L$ must have dimension $< r+s$, so $L \not\in \cal J(X\times Y)$.  For $|L \cap \{N+1, \ldots, N+M\}| = s$, we have $L \in \cal J(X\times Y)$ if and only if the projection of $Y$ to $(\bP^1)^{L \cap \{N+1, \ldots, N+M\}}$ has dimension $s$ and $L \cap \{1, \ldots, N\}$ is in $\cal J(X)$.  For such an index set, we see that the essential shadow coincides with $(E_J, \cH_J)$ with $J=L\cap\{1,\ldots,N\}$.  Similarly for index sets with $|L\cap \{1, \ldots, N\}|=r$, and the proof is complete.
\end{proof}

\subsection{Decomposition of the bifurcation current}
For a subset $A\subset\{1,\ldots,N\}$, write $\Phi_A$ for the corresponding subproduct of maps $f_i$ with $i \in A$.  The following formula explains why the shadows contain exactly the information relevant to stability.

\begin{proposition} \label{essential stability}
There are positive rational constants $c_J$ such that
	\[	T_{\Phi,\cX}  =\sum_{J\in\mathcal J(X)}  c_J\,T_{\Phi_{E_J},\cH_J}. \]
In particular,
	\[ T_{\Phi,\cX}=0  \quad\Longleftrightarrow\quad T_{\Phi_{E_J},\cH_J}=0 
			\quad\text{for every }J\in\mathcal J(X). \]
\end{proposition}

\begin{proof}
Write $T_i$ for the pullback of $\hat T_{f_i}$ to $S\times(\bP^1)^N$. Since $T_i\wedge T_i=0$, we have
\[
\hat T_\Phi^{\wedge(r+1)}\wedge[\cX]
=(r+1)!\sum_{|J|=r+1}\left(\bigwedge_{j\in J}T_j\right)\wedge[\cX].
\]
Let $\pi_J:S\times(\bP^1)^N\to S\times(\bP^1)^J$ be the coordinate projection. If $\dim Y_J<r$, then $(\pi_J)_*[\cX]=0$, so the corresponding summand has zero pushforward to $S$ by the projection formula. For $J\in\mathcal J(X)$, put $\delta_J=\deg(X\to Y_J)$. The projection formula gives
\[
(\pi_J)_*\left(\left(\bigwedge_{j\in J}T_j\right)\wedge[\cX]\right)
=\frac{\delta_J}{(r+1)!}\,
\hat T_{\Phi_J}^{\wedge(r+1)}\wedge[\cY_J].
\]

Using \eqref{shadow product}, we now integrate over the factors indexed by $J\setminus E_J$. Each coordinate Green measure has mass $1$, so
\[
(\pi_{E_J}^J)_*\left(\hat T_{\Phi_J}^{\wedge(r+1)}\wedge[\cY_J]\right)
=\frac{(r+1)!}{|E_J|!}\,
\hat T_{\Phi_{E_J}}^{\wedge|E_J|}\wedge[\cH_J].
\]
Pushing to $S$ proves the asserted formula, with
\[
c_J=\delta_J\frac{(r+1)!}{|E_J|!}>0.
\]
All terms are positive, which proves the vanishing criterion.
\end{proof}

\subsection{Activity dimension}
\label{active dim}
Recall that $\mathscr H(X)=\{(E_J,H_J):J\in\mathcal J(X)\}$ denotes the essential hypersurface decomposition of $X$.  We ignore the special shadows, since their bifurcation currents vanish.  Among the remaining, non-special essential hypersurfaces, we consider the {\bf elliptic shadows}, 
	\[ \mathcal J^{\mathrm{ell}} =
	\{J:H_J\text{ is strictly weakly special of elliptic type}\} \]
and the {\bf ordinary shadows},
	\[ \mathcal J^{\mathrm{ord}} = \{J:H_J\text{ is non-special and }J\notin\mathcal J^{\mathrm{ell}}\}. \]
Note that a strictly weakly special hypersurface of multiplicative type is placed in the ordinary collection.  

For $J\in\mathcal J^{\mathrm{ord}}$, let
\begin{equation} \label{pair moduli}
	 \rho_J:S\dashrightarrow\mathscr M_J 
\end{equation}
be the moduli map of the pair $(\Phi_{E_J},\cH_J)$, as defined in \S\ref{moduli map}.  For $J\in\mathcal J^{\mathrm{ell}}$, let
	\[ \bar\sigma_J:S\rightarrow\cQ_J \]
be the quotient section from \eqref{weak section}.  On a simply-connected analytic open set $U\subset S$, choose a Betti map and put
\begin{equation} \label{Betti J}
	\beta_J=\beta_{\cQ_J}\circ\bar\sigma_J: U\rightarrow\R^2/\Z^2.  
\end{equation}

We form the {\bf activity map}
\begin{equation} \label{collective activity map}
	\mathscr A_{\Phi,\cX}:U\rightarrow \prod_{J\in\mathcal J^{\mathrm{ord}}}\mathscr M_J
		\times  \prod_{J\in\mathcal J^{\mathrm{ell}}}\R^2/\Z^2
\end{equation}
from all the maps $\rho_J$ and $\beta_J$.  Its generic real rank is even and does not depend on the Betti charts.  We define the {\bf activity dimension} by
\begin{equation} \label{activity dimension definition}
	\actdim(\Phi,\cX) =\frac12\rank_{\R}\mathscr A_{\Phi,\cX}.
\end{equation}
If there are no non-special shadows, we say the activity dimension is $0$.  The evenness of $\rank_{\R}\mathscr A_{\Phi,\cX}$ follows because the kernels of Betti differentials occurring here are complex subspaces; see \cite[\S4.2, equations (4.6)--(4.9)]{ACZ:Betti} and \cite[\S4]{GaoBettiRank}.  If there are no elliptic shadows, note that the activity dimension is the complex dimension of the image of the product of the ordinary moduli maps.

\subsection{Proof of the true-degeneracy criterion}
We are now ready to prove Theorem \ref{true deg char}.  As usual, set $\ell=\dim S>0$.  

We first outline the proof idea here.  If $\cX$ has maximal activity dimension, we will use the activity of the essential hypersurface shadows $(\Phi_{E_J}, \cH_J)$ and Theorem \ref{hypersurfaces} to know that certain wedge powers of their bifurcation currents are positive.  Via the analytic characterization of this positivity (Proposition \ref{analytic characterization}), we can find  index sets 
	\[ J_1, \ldots, J_\ell \in \cal J(X) \]
for which we can build a rigid repeller in 
	\[ \cal H := \cH_{J_1} \times_S \cdots \times_S \cH_{J_\ell} \]
for the fiber product $\Psi := \Phi_{E_{J_1}} \times_S \cdots \times_S \Phi_{E_{J_\ell}}$.  The existence of this rigid repeller implies the positivity of the current 
\begin{equation} \label{H pos}
	(\hat T_\Psi)^{\wedge \left(\sum_i |E_{J_i}|\right)} \wedge [\cH] 
\end{equation}
in $S \times (\bP^1)^{\left(\sum_i |E_{J_i}|\right)}$ by Proposition \ref{analytic characterization}.  From this, we will deduce the positivity of
 	\[ (\hat T_{\Phi^{(\ell)}})^{\wedge \ell(r+1)} \wedge [\cX^{(\ell)}]. \]
For the converse implication, if the activity dimension is smaller than $\ell$, then every parameter lies on a positive-dimensional complex-analytic fiber of the activity map, along which all essential shadows are stable. If the bifurcation measure is nonzero, Proposition \ref{analytic characterization} will be used to show existence of rigid repellers for some coordinate projections of the top fiber power of $\cX$.  On the other hand, stability of the essential hypersurface shadows would force the projected periodic points to persist in the subvarieties in a family, which would contradict rigidity of the intersections.    

\begin{proof}[Proof of Theorem \ref{true deg char}]
Let $r=\dim_S\cX$ be the relative dimension of $\cX$, so that $1\leq r<N$.  After finite surjective base change and shrinking $S$, we may assume that $S$ is smooth, the ordinary moduli maps $\rho_J$ are regular, and all families and projected coordinate images are flat over $S$ with geometrically integral fibers of constant dimension and degree.  These reductions preserve both conditions in the statement of Theorem \ref{true deg char}, by the projection formula and the fact that the currents do not give mass to proper algebraic subsets of $\cX^{(\ell)}$.

Assume first that $\actdim(\Phi,\cX)=\ell$. For each ordinary shadow which is not weakly special, remove from $S$ the Zariski closure of the parameters with special fibers. 
These closures are proper by Corollary \ref{dense special hypersurfaces}, applied after spreading the coefficients out over $\Qbar$.  

Let $\rho^{\mathrm{ord}}$ be the product of the ordinary moduli maps, and put $o=\rank_\C\rho^{\mathrm{ord}}$, the generic value of its rank. If $o=\ell$, put $e=0$ and $S^{\mathrm{ell}}=S$. Suppose instead that $o<\ell$. For $J\in\mathcal J^{\mathrm{ell}}$, let $\bar\sigma_J:S\to\cQ_J$ be the quotient section defined in \S\ref{weakly special quotient}, and choose a small analytic open set $U$ on which its Betti map $\beta_J$ \eqref{Betti J} is defined. Put
\[ e=\tfrac12\rank_\R \prod_{J\in\mathcal J^{\mathrm{ell}}} \beta_J.  \]
Since the Betti differentials have complex kernels, we can choose distinct $J_1,\ldots,J_e$ so that $\beta:=(\beta_{J_1},\ldots,\beta_{J_e})$ has rank $2e$. Shrink $U$ so that $\beta$ has this rank and $(\rho^{\mathrm{ord}},\beta)$ has rank $2\ell$ throughout $U$.

Choose a rational value in the open image $\beta(U)$, and let $S^{\mathrm{ell}}$ be the Zariski closure of a connected component of its fiber. This is an irreducible algebraic subvariety of codimension $e$. After finite base change, its local equations are $\bar\sigma_{J_i}=\tau_i$, where the $\tau_i$ are torsion sections of the corresponding quotient elliptic schemes. Choose torsion lifts of these sections in the abelian scheme \eqref{abelian scheme product}, the product of elliptic schemes for the Latt\`es factors indexed by $E_{J_i}$. Their images in $S\times(\bP^1)^{E_{J_i}}$ give a locally-defined preperiodic section $x^{\mathrm{ell}}$ for
	\[ \Phi^{\mathrm{ell}}=\Phi_{E_{J_1}}\times_S\cdots\times_S\Phi_{E_{J_e}} \]
over $U \cap S^{\mathrm{ell}}$.  Since all factors are Latt\`es, the point $x^{\mathrm{ell}}$ is preperiodic to a repelling cycle along all of $S^{\mathrm{ell}}$ and its factors lie persistently in the hypersurfaces $\cH_{J_i}$.  Note that, by our definition of $e$ and the specialness of each of these shadows along $S^{\mathrm{ell}}$, we can choose this marked preperiodic point $x^{\mathrm{ell}}$ so that it cannot lie in all of these shadows persistently on any larger-dimensional subvariety of $S$.  Moreover, by our assumption that the rank of $(\rho^{\mathrm{ord}},\beta)$ is $2\ell$ on $U$ and the maximality of the activity dimension, we see that $\rho^{\mathrm{ord}}|_{S^{\mathrm{ell}}}$ has rank $\ell-e$.

Set $V_0=S^{\mathrm{ell}}$. We inductively construct subvarieties $V_j\subset V_{j-1}$ as follows. Choose $K_j\in\mathcal J^{\mathrm{ord}}$ with
	\[ \ell_{K_j}=\dim\overline{\rho_{K_j}(V_{j-1})}>0. \]
After shrinking $V_{j-1}$, we may assume that the moduli map $\rho_{K_j}$ is smooth onto its image. Choose an $\ell_{K_j}$-dimensional algebraic subvariety $S_{K_j}\subset V_{j-1}$ transverse to the fibers of $\rho_{K_j}$, so that $\rho_{K_j}|_{S_{K_j}}$ is generically \'etale onto its image. Theorem \ref{hypersurfaces} and Proposition \ref{analytic characterization}
give a Zariski-dense set of rigid repellers for $\cH_{K_j}^{(\ell_{K_j})}$ over $S_{K_j}$. Choose one over a parameter $s_j$ where the moduli map is smooth and the product of the remaining
ordinary moduli maps has maximal rank. The local continuation of the repeller defines a locally-defined algebraic section $\eta_j$.  We claim that the incidence with $\eta_j$ cuts out, locally near $s_j$ in $V_{j-1}$, exactly the fiber of $\rho_{K_j}$ through $s_j$.  Indeed, after finite base change, the pair $(\Phi_{E_{K_j}},\cH_{K_j})$ is pulled back from its moduli image, up to coordinatewise conjugacy.  The marked repelling point and the equations defining the intersections can therefore also be chosen to depend only on the moduli parameters.  Since the intersection is isolated on the transverse subvariety $S_{K_j}$, the locus where it lies in $\cH_{K_j}^{(\ell_{K_j})}$ within $V_{j-1}$ is precisely the moduli fiber.  Let $V_j$ be the irreducible component of this fiber containing $s_j$. Then
	\[ \dim V_j=\dim V_{j-1}-\ell_{K_j}. \]
As the product of the ordinary moduli maps has maximal rank on $V_{j-1}$, its remaining components have rank $\dim V_j$ on $V_j$, so the induction continues. The previously marked repelling points remain persistently in their hypersurfaces along $V_j$, and remain repelling because the corresponding pairs are constant up to conjugacy. We may therefore repeat the construction after removing the relevant singular and exceptional loci.

The construction ends at a parameter $s_0$, and the dimension drops give
	\[e+\sum_{j=1}^m\ell_{K_j}=\ell.\]
If $e=\ell$, no ordinary shadows are needed and we put $m=0$. Put
	\[\cH^{\mathrm{act}}=\cH_{J_1}\times_S\cdots\times_S\cH_{J_e}\times_S
	\cH_{K_1}^{(\ell_{K_1})}\times_S\cdots\times_S\cH_{K_m}^{(\ell_{K_m})},\]
and let $\Phi^{\mathrm{act}}$ be the corresponding product map on $S\times(\bP^1)^n$, where $n=\sum_i|E_{J_i}|+\sum_j\ell_{K_j}|E_{K_j}|$. The periodic images of the chosen preperiodic sections meet the product of the corresponding image families in an isolated point over $s_0$ and this product is flat of codimension $\ell$. Proposition \ref{analytic characterization} therefore gives 
\begin{equation}\label{active quotient positive}
	 \hat T_{\Phi^{\mathrm{act}}}^{\wedge n}\wedge[\cH^{\mathrm{act}}]\ne0. 
\end{equation}

For each $J \in \{J_1, \ldots, J_e, K_1, \ldots, K_m\}$, we set $\cY_J = (\pi_{E_J}^J)^{-1}(\cH_J)$, and let $\cY \subset S \times (\bP^1)^{\ell(r+1)}$ be their fiber product. As in the proof of Proposition \ref{essential stability}, the nonvanishing in \eqref{active quotient positive} and integration over the unused $\bP^1$-coordinates shows that the corresponding canonical measure on $\mathcal Y$ is nonzero in $S\times (\bP^1)^{\ell(r+1)}$.  As $\cX^{(\ell)}\to\mathcal Y$ is generically finite, the projection formula therefore gives
	\[\hat T_{\Phi^{(\ell)}}^{\wedge\ell(r+1)}\wedge[\cX^{(\ell)}]\ne0.\]
This proves that $\cX$ is potentially non-degenerate.

Before proceeding to the converse implication, we point out that we have proved something more than potential non-degeneracy of $\cX$.  Indeed, the family $\cH^{\mathrm{act}}$ constructed above is a fiber product of exactly $\ell$ essential hypersurface shadows, counting repetitions.  Denote these by $\cH_1,\ldots,\cH_\ell$, and put $d_i=|E_i|$, so that $n=\sum_i d_i$.  Expanding the top wedge power of the Green current on $\cH_{\mathrm{act}}$, and using the vanishing of the square of each coordinate Green current, we obtain
	\[\pi_*\left(\hat T_{\Phi^{\mathrm{act}}}^{\wedge n}\wedge[\cH^{\mathrm{act}}]\right)
	= \frac{n!}{d_1!\cdots d_\ell!} T_{\Phi_{E_1},\cH_1}\wedge\cdots\wedge T_{\Phi_{E_\ell},\cH_\ell}. \]
The left side is a nonzero positive measure by \eqref{active quotient positive}.  Consequently, this wedge product of $\ell$ bifurcation currents is also nonzero.  Since Proposition \ref{essential stability} expresses $T_{\Phi,\cX}$ as a positive linear combination of the bifurcation currents of the essential hypersurface shadows, we conclude that
\begin{equation} \label{old bif measure}
	T_{\Phi,\cX}^{\wedge\ell} \not= 0. 
\end{equation}

Now assume that $\actdim(\Phi,\cX) < \ell$. Let $U \subset S$ denote an analytic open subset on which the activity map is defined, or let $U$ be all of $S$ if there are no non-special elliptic shadows. When the Betti components all have torsion values (or if there are no non-special elliptic shadows), a fiber $F$ of the activity map is a positive-dimensional complex-algebraic subvariety in $S\cap U$.  Note that each of the essential hypersurface shadows is stable along these subvarieties $F$ in $S\cap U$ (from Proposition \ref{easy stability}).  This $F$ extends to a positive-dimensional algebraic subvariety of $S$ that we will also denote by $F$. The essential hypersurface shadows remain stable along this algebraic subvariety $F$, because the ordinary components are not varying, while the elliptic components are persistently special. From Proposition \ref{essential product}, we know that the essential hypersurface shadows of the fiber power $\cX^{(\ell)}$ coincide with those of the original variety $\cX$. It follows from Proposition \ref{essential stability} that $\cX^{(\ell)}$ and all of its shadows are stable along $F$. That is, we have
\begin{equation} \label{stable on F}
	T_{\Phi^{(\ell)}|_F,\cX^{(\ell)}|_F} = 0
\end{equation}
on $F$.   In particular, restricting to the normalization of any algebraic curve $C \subset F$, we also have that
\begin{equation} \label{stable proj on C}
	\hat T_{(\Phi^{(\ell)})_J}^{\wedge (r\ell+1)} \wedge [\cY_J|_C] = 0
\end{equation}
for any projection of $\cX^{(\ell)}$ to $\cY_J$ for subsets $J \subset \{1,\ldots,N\ell\}$ with $|J|=r\ell+1$ and $\dim_S\cY_J=r\ell$. In particular, any intersection between an iterated image of $\cY_J|_C$ and a repelling periodic section must persist along the curve $C$, by invariance of the Green currents and Proposition \ref{analytic characterization}.

We now show that $\cX$ must be truly degenerate. From the projection formula, the positivity of $(\hat T_{\Phi^{(\ell)}})^{\ell(r+1)} \wedge [\cX^{(\ell)}]$ implies the positivity of a measure of the form
	\[ (\hat T_{(\Phi^{(\ell)})_K})^{\ell(r+1)} \wedge [\cY_K] \ne 0 \]
in $S \times (\bP^1)^{\ell(r+1)}$ for a projection of $\cX^{(\ell)}$ to $\cY_K\subset S \times (\bP^1)^K$ for some set of indices $K \subset \{1,\ldots,N\ell\}$ with $|K|=\ell(r+1)$. In particular, the projection satisfies $\dim\cY_K = \dim\cX^{(\ell)}=\ell(r+1)$. 

From Proposition \ref{analytic characterization}, the variety $\cY_K$ contains a Zariski-dense set of points that are preperiodic to repelling cycles for $(\Phi^{(\ell)})_K$ and intersect transversely with $\cY_K$.   It follows that we can choose one of these rigid repellers in $\cY_K$ that does not lie in the intersection of the distinct irreducible components of $Z(\cY_K)$, as defined in Proposition \ref{essential intersection}.   Choose such a rigid repeller $v_0\in \cY_K$, and let $s_0$ be its parameter. Then, along any algebraic curve $C$ through $s_0$ in $S$, we obtain a rigid repeller for at least one of the essential hypersurface shadows of $\cY_K$.  Indeed, if the projected preperiodic point was contained persistently in every essential shadow of $\cY_K$ along the curve, then the original point would remain in $Z(\cY_K)$. 

To obtain a contradiction, we now apply the preceding construction of $F$ using only the activity components corresponding to the essential hypersurface shadows of $\cY_K$. Each of these shadows of $\cY_K$ is also an essential hypersurface shadow of $\cX^{(\ell)}$, hence of $\cX$ by Proposition \ref{essential product}, so the activity rank bound gives a positive-dimensional fiber $F$ of the (restricted) activity map.  Each elliptic shadow under consideration contains a preperiodic point over $s_0$, and its quotient section has a torsion value there.  So all of these hypersurface shadows are stable along $F$, as described above.  Restricting to an algebraic curve $C\subset F$ through $s_0$ contradicts the rigid repeller just obtained. This completes the proof.
\end{proof}

\begin{corollary} \label{measure cor}
Under the hypotheses of Theorem \ref{true deg char}, with $\ell = \dim S$, the following are equivalent:
\begin{enumerate} 
\item		$\actdim(\Phi, \cX) = \ell$; 
\item		$\mu_{\Phi, \cX} \not=0$; and 
\item 	$(T_{\Phi, \cX})^{\wedge \ell} \not=0$.
\end{enumerate}
\end{corollary}

\begin{proof}
Theorem \ref{true deg char} provides the equivalence of (1) and (2).  That (3) $\implies$ (2) is a direct computation; see \cite[Proposition 1.4]{Gauthier:Taflin:Vigny}.  The implication (1) $\implies$ (3) appears in the proof of Theorem \ref{true deg char} leading up to \eqref{old bif measure}.
\end{proof}

\bigskip
\section{Proof of Theorem \ref{main1} and its uniformity consequences}
\label{proof main1}

In this section we prove Theorem \ref{main1} and explain how to derive the consequences Theorems \ref{degree bound}, \ref{preperiodbound}, \ref{theorem: common}, \ref{thm:mixed}, and Corollary \ref{iterateshare}.  We assume that
	\[\Phi = (f_1, \ldots, f_N): S \times(\bP^1)^N \to S\times (\bP^1)^N \]
is an algebraic family of product maps, with each $f_i$ of degree $>1$, and 
	\[ \cX\subset S\times(\bP^1)^N \]
is irreducible and flat over the irreducible base $S$, with geometrically integral generic fiber of dimension $0 < r < N$ dominating every coordinate.

\subsection{Simultaneous torsion for the exceptional shadows}\label{simultaneous torsion}
Recall the essential hypersurface decomposition of $\cX$ introduced in \S\ref{decomposition}.  
After shrinking $S$ and making one common generically finite base change, every weakly special exceptional essential shadow admits a group lift
	\[ \cB_J+\sigma_J\subset\cG_J^{E_J} \]
of codimension 1, and a quotient section
	\[ \bar\sigma_J:S\rightarrow \cQ_J:=\cG_J^{E_J}/\cB_J. \]
At a parameter $s$ in $S$, recall that we have
	\[ H_{J,s} \mbox{ is } \Phi_{E_J,s}\mbox{-special} \quad
	\Longleftrightarrow\quad \bar\sigma_J(s)\text{ is torsion in }\cQ_J. \]
We say that the exceptional shadows of $\cX$ are {\bf simultaneously torsion weakly special} if each such shadow is weakly special and 
\begin{equation} \label{simultaneous torsion locus}
	S^{\rm tor}_{\cX} := \{s\in S(\bC):\bar\sigma_J(s)\text{ is torsion for every such }J\}
\end{equation}
is Zariski dense in $S$.  The definition is independent of the common cover and of the chosen group lifts.  If there are no exceptional shadows, or if all of the weakly special shadows are special, we put $S^{\rm tor}_{\cX}=S$.

Because of the structure of these groups, we can also see that the existence of a single torsion point in the coset $(\cB_J+\sigma_J)_s$ for a parameter $s \in S$ implies that the weakly special hypersurface $H_{J,s}$ is in fact $\Phi_{E_J, s}$-special.  For each weakly special exceptional shadow, let 
	\[\mathcal{U}_J \subset S \times (\bP^1)^{E_J} \]
be the image of the group $\cG_J^{E_J}$ under the associated quotient; note that $\mathcal{U}_J$ is Zariski open but is not the total space in the case where $\cG_J$ is multiplicative.  We deduce that 
\begin{equation} \label{torsion versus preperiodic}
  H_{J,s} \mbox{ is } \Phi_{E_J,s}\mbox{-special} \quad
	\iff \quad  H_{J,s}\cap \mathcal{U}_J \cap \Preper(\Phi_{E_J}) \not= \emptyset.
\end{equation}

\subsection{Proof of Theorem \ref{main1} over $\Qbar$} \label{main1 over Qbar}
Put $\ell=\dim S$ and $m=\ell+1$.  We first work over $\Qbar$.

Assume (1) in the statement of Theorem \ref{main1}; that is, we assume that the special fibers are dense.  It follows immediately that every essential shadow is special at each of these parameters.  Applying Corollary \ref{dense special hypersurfaces}, we see that each essential shadow is weakly special.  Our assumption then ensures that the same dense set of
parameters works for every shadow, proving that the exceptional shadows are simultaneously torsion weakly special.

Conversely, suppose that (2) holds.  Then by definition, there is a dense set (which was denoted by $S^{\rm tor}_{\cX}$ above) where all essential hypersurface shadows are special.  Proposition \ref{essential intersection} then shows that $X_s$ is special for each $s \in S^{\rm tor}_{\cX}$.  This shows that (1) and (2) are equivalent.  

Assuming (1), we use the Zariski-density of preperiodic points in any special subvariety to deduce (3).  

It remains to show that (3) implies (2).  For every essential shadow, the dominant projection
	\[ \cX^{(m)}\rightarrow\cH_J^{(m)} \]
sends preperiodic points to preperiodic points.  Hence condition (3) implies that
	\[ \Preper(\Phi_{E_J}^{(m)})\cap\cH_J^{(m)} \]
is Zariski dense for every $J$.  Corollary \ref{dense special hypersurfaces} implies that each non-exceptional essential shadow is special and each exceptional shadow is weakly special.

It remains to show that the exceptional shadows are simultaneously torsion weakly special, which relies on \eqref{torsion versus preperiodic}.  Indeed, the Zariski density $\Preper(\Phi_{E_J}^{(m)})\cap\cH_J^{(m)}$ implies that these points must intersect the open $\mathcal{U}_J$, and therefore $\cH_{J,s}$ is special at each of these parameters.  From condition (3), we can find a Zariski-dense set of parameters $s$ at which this holds for all of the exceptional weakly special shadows, and therefore they are simultaneously torsion weakly special.  

This proves Theorem \ref{main1} over $\Qbar$.

\subsection{Unformity of degrees:  proof of Theorem \ref{degree bound}}\label{proof degree bound}
Building on the conclusion of Theorem \ref{main1} for families $(\Phi, \cX)$ defined over $\Qbar$, we are able to deduce Theorem \ref{degree bound}.  

This argument is now standard in this area; compare \cite{Yuan:bigness, Gao:Ge:Kuhne, Gao:Habegger:RMM, DMY:UniformDAO}.

Fix a projective embedding of $(\bP^1)^N$.  There is a finite union of Chow varieties, defined over $\Q$, parameterizing all irreducible subvarieties $V$ of degree $<E$, in every possible dimension, and a finite union of parameter spaces for tuples $\Phi=(f_1,\ldots,f_N)$ with $2\leq\deg f_i\leq D$.  Stratify their product so that the universal cycle
	\[ \cV\rightarrow S \]
is flat with geometrically integral fibers and all coordinate projections are dominant. These strata are defined over $\Qbar$.  

We work on one irreducible stratum $S$ and argue by induction on the fiber dimension $r$.  Put $\ell=\dim S$ and $m=\ell+1$.  If
	\[ \Preper(\Phi^{(m)})\cap\cV^{(m)} \]
is not Zariski dense, let $\cY$ be its Zariski closure.  After shrinking $S$, every fiber $Y_s$ is a proper subvariety of $V_s^m$, with uniformly bounded number of components and degrees.  Put
	\[ \Sigma_s=V_s\cap\Preper(\Phi_s). \]
Then $\Sigma_s^m\subset Y_s$.  The elementary product-set lemma \cite[Lemma 4.3]{Gao:Ge:Kuhne} gives a proper subvariety $Z_s\subset V_s$ which contains $\Sigma_s$ and whose number of components and degrees are bounded only in terms of $m$, $\dim V_s$, $\deg V_s$, and the geometry of $Y_s$. Apply the induction hypothesis to the irreducible components of $Z_s$, restricting to their non-constant coordinates; a component with a non-preperiodic constant coordinate contributes no preperiodic points. This gives a uniform bound for the degree and number of components of
	\[ \overline{V_s\cap\Preper(\Phi_s)}^{\rm Zar} \]
on a nonempty open subset of $S$.

Suppose instead that the preperiodic points are dense in $\cV^{(m)}$.  By Theorem \ref{main1} over $\Qbar$, every essential shadow is weakly special. After shrinking $S$, this remains true fiberwise.  If all essential shadows of $V_s$ are special, then $V_s$ is special by Proposition \ref{essential intersection}, and its preperiodic closure is $V_s$.   Otherwise, choose a non-special exceptional shadow.  By \eqref{torsion versus preperiodic}, $\Sigma_s$ is empty if this shadow is elliptic.  In the multiplicative case, $\Sigma_s$ is contained in a union of at most $2N$ coordinate sections $V_s\cap\{z_i=a\}$.  Each section is proper, since $V_s$ dominates every coordinate, and satisfies $\deg(V_s\cap\{z_i=a\})\leq\deg V_s$. Applying induction to their
irreducible components gives the desired bound.

Finally, repeat the argument on the finitely many irreducible components of the complement of the chosen open subset of $S$.  Noetherian induction on the dimension of the parameter space terminates after finitely many steps.  Taking the maximum over the finitely many Chow strata and possible fiber dimensions proves the theorem.  The argument also applies to complex parameters: each finite orbit-relation locus in the universal family is defined over $\Qbar$ and has Zariski-dense $\Qbar$-points, so the containment $\Sigma_s^m\subset Y_s$ holds also for $s\in S(\C)$.
\qed

\subsection{Proof of Theorem \ref{main1} over $\C$}
Given the conclusion in \S\ref{main1 over Qbar}, it remains only to pass from $\Qbar$ to $\C$.  Suppose the pair $(\Phi, \cX)$ is defined over $\C$.  As usual, we let $\ell = \dim_\C S$, and set $r = \dim_S \cX$, the relative dimension of $\cX$.  We write $m = \ell + 1$.

Choose a finitely generated subfield $K\subset\C$ over which the pair $(\Phi, \cX)$ is defined, and a quasiprojective $\Qbar$-variety $B$ with function field $K$.  After shrinking $B$, all of the maps, subvarieties, projections, and group quotients spread out over $B$.  Let $(\Phi_B, \cX_B)$ denote models over $B$, with parameter space $S_B$, all defined over $\Qbar$.  

Suppose that condition (1) of Theorem \ref{main1} holds.  The equations determining specialness are given by polynomial identities (over $\Q$) in the data defining $\Phi$ and $\cX$, and therefore the locus of special fibers also spreads out over $B$.  That is, the Zariski closure in $S_B$ of locus $\{s: X_s \mbox{ is } \Phi_s\mbox{-special}\}$ must be all of $S_B$.  

Similarly, the properties in (2) of Theorem \ref{main1} are preserved in this model $(\Phi_B, \cX_B)$; namely, the essential hypersurface decomposition remains the same, and each of the shadows is special or weakly special if and only if they were over $\C$. 

Applying the equivalence of (1) and (2) over $\Qbar$ to $(\Phi_B, \cX_B)$ over $S_B$, we see that (1) $\iff$ (2) over $\C$.

Furthermore, conditions (1) and (2) clearly imply condition (3) over $\C$ by the Zariski-density of preperiodic points in any special subvariety.  

Now we assume (3). We will show that (2) holds. Recall that we know that (3) implies (1) and (2) over $\Qbar$. We first explain that we may reduce the problem to case that $\cal{X}$ is an essential hypersurface. If $r=N$, then $\cX$ is the whole ambient family and there is nothing to prove.  Otherwise, for each essential shadow $\cH_J$, the dominant projection
 $\cX^{(m)}\to\cH_J^{(m)}$, sends $\Phi^{(m)}$-preperiodic points to $\Phi_{E_J}^{(m)}$-preperiodic points. Hence condition (3) holds for every essential shadow. 
It suffices to prove that this density forces $\cH_J$ to be weakly
special. Indeed, once this is known for every shadow, shrink $S$ and make a common generically finite base change so that the exceptional shadows admit the group opens $\mathcal U_J$ of \S\ref{simultaneous torsion}. The intersection 
	\[W=\bigcap_{J\text{ exceptional}}q_J^{-1}(\mathcal U_J)\subset\cX\]
is a dense open subset. By (3), the preperiodic points of $\cX$ are Zariski dense. Their intersection with $W$ therefore lies over a Zariski-dense set of parameters $s$. For each such $s$, choose a preperiodic point $x\in W\cap X_s$. Then $q_J(x)$ is a preperiodic point in $H_{J,s}\cap\mathcal U_{J,s}$ for every exceptional shadow.
By \eqref{torsion versus preperiodic}, all exceptional shadows
are special at $s$. Proposition
\ref{essential intersection} now gives (1), and hence (2).

We may therefore replace $(\Phi,\cX)$ by one of its essential shadows. Thus, we henceforth assume that  $\cX\subset S\times(\bP^1)^N$ is an essential hypersurface, of relative dimension $r=N-1$, and prove that it is weakly special.

We use the `classical' specialization method,  treating transcendental coefficients as additional parameters after spreading them out, as in Zilber \cite[Theorem 1]{Zilber2002}, Bombieri--Masser--Zannier \cite[\S3]{BMZ2008}, \cite[\S3]{CMZ2018} and others. We spread out one
transcendence variable at a time, arguing with induction on the transcendence degree, as in Barroero--Dill \cite[Proposition 4.1]{Barroero:Dill} and Gao--Habegger
\cite[\S10]{Gao:Habegger:RMM}. 

Spreading out one transcendence variable gives a parameter space of dimension $m=\ell+1$, matching the number of copies in (3). If the resulting pair has smaller-dimensional moduli image, we apply induction on that image. In the maximally varying case, Theorem \ref{hypersurfaces} gives geometric adelic bigness on the $m$-th fiber power when the hypersurface is not weakly special. This yields a function-field height gap, contradicting the density of preperiodic points. The height gap replaces the argument with degeneracy loci and Betti rank in \cite[\S10]{Gao:Habegger:RMM}.

Let $K\subset\C$ denote a subfield, finitely generated over $\Qbar$, over which $S$, $\Phi$, and $\cX$ are defined, and let $L$ be the algebraic closure of $K$ inside $\C$. We use the same notation for the corresponding models over $L$. Then
\[d:=\operatorname{trdeg}_{\Qbar}L=\operatorname{trdeg}_{\Qbar}K<\infty.\]
Condition (3) also holds over $L$.  We argue by induction on $d$. 

If $d=0$, then $L=\Qbar$, and the assertion follows from Corollary \ref{dense special hypersurfaces}.  If $\ell=0$, it follows from the usual dynamical Manin--Mumford results, for any $d$.  We may therefore assume that $d>0$ and $\ell>0$.  Choose an algebraically closed subfield $k\subset L$ with  $\operatorname{trdeg}_{\Qbar}k=d-1$ and a one-variable function field $F/k$ over which the pair is defined, with $\overline F=L$.
Let $B_0/k$ be a smooth curve with function field $F$ and spread out the
pair $(\Phi,  \cX)$ to
	\[\cX_{B_0}\rightarrow S_{B_0}\rightarrow B_0.\]
After shrinking the base, these morphisms are flat, the hypersurfaces remain essential, and $\dim_k S_{B_0}=m$.
All fiber powers below are taken over $S_{B_0}$. Note that 
\[\Preper(\Phi_{B_0}^{(m)})\cap\cX_{B_0}^{(m)}(k)\]
is Zariski dense in $\cX_{B_0}^{(m)}$. Indeed, fixing the periods
and preperiods defines a closed locus over $k$, whose $k$-points
are dense because $k$ is algebraically closed, so this follows from our assumption (3) over $L$. 

Suppose that $\cX$ is not weakly special, so that $\cX_{B_0}$ is not weakly special either. Set
\[p:=\dim\overline{\rho_{\Phi_{B_0},\cX_{B_0}}(S_{B_0})}.\]
For $p<m$, after making a generically finite base change and shrinking, the pair $(\Phi_{B_0},  \cX_{B_0})$ descends (up to coordinatewise conjugacy) to a pair $(\Phi_T,\cX_T)$ over a $k$-variety $T$ of dimension $p$.  As in the proof of Corollary \ref{dense special hypersurfaces}, the density of preperiodic points holds in $\cX_T^{(m)}$ and so also in $\cX_T^{(p+1)}$, since $p+1\leq m$. The induction hypothesis over $k$ makes $\cX_T$ weakly special, a contradiction. Thus $p=m$ and $(\Phi_{B_0},\cX_{B_0})$ is maximally varying.

To finish, we use the following height inequality, which requires the input of Theorem \ref{hypersurfaces}.  Write $\cX_F\to S_F$ for the model of $\cX\to S$ over $F$, so that
\[\cX_F^{(m)}=\cX_{B_0}^{(m)}\times_{B_0}\operatorname{Spec}F,~
(\cX_F^{(m)})_L=\cX^{(m)}.\]
All heights below are function-field heights over $F/k$, normalized by dividing degrees over a finite extension $E/F$ by $[E:F]$.

\begin{lemma}
\label{function field height gap}
Suppose that $(\Phi_{B_0},\cX_{B_0})$ is maximally varying over $k$
and that $\cX_{B_0}$ is not weakly special. There exist a nonempty
Zariski-open subset $U\subset\cX_F^{(m)}$ and $\epsilon>0$ such that
\begin{equation}\label{specialization height gap}
\sum_{a=1}^m\sum_{i=1}^n
\hat h_{(f_i)_s}(x_{a,i})\geq\epsilon
\end{equation}
for every $P=(s;x_1,\ldots,x_m)\in U(L)$, where
$x_a=(x_{a,1},\ldots,x_{a,n})$.
\end{lemma}

\begin{proof}
After replacing $S_{B_0}$ by an affine open subset, put
\[D=\dim_k\cX_{B_0}^{(m)}=m+m(n-1)=mn,
~\Mbar:=\Lbar_{\Phi_{B_0}^{(m)}}|_{\cX_{B_0}^{(m)}}
\in\Pichat(\cX_{B_0}^{(m)}/k)\]
and recall that $\Mbar$ is a strongly nef geometric adelic line bundle. 
Let $\Mbar_{\mathrm{loc}}$ be the corresponding strongly nef local
complex canonical bundle on $\cX_{B_0,\C}^{(m)}$, recalling that $k$ is a subfield of $\C$, and let $\widetilde M_{\mathrm{loc}}$ be its geometric adelic class obtained
by forgetting the metrics, as in \S\ref{section: prelims}.
The Hilbert--Samuel formula \cite[Theorem 5.2.2(1)]{Yuan:Zhang:quasiprojective} and Guo's integration formula \cite[Theorem 1.2]{Guo}, recalled in
\eqref{geometric volume formula}, give
\begin{equation}\label{specialization volume mass}
\begin{aligned}
\volhat(\Mbar)
	&=\Mbar^D=(\widetilde M_{\mathrm{loc}})^D =\volhat(\widetilde M_{\mathrm{loc}})\\
	&=\int_{\cX_{B_0,\C}^{(m)}(\C)} \hat T_{\Phi_{B_0,\C}^{(m)}}^{\wedge D}>0.
\end{aligned}
\end{equation}
To explain these equalities in more detail, note that the two volumes are geometric adelic volumes over $k$ and $\C$, respectively.  The limiting construction in \cite{Yuan:Zhang:quasiprojective} represents $\overline M$ by nef line bundles $M_j$ on projective models $Y_j$ over $k$.  Over $\C$, the scalar extension of these line bundles $M_{j,\mathbb C}$ represent $\widetilde M_{\mathrm{loc}}$. Intersection degrees are preserved by scalar extension, so the limiting construction of the adelic intersection number in \cite[Prop 4.1.1]{Yuan:Zhang:quasiprojective} gives
	\[\overline M^D =\lim_j(M_j^D)|_{Y_j/k} =\lim_j(M_{j,\mathbb C}^D)|_{Y_{j,\mathbb C}/\mathbb C} =\widetilde M_{\mathrm{loc}}^D.\]
The last equality in \eqref{specialization volume mass} uses $c_1(\Mbar_{\mathrm{loc}}) =\hat T_{\Phi_{B_0,\C}^{(m)}}|_{\cX_{B_0,\C}^{(m)}}$.  The strict positivity follows from Theorem \ref{hypersurfaces}, applied to $(\Phi_{B_0,\C},\cX_{B_0,\C})$ over the $m$-dimensional parameter space $S_{B_0,\C}$; this pair is maximally varying and $\cX_{B_0,\C}$ is not weakly special.

We also write $\Mbar$ for its restriction to $\cX_F^{(m)}$, viewed
as an adelic line bundle over $k$. This restriction remains big by
\cite[Theorem 5.2.7]{Yuan:Zhang:quasiprojective}, since the natural morphism
$\cX_F^{(m)}\to\cX_{B_0}^{(m)}$ induces an isomorphism of function fields.
Let $\overline B_0$ be the smooth projective completion of $B_0$.
Choose $v_0\in\overline B_0(k)$, and let
$\Nbar=\overline{\mathcal O}_{\varphi}\in\Pichat(\operatorname{Spec}F/k)$
be the trivial line bundle represented by the metrized divisor
$(0,\{g_v\}_v),~
 g_{v_0}\equiv1,~ g_v\equiv0\quad(v\ne v_0).$
Here $g_v=-\log\|1\|_v$, and we normalize the absolute values by
$|a|_v=\exp(-\operatorname{ord}_v(a))$.
Since $k$ is algebraically closed, every place has degree $1$, so
$h_{\Nbar}\equiv1$. Applying
\cite[Theorem 5.3.7(1)]{Yuan:Zhang:quasiprojective} to
$\cX_F^{(m)}\to\operatorname{Spec}F$ therefore gives a nonempty
Zariski-open subset $U\subset\cX_F^{(m)}$ and $\epsilon>0$ such that
\[h_{\Mbar}(P)\geq\epsilon \text{ for every }P\in U(L),\]
so the lemma follows from  \eqref{product height sum}.
\end{proof}

By (3), the preperiodic points are Zariski dense in
$\cX^{(m)}=(\cX_F^{(m)})_L$, and hence one lies in $U(L)$, contradicting
\eqref{specialization height gap}. Thus $\cX$ is weakly special,
completing the induction and the proof of $(3)\Rightarrow(2)$.
\qed

\subsection{Further uniformity consequences}
\label{uniformity}

We conclude this section with particular consequences of Theorem \ref{main1} over $\Qbar$ and Theorem \ref{degree bound}.

\begin{proof}[Proof of Theorem \ref{preperiodbound}]
For a given map $F = (f_1, \ldots, f_N)$ on $(\bP^1)^N$ and vector of positive integers $\mathbf a=(a_1,\ldots,a_N)$, we set 
	\[F_{\mathbf a} :=(f_1^{a_1},\ldots,f_N^{a_N}). \]

Fixing integers $D$ and $E$, let $\mathcal P$ be the finite union of parameter spaces for tuples of non-exceptional maps of degrees between $2$ and $D$ and irreducible cycles in $(\bP^1)^N$ of degree at most $E$, restricted to the locus where the cycle dominates every coordinate and the maps are non-exceptional.  For given $\mathbf a,m,n$, the condition
	\[ F_{\mathbf a}^n(V)=F_{\mathbf a}^m(V) \]
is Zariski closed in $\mathcal P$.  

Let $\mathcal R \subset \mathcal P$ be the union of these closed loci over all $\mathbf a,m,n$.  Its points are precisely the relations under consideration, by the classification of special subvarieties for non-exceptional product maps.  Let $S$ be a positive-dimensional irreducible component of the Zariski closure of $\mathcal R$.  The special fibers of the universal family over $S$ are Zariski dense, by definition.  Applying Theorem \ref{main1}, and recalling that all maps are non-exceptional in this parameter space, we conclude that the geometric generic fiber is special.  It therefore satisfies one fixed relation
	\[ F_{\mathbf a}^n(V)=F_{\mathbf a}^m(V) \]
over the function field of $S$.  Since the corresponding relation locus is closed, it contains all of $S$.  Thus, $\mathcal R$ is closed.  It has finitely many irreducible components, and one fixed triple $(\mathbf a,m,n)$ works on each component.  Taking the maximum of the finitely many exponents proves the theorem.
\end{proof}

\begin{proof}[Proof of Corollary \ref{iterateshare}]
Write $d_f = \deg f$ and $d_g = \deg g$.  If $f$ and $g$ share an iterate, the diagonal $\Delta \subset \bP^1\times\bP^1$ is a relation for the pair, as observed in Proposition \ref{diagonal relations}.  In fact, the diagonal is necessarily periodic for some iterates $f^a$ and $g^b$ for which $d_f^a = d_g^b$.  Theorem \ref{preperiodbound} bounds the coordinatewise exponents and the orbit data of this relation. 
\end{proof}

\begin{proof}[Proof of Theorem \ref{theorem: common}]
Apply Theorem \ref{degree bound} with $N=2$, and let $V=\Delta$ be the diagonal in $\bP^1\times\bP^1$.  For any pair of maps $f, g: \bP^1\to \bP^1$ defined over $\C$, the intersection
	\[ \Delta\cap\Preper(f,g) \]
in $\bP^1\times\bP^1$ is naturally identified with $\Preper(f)\cap\Preper(g)$ in $\bP^1$.  Its Zariski closure in $\Delta$ is either zero-dimensional or all of $\Delta$.  In the first case its cardinality equals its degree and is uniformly bounded by Theorem \ref{degree bound}.  In the second case, the diagonal contains infinitely many $(f,g)$-preperiodic points, and we know from \cite[Theorem 1.2]{BD:preperiodic} or \cite[Theorem 1.3]{Yuan:Zhang:II} that
	 \[ \Preper(f)=\Preper(g). \]
This proves part (1).

Assume now that $f$ and $g$ are non-exceptional and have the same preperiodic set.  Then the diagonal is preperiodic for $(f^a,g^b)$ for some $a,b$ with $\deg(f)^a=\deg(g)^b$, combining the results of \cite[Theorem 1.4]{Yuan:Zhang:II} and \cite[Theorem A]{Levin:Przytycki} and the observation of Proposition \ref{diagonal relations}(2); see also \cite{Ghioca:Nguyen:Ye:DMM1}.  Theorem \ref{preperiodbound} bounds the orbit data of the diagonal, and Proposition \ref{diagonal relations} provides bounds on the integers $n,m$ such that
	\[ f^n\circ g^m=f^{2n}, \qquad  g^m\circ f^n=g^{2m}. \]
Conversely, these identities imply equality of the preperiodic sets by the preperiodicity of $\Delta$, also from Proposition \ref{diagonal relations}.  For bounded $n,m$, the displayed identities are algebraic equations in the coefficients of $f$ and $g$. The desired locus is therefore a finite union of closed algebraic subvarieties.  This proves part (2).
\end{proof}

\begin{proof}[Proof of Theorem \ref{thm:mixed}]
Let $L_E:\bP^1\to\bP^1$ be the Latt\`es map induced by multiplication by $2$ on $E$, and let $p(z)=z^2$.  Consider the product
	\[ \Psi=(L_E,f,p,\ldots,p) \]
on $(\bP^1)^{q+2}$ and the projective closure $H$ in $(\bP^1)^{q+2}$ of the affine hyperplane
	\[ a_0x+a_1y+\sum_{j=1}^q b_jz_j=1. \]
The degree of $H$ (upon embedding $(\bP^1)^{q+2}$ into some $\bP^N$) is bounded only in terms of $q$, for any choice of coefficients or maps, while the degrees of the coordinate maps are bounded in terms of $d$. Moreover, $H$ is irreducible and projects dominantly to each coordinate. Every solution in the statement is a point of $H\cap\Preper(\Psi)$.  Theorem \ref{degree bound} therefore gives a bound, depending only on $d$ and $q$, for the degree and the number of components of the Zariski closure $\overline{H\cap\Preper(\Psi)}$.

It remains to observe that every positive-dimensional component of this closure intersecting $U:=\bA^1\times\bA^1\times\mathbb G_m^q$ is contained in a degenerate subsum locus involving only multiplicative terms. Let $V$ be such a component.  By the classification of special subvarieties from \cite{Ghioca:Nguyen:Ye:DMM2, MSW2023}, no positive-dimensional relation can mix the ordinary $f$-coordinate with an exceptional coordinate. Thus, if $y$ were nonconstant on $V$, it would vary independently of the other coordinates, contradicting the equation defining $H$ because $a_1\neq0$. Hence $y=y_0$ is constant on $V$.

Let $W$ be the projection of $V$ onto the exceptional coordinates $(x,z_1,\ldots,z_q)$. Since $y=y_0$ on $V$, we will henceforth identify $V$ with $W$. By the description of special subvarieties,
$W\cap(\mathbb P^1\times\mathbb G_m^q)$
is the image under a quotient map of a torsion translate of a connected algebraic subgroup of $E\times\mathbb G_m^q$. Every such subgroup is a product $A\times T$, where $A$ is either $E$ or $\{O\}\subset E$ and $T\subset\mathbb G_m^q$ is a torsion coset.  If $A=E$, then, all points$(x,y_0,z_1,\ldots,z_q)\in\mathbb{A}^1\times\{y_0\}\times\tau T$, would satisfy the equation defining $H$, contradicting that $a_0\neq0$. Hence $A=\{O\}$, so $x=x_0$ is constant on $V$. Since $V$ meets $U$, we have $x_0\in\mathbb A^1$, and therefore
	\[V\cap U=\{(x_0,y_0)\}\times\tau T,\]
and $\dim T=\dim V>0$.  Since $\dim T>0$, there exist integers $m_1,\ldots,m_q$, not all
zero, such that
	\[(s^{m_1},\ldots,s^{m_q})\in T~\text{for every }s\in\mathbb C^\star.\]
Fix a nonzero integer $m$ occurring among the $m_j$, and set $I=\{j:m_j=m\}$. For every $(z_1,\ldots,z_q)\in\tau T$, we also have $(z_1s^{m_1},\ldots,z_qs^{m_q})\in\tau T.$
The equation defining $H$ therefore gives
	\[\sum_{j=1}^q b_jz_js^{m_j}=1-a_0x_0-a_1y_0 ~\text{for every } s\in\mathbb C^\star.\]
Comparing the coefficients of $s^m$, and since $m\neq0$, we obtain
	\[\sum_{j\in I}b_jz_j=0.\]
Thus the same nonempty subsum vanishes throughout $V\cap U$.

Since $V$ was an arbitrary positive-dimensional irreducible component
of $Z$ meeting $U$, every nondegenerate solution is a zero-dimensional
irreducible component of $Z$. The number of these components is
uniformly bounded in terms of $d$ and $q$ by Theorem
\ref{degree bound}, as established above. This proves the theorem.
\end{proof}

\bigskip
\section{Arithmetic bigness in general}
\label{arithmetic bigness}

We now prove Theorem \ref{main3}.  Throughout the section, all data are
defined over $\Qbar$, and
	\[ \ell=\dim S>0, \qquad r=\dim X_s. \]
After restriction to a nonempty open subset of $S$, and a generically finite base change, we assume that all essential shadow families and all exceptional quotient data are defined over $S$.  These operations do not change the vanishing, bigness, or density conditions under consideration.

\subsection{The Deligne activity bundle}
Define
\begin{equation} \label{Deligne activity bundle}
\overline{\mathcal M}_{\Phi,\cX} =
	\left\langle \Lbar_\Phi|_{\cX},\ldots,\Lbar_\Phi|_{\cX} \right\rangle_{\cX/S}
		\in\Pichat(S/\bZ)_{\nef},
\end{equation}
where the Deligne pairing has $r+1$ entries.  Its height at $s\in S(\Qbar)$ is
\begin{equation} \label{Deligne activity height}
	h_{\overline{\mathcal M}_{\Phi,\cX}}(s) 
	=(r+1)\deg_{L_{\Phi_s}}(X_s)\,\hat h_{\overline{L}_s}(X_s).
\end{equation}
The degree is positive and constant over $S$.

The following proposition follows from \cite{Yuan:bigness}.

\begin{proposition} \label{fiber power criterion}
Set $m=\ell+1$.  Then
	\[ \volhat\left( \Lbar_{\Phi^{(m)}}|_{\cX^{(m)}} \right)=0 ~
 	\Leftrightarrow~\volhat\left(S,\overline{\mathcal M}_{\Phi,\cX}\right)=0.\]
Moreover, for one, and hence every, ample Weil height $h_S$ on a projective compactification of $S$, these conditions are equivalent to the existence of a generic sequence $s_n\in S(\Qbar)$ for which
	\[\frac{h_{\overline{\mathcal M}_{\Phi,\cX}}(s_n)} {\max\{1,h_S(s_n)\}} \rightarrow0.\]
\end{proposition}

\begin{proof}
The first assertion is the fiber-power formulation of arithmetic bigness for the Deligne pairing \cite[Theorem 4.1]{Yuan:bigness}.  

Note that if $\overline{\mathcal M}_{\Phi,\cX}$ is arithmetically big, and if we fix an ample height $h_S$, then \cite[Theorem 5.3.7(1)]{Yuan:Zhang:quasiprojective} gives an $\epsilon > 0$ so that $h_{\overline{\mathcal M}_{\Phi,\cX}}\geq\epsilon\max\{1,h_S\}$ on a nonempty open subset.  So a generic sequence like the $\{s_n\}$ of the second assertion cannot exist.  Conversely, absence of such a sequence gives the same lower bound on some nonempty open subset.  Then \cite[Theorem 5.3.8]{Yuan:Zhang:quasiprojective} implies that $\widetilde{\mathcal M}_{\Phi,\cX}$ is big, and consequently, the positive essential minimum implies arithmetic bigness by the fundamental inequality \cite[Theorem 5.3.3]{Yuan:Zhang:quasiprojective}.
\end{proof}

Recall the notion of activity dimension from \S\ref{active dim}.  

\begin{proposition} \label{geometric Deligne nondegeneracy}
The geometric part $\widetilde{\mathcal M}_{\Phi,\cX}$ is big if and only if $\cX$ is potentially non-degenerate.  Equivalently,
	\[ \volhat\left(\widetilde{\mathcal M}_{\Phi,\cX}\right)>0 \quad
 	\Longleftrightarrow\quad \actdim(\Phi,\cX)=\ell. \]
\end{proposition}

\begin{proof}
The curvature (1,1)-current of the Deligne activity bundle $\overline{\mathcal M}_{\Phi,\cX}$ at an archimedean place is $T_{\Phi,\cX}$.  The proposition follows immediately from Corollary \ref{measure cor} and the volume formula \eqref{geometric volume formula}.
\end{proof}

\subsection{Decomposition into essential shadow bundles}  Recall the essential hypersurface decomposition from \S\ref{decomposition}.  For $J\in\mathcal J(X)$, recall that
	\[ Y_J=H_J\times(\bP^1)^{J\setminus E_J}.	\]
Set
\begin{equation} \label{e_J}
	 e_J:=|E_J| \quad\mbox{and} \quad \delta_J=\deg(X\rightarrow Y_J), 
\end{equation}
and define
\begin{equation} \label{shadow Deligne bundle}
\overline{\mathcal M}_J := \left\langle \Lbar_{\Phi_{E_J}}|_{\cH_J},
 			\ldots, \Lbar_{\Phi_{E_J}}|_{\cH_J} \right\rangle_{\cH_J/S}\in \Pichat(S/\bZ)_{\Q,\mathrm{int}},
\end{equation}
where the Deligne pairing has $e_J$ entries.

This next proposition is an arithmetic analog of Proposition \ref{essential stability}.

\begin{proposition}
\label{Deligne essential decomposition}
In $\Pichat(S/\bZ)_{\Q,\mathrm{int}}$ we have
\begin{equation} \label{Deligne essential decomposition equation}
\overline{\mathcal M}_{\Phi,\cX} = \sum_{J\in\mathcal J(X)} c_J\overline{\mathcal M}_J,
		\qquad  c_J=\delta_J\frac{(r+1)!}{e_J!}>0.
\end{equation}
Every $\overline{\mathcal M}_J$ is nef, and
\begin{equation} \label{shadow fiber height}
h_{\overline{\mathcal M}_J}(s) \; = \;
	e_J\; \deg_{L_{\Phi_{E_J},s}}(H_{J,s}) \; \hat h_{\Phi_{E_J},s}(H_{J,s})\geq0.
\end{equation}
\end{proposition}

\begin{proof}
Recall that we write 
$\Lbar_\Phi=\sum_{i=1}^N\Lbar_i$.  Multilinearity of the Deligne pairing as in \cite[Theorem 4.1.3]{Yuan:Zhang:quasiprojective}, gives
	\[\overline{\mathcal M}_{\Phi,\cX} = \sum_{(i_0,\ldots,i_r)\in\{1,\ldots,N\}^{r+1}}
			\left\langle \Lbar_{i_0}|_{\cX},\ldots,\Lbar_{i_r}|_{\cX} \right\rangle_{\cX/S}. \]
Let $I$ be the set of distinct indices in one summand.  If the image of
$X$ in $(\bP^1)^I$ has dimension $<r$, the summand vanishes by the
relative projection formula \cite[Lemma 4.6.1(3)]{Yuan:Zhang:quasiprojective}: at least one positive-dimensional fiber is met only by line bundles pulled back from the image.

Suppose the image has dimension $r$.  If $|I|=r$, one coordinate occurs
twice.  Push the pairing to $(\bP^1)^I$ and apply the finite endomorphism $\Phi_I$.  By linearity, the functional equations
	\[ \Phi_I^*\Lbar_i=d_i\Lbar_i \]
scale the pairing once by $d_i$ for every coordinate in $I$, and one additional time by the repeated degree.  However, from \cite[Lemma 4.6.1(4)]{Yuan:Zhang:quasiprojective}, pulling back by the finite map scales it only by the product of the degrees from distinct coordinates.  Since the repeated degree is $>1$, the pairing is zero.  Thus only tuples with $r+1$ distinct entries contribute. 

For a fixed $J\in\mathcal J(X)$, there are $(r+1)!$ orders of its
coordinates.  Pushing forward from $X$ to $Y_J$ contributes the factor
$\delta_J$.  The product decomposition of $Y_J$ and repeated use of the
projection formula integrate out the coordinates in $J\setminus E_J$,
leaving
	\[ \left\langle (\Lbar_i|_{\cH_J})_{i\in E_J} \right\rangle_{\cH_J/S}. \]
Expanding \eqref{shadow Deligne bundle}, the same repeated-coordinate
argument shows that only the $e_J!$ permutations in which every
essential coordinate occurs once survive.  This gives
\eqref{Deligne essential decomposition equation}.

Nefness follows from positivity of Deligne pairings \cite[Theorem 4.1.3]{Yuan:Zhang:quasiprojective}.  The fiber-height identity is base change for the pairing together with the definition of the canonical height of a subvariety; see \cite[Theorem 4.1.3 and \S5.3.1]{Yuan:Zhang:quasiprojective}.
\end{proof}

For a weakly special shadow with quotient section $\bar\sigma_J:S\to\cQ_J$, we consider the canonical heights on the group scheme $\cQ_J$, and we let $\overline{\mathcal N}_J$ on $S$ denote its pullback by $\bar\sigma_J$ to $S$, so that
\begin{equation} \label{quotient height bundle}
	h_{\overline{\mathcal N}_J}(s) = \hat h_{\cQ_{J,s}}(\bar\sigma_J(s)).
\end{equation}
Here the right-hand side is the N\'eron--Tate height in the elliptic case and the usual canonical height on $\mathbb G_m$ in the multiplicative case.  

The next lemma is an arithmetic analog of Proposition \ref{weakly special bif}.

\begin{lemma} \label{weak shadow height}
If $\cH_J$ is weakly special, then there is a constant $b_J>0$ such that
	\[ \overline{\mathcal M}_J=b_J \, \overline{\mathcal N}_J \]
up to addition of a torsion adelic line bundle.  In particular,
	\[ h_{\overline{\mathcal M}_J}(s)=0 ~\Leftrightarrow~\bar\sigma_J(s)\text{ is torsion}. \]
If $\cH_J$ is special, then $\overline{\mathcal M}_J=0$.
\end{lemma}

\begin{proof}
Let $\cH_J$ be a weakly special shadow, with group lift
	\[\cB_J+\sigma_J\subset\cG_J^{E_J}\]
as in \S\ref{simultaneous torsion}. Let $q:\cG_J^{E_J}\to\cQ_J$ be the quotient by $\cB_J$.
Let $\overline L$ be the pullback of $\Lbar_{\Phi_{E_J}}$ under the coordinatewise quotient map
$\cG_J^{E_J}\to S\times(\bP^1)^{E_J}$.  Set $e= |E_J|$, and
	\[\overline A=\langle\overline L,\ldots,\overline L\rangle_{\cG_J^{E_J}/\cQ_J},\]
with $e$ entries.

For $n>1$, write $\lambda=n^2$ in the elliptic case and $\lambda=n$ in the multiplicative case. Since $[n]^*\overline L=\lambda\overline L$ and $\deg([n]|_{{\cB_{J,s}}})=\lambda^{e-1}$ for $s\in S$, base change for the Deligne pairing and the projection formula
\cite[Theorem 4.1.3 and Lemma 4.6.1(4)]{Yuan:Zhang:quasiprojective} give  
	\[\lambda^e\overline A=\lambda^{e-1}[n]^*\overline A.\]
Thus $[n]^*\overline A=\lambda\overline A$. The underlying line bundle $A$ has positive degree on the fibers of $\cQ_J\to S$, and in the elliptic case it is also symmetric.  Since the quotient is one-dimensional, uniqueness of the canonical adelic line bundles therefore gives
	\[\bar\sigma_J^*\overline A=c \, \overline{\mathcal N}_J  ~\text{for some }c>0;\]
see \cite[Theorems 6.1.2--6.1.3]{Yuan:Zhang:quasiprojective}.  If the projection $\cB_J+\sigma_J\to\cH_J$ has degree $\delta$, base change and the projection formula now yield
	\[\delta\,\overline{\mathcal M}_J=\bar\sigma_J^*\overline A=c\,\overline{\mathcal N}_J.\]
This proves the identity with $b_J=c/\delta$. 

The second statement follows from the first since the canonical height on an elliptic curve or the multiplicative group vanish precisely when the point is torsion.
\end{proof}

\subsection{The relative Bogomolov input}
The following proposition is where the hypothesis in Theorem \ref{main3}, excluding the mixed group types, is used.

\begin{proposition}
\label{homogeneous relative Bogomolov}
Let $B$ be normal, irreducible, and quasiprojective over $\Qbar$.  For integer $q\geq 2$, and for each $a\in\{1,\ldots,q\}$, let $\cQ_a\to B$ be a one-dimensional group scheme.  Assume that the $\cQ_a$ are either all either multiplicative or all elliptic schemes, and let
	$\xi_a:B\rightarrow\cQ_a$
be sections.  Let $\overline{\mathcal N}_a$ be the pullback of the canonical-height bundle on $\cQ_a$ to $B$, and set
	$\overline{\mathcal N}=\sum_{a=1}^q\overline{\mathcal N}_a.$
Then the essential minimum is $e(B,\overline{\mathcal N})=0$ if and only if the set 
	\[ \{b\in B(\Qbar):\xi_a(b)\text{ is torsion for every }a\}\] 
is Zariski dense in $B$.
\end{proposition}

\begin{proof}
Set $T = \{b\in B(\Qbar):\xi_a(b)\text{ is torsion for every }a\}$.  It is clear that the density of $T$ in $B$ implies that the essential minimum is 0.  We now explain the converse.  For elliptic schemes, this follows from K\"uhne's relative Bogomolov theorem \cite{Kuhne:RBC}.  More precisely, after shrinking $B$, replace $\xi:=(\xi_a)_a$ by a suitable multiple and the ambient
fiber product by the connected abelian subscheme generated by its generic value, which we denote by $\mathcal A\to B$.  Thus $\overline{\mathbb{Z}\, \xi(B)}^{\mathrm{Zar}}=\mathcal A$. 
If the Betti map restricted to $\xi(B)$ has rank $<2\dim(\mathcal A/B)$, then, following
\cite[\S9]{Gao:Habegger:RMM}, in view of \cite[Theorem 1.1(i)]{GaoBettiRank}, after a finite base change there is an abelian quotient of $\mathcal A$ whose moduli image $\mathcal A_0\to B_0$ contains the image $X_0$ of $\xi(B)$ with
	\[\dim X_0<\dim(\mathcal A_0/B_0), 
	~\overline{\mathbb Z X_0}^{\mathrm{Zar}}=\mathcal A_0.\]
Small points remain dense in $X_0$ by comparison of fiberwise canonical heights, contradicting \cite[Theorem 1]{Kuhne:RBC} since $\mathcal A_0$ remains isogenous to a product of elliptic
schemes. Hence the Betti map has full rank and the sections are torsion simultaneously on a  Zariski-dense set.  For multiplicative groups, this statement follows from Zhang's multiplicative Bogomolov theorem
\cite[Theorem 6.2]{Zhang:positive}.  
\end{proof}

\subsection{Proof of Theorem \ref{main3}}
By Proposition \ref{fiber power criterion}, conditions (1) and (2) of Theorem \ref{main3} are each equivalent to
\begin{equation} \label{Deligne volume zero}
	\volhat\left(\overline{\mathcal M}_{\Phi,\cX}\right)=0.
\end{equation}
We compare this condition with part (3).

Suppose first that $\cX$ is truly degenerate, so its bifurcation measure $\mu_{\Phi, \cX}$ vanishes on $S$.  Expansion of the power of the dynamical Green current on $\cX^{(\ell)}$ and fiber integration give
	\[ \mu_{\Phi,\cX}\geq \frac{(\ell(r+1))!}{((r+1)!)^\ell}\, T_{\Phi,\cX}^{\wedge\ell}; \]
see \cite[Proposition 1.4]{Gauthier:Taflin:Vigny}.  The curvature (1,1)-current of the Deligne activity bundle $\overline{\mathcal M}_{\Phi,\cX}$ at an archimedean place is $T_{\Phi,\cX}$. Hence true degeneracy implies that the geometric part of $\overline{\mathcal M}_{\Phi,\cX}$ is not big.  Arithmetic bigness would imply geometric bigness \cite[Theorems 5.3.7(1) and 5.3.8]{Yuan:Zhang:quasiprojective}, so \eqref{Deligne volume zero} holds.

Suppose next that all non-special essential hypersurface shadows are simultaneously torsion weakly special.  At every parameter in the common dense torsion locus, each essential shadow is special.  Equations \eqref{Deligne essential decomposition equation} and \eqref{shadow fiber height} give
	\[ h_{\overline{\mathcal M}_{\Phi,\cX}}(s)=0 \]
on this dense set.  If the adelic bundle $\overline{\mathcal M}_{\Phi,\cX}$ were big, the height inequality would bound its height below by a positive number away from a proper closed subset \cite[Theorem 5.3.7]{Yuan:Zhang:quasiprojective}.  This is impossible, proving \eqref{Deligne volume zero}.  Thus (3) implies (1).

For the converse, assume \eqref{Deligne volume zero} and suppose that $\cX$ is not truly degenerate, meaning that it is potentially non-degenerate.  Then $\mu_{\Phi, \cX} \not=0$ on $S(\C)$, and so by Corollary \ref{measure cor}, we also have $T_{\Phi,\cX}^{\wedge\ell} \not= 0$.  In other words, the geometric part $\widetilde{\mathcal M}_{\Phi,\cX}$ is big.  By \cite[Theorem 5.2.2(1)]{Yuan:Zhang:quasiprojective} and the definition of the height, \eqref{Deligne volume zero} says that the normalized $\overline{\mathcal M}_{\Phi,\cX}$-height of $S$ is zero.  The fundamental inequality \cite[Theorem 5.3.3]{Yuan:Zhang:quasiprojective} gives
	\[ e\left(S,\overline{\mathcal M}_{\Phi,\cX}\right)=0. \]
Choose a generic sequence $s_n$ with $h_{\overline{\mathcal M}_{\Phi,\cX}}(s_n)\rightarrow0$.
Every summand in \eqref{Deligne essential decomposition equation} has nonnegative height.  Hence
\begin{equation} \label{each shadow small}
	h_{\overline{\mathcal M}_J}(s_n)\rightarrow0
\end{equation}
for every essential hypersurface shadow.

We claim that every non-special shadow is weakly special.  Suppose that
$\cH_J$ is not weakly special.  Let $T_J$ be the normalization of the image of $S$ by its pair-moduli map $\nu_J$, defined in \eqref{pair moduli}.  After a generically finite base change and coordinatewise conjugacy, the pair over $S$ is pulled back from a maximally varying pair
	$ (\Phi_{J,T_J},\cH_{J,T_J}) $
over $T_J$, and
	\[ \overline{\mathcal M}_J = \nu_J^*\overline{\mathcal M}_{J,T_J}.  \]
If $\dim T_J>0$, Theorem \ref{arithmetic bigness: essential case} shows that $\overline{\mathcal M}_{J,T_J}$ is arithmetically big.  As before, the fundamental inequality gives positive
essential minimum.  This contradicts \eqref{each shadow small}.  If $\dim T_J=0$, the height in \eqref{shadow fiber height} is a positive constant by the dynamical Bogomolov theorem for the fixed map; see \cite[Theorem 1.1]{Ghioca:Nguyen:Ye:DMM1},  \cite[Theorem 2.2(2)]{Ghioca:Nguyen:Ye:DMM2}, \cite[Theorem 1.1]{MSW2023}, \cite[Theorem 6.2]{Zhang:positive}
and \cite[Corollary 3]{Zhang:ICM}. We have again reached a contradiction, so the claim follows.

For the special shadows, the Deligne bundles vanish, and we may ignore them.  By Lemma \ref{weak shadow height}, the remaining part of \eqref{Deligne essential decomposition equation} is a positive linear combination of the quotient-height bundles $\overline{\mathcal N}_J$ defined in \eqref{quotient height bundle}.  The sequence $s_n$ is generic and small for their sum.  All quotient groups have the same type by hypothesis. Proposition \ref{homogeneous relative Bogomolov} therefore implies that
	\[ \{s\in S(\Qbar):\bar\sigma_J(s)\text{ is torsion for every non-special
shadow }J\}  \]
is Zariski dense.  Thus the non-special essential hypersurface shadows are simultaneously torsion weakly special.  This proves (1) implies (3) and completes the proof of Theorem \ref{main3}.

\bigskip \bigskip

\def\cprime{$'$}

\end{document}